\documentclass[10pt]{article}
\usepackage[left=0.8in,top=0.8in,right=0.8in,bottom=0.8in]{geometry}

\usepackage{graphicx}
\usepackage{amsmath}
\usepackage{amssymb}
\usepackage{amsthm}
\usepackage{latexsym}
\usepackage{bm}
\usepackage{array,supertabular}
\usepackage{color}
\usepackage{framed}
\usepackage{setspace}
\usepackage{fancyhdr}
\usepackage{stmaryrd}
\usepackage{mathrsfs}
\usepackage{mdframed}
\usepackage{url}
\usepackage{multirow}
\usepackage{algorithm}
\usepackage{algpseudocode}
\usepackage{pifont}
\usepackage{longtable}
\usepackage{booktabs}
\usepackage{natbib}
\usepackage{subcaption}

\setcitestyle{authoryear,open={(},close={)}}

\makeatletter
\newcommand*{\rom}[1]{\expandafter\@slowromancap\romannumeral #1@}
\makeatother

\newcolumntype{P}[1]{>{\centering\arraybackslash}p{#1}}

\SetSymbolFont{stmry}{bold}{U}{stmry}{m}{n}

\newtheorem{proposition}{Proposition}
\newtheorem{remark}{Remark}
\newtheorem{assumption}{Assumption}

\numberwithin{equation}{section}

\newenvironment{myenv}[1]
  {\mdfsetup{
    frametitle={\colorbox{white}{\space#1\space}},
    innertopmargin=10pt,
    frametitleaboveskip=-\ht\strutbox,
    frametitlealignment=\center
    }
  \begin{mdframed}
  }
  {\end{mdframed}}

\begin{document}
\title{Modeling finite viscoelasticity based on the Green-Naghdi kinematic assumption and generalized strains}
\author{Ju Liu, Chongran Zhao, and Jiashen Guan\\
\textit{\small Department of Mechanics and Aerospace Engineering,}\\
\textit{\small Southern University of Science and Technology,}\\
\textit{\small 1088 Xueyuan Avenue, Shenzhen, Guangdong 518055, China}}
\date{}
\maketitle

\section*{Abstract}
We propose a modeling framework for finite viscoelasticity, inspired by the kinematic assumption made by Green and Naghdi in plasticity. This approach fundamentally differs from the widely used multiplicative decomposition of the deformation gradient, as the intermediate configuration, a concept that remains debated, becomes unnecessary. The advent of the concept of generalized strains allows the Green-Naghdi assumption to be employed with different strains, offering a flexible mechanism to separate inelastic deformation from total deformation. This leads to a constitutive theory in which the kinematic separation is adjustable and can be calibrated. For quadratic configurational free energy, the framework yields a suite of finite linear viscoelasticity models governed by linear evolution equations. Notably, these models recover established models, including those by \cite{Green1946} and \cite{Simo1987}, when the Seth-Hill strain is chosen with the strain parameter being $-2$ and $2$, respectively. It is also related to the model of \cite{Miehe2000} when the strain is of the Hencky type. We further extend the approach by adopting coercive strains, which allows us to define an elastic deformation tensor locally. This facilitates modeling the viscous branch using general forms of the configurational free energy, and we construct a micromechanical viscoelastic model as a representative instantiation. The constitutive integration algorithms of the proposed models are detailed. We employ the experimental data of VHB 4910 to examine the proposed models, which demonstrate their effectiveness and potential advantages in the quality of fitting and prediction. Three-dimensional finite element analysis is also conducted to assess the influence of different strains on the viscoelastic behavior.

\vspace{5mm}

\noindent \textbf{Keywords:} Viscoelasticity, Generalized strains, Green-Naghdi plasticity, Constitutive theory, Finite linear viscoelasticity, Nonlinear finite element analysis

\section{Introduction}
\label{sec:introduction}
Viscoelastomers and biological tissues are highly deformable materials that exhibit rate-dependent and dissipative mechanical responses. They play essential roles in automotive, civil, and biomedical applications. Recent years have witnessed significant progress in developing novel devices utilizing these materials, including hydrogels (\cite{Mao2017}), dielectric  elastomer (\cite{Hong2011}), hard magnetic soft materials (\cite{Li2024}), tissue engineering (\cite{Terzano2023}), to name a few. Constitutive modeling serves as a cornerstone for understanding the complex mechanical behavior of these materials, paving the way for predicting and optimizing their performance.

In the modeling of finite inelasticity, a first fundamental question is to quantify and distinguish elastic and inelastic kinematic responses within the total deformation. The multiplicative decomposition of the deformation gradient (\cite{Kroener1959,Lee1969}) has been a widely accepted and effective strategy in capturing a range of inelastic material behaviors. This kinematic assumption equivalently leads to the concept of an imaginary intermediate configuration, which has been physically well-grounded in single crystal plasticity. \cite{Sidoroff1974} introduces this kinematic assumption in the modeling of finite viscoelasticity. \cite{Lubliner1985} leveraged the multiplicative decomposition and heuristically proposed a linear evolution equation for the inverse of the viscous deformation tensor, drawing inspiration from the earlier work of \cite{Green1946}. In the 1990s, systematic development of finite strain viscoelasticity models was undertaken by several groups (\cite{Tallec1993,Lion1997,Reese1998}). More recently, \cite{Gouhier2024} revealed that those three are essentially equivalent, differing only in their formulations based on the referential, intermediate, or current configurations. This finding underscores the critical importance of the kinematic decomposition, as it fundamentally shapes the resulting constitutive relations. 

It needs to be pointed out that the multiplicative decomposition is not without shortcomings. The chosen internal variable belongs to $GL(3)_{+}$, meaning the intermediate configuration is not uniquely defined. For anisotropic materials, this becomes a severe issue as the evolution of the material symmetry groups need to be specified (\cite{Ciambella2021}). \cite{Latorre2016} introduced a reverse decomposition of the deformation gradient, enabling the formulation of anisotropic models on the referential configuration. To date, a proper setup for the intermediate configuration concept remains a topic of debate (\cite{Sadik2024,Bahreman2022}).

Without clarifying the underlying kinematic assumption, \cite{Simo1987} proposed a viscoelastic model with the viscous evolution equation governed by a linear rate equation, which can be viewed as a heuristic extension of the standard linear solid model to the finite strain regime. On top of this framework, the identical polymer chain assumption and damage model have been incorporated to account for more realistic polymer behaviors (\cite{Govindjee1992}). An appealing aspect is that the constitutive relation can be integrated by a single-step recurrence formula, which renders it quite amenable to finite element implementation. To date, this model is adopted in mainstream commercial software (\cite{Gouhier2024}). Moreover, the theory of \cite{Simo1987} does not involve the intermediate configuration concept, making it quite convenient to account for material anisotropy. Exploiting the evolution equation, \cite{Holzapfel2001} introduced additional structural tensors to model composites reinforced by two families of fibers.

Due to the linear nature of the evolution equations, both models of \cite{Lubliner1985} and \cite{Simo1987} are referred to as \textit{finite deformation linear viscoelasticity} in the literature. Unfortunately, neither author provided a derivation of these linear evolution equations from a rigorous thermomechanical framework. A critical issue with the model of \cite{Simo1987} was later identified by \cite{Govindjee2014}, who observed a finite-time blow-up solution when analyzing a freely spinning cylinder. This finding signified a concerning issue, that is, the widely used viscoelasticity model is unstable in energy. Building on an earlier effort by \cite{Holzapfel1996}, \cite{Liu2021b} addressed this issue by identifying the configurational free energy and deriving Simo's viscoelastic model based on the second law of thermodynamics. Some inconsistently used terms were rectified through the derivation. For example, the non-equilibrium stress and the thermodynamic force conjugate to the internal variable are distinguished. Moreover, \cite{Liu2021b} emphasized the importance of the relaxation property, which poses an additional design principle for the configurational free energy, ensuring the physical reliability of the model.

In \cite{Liu2021b}, it was found that a particular instantiation of Simo's model can be recovered by the configurational free energy, denoted by $\Upsilon$, taking the form
\begin{align*}
\Upsilon = \frac{\mu^{\mathrm{neq}}}{2} \left\lvert \tilde{\bm E} - \bm E^{\mathrm v} \right\rvert^2,
\end{align*}
where $\tilde{\bm E} := (\tilde{\bm C}-\bm I)/2$, $\bm E^{\mathrm v}:= (\bm \Gamma - \bm I)/2$, and $\bm \Gamma  \in \mathrm{Sym}_{+}(3)$ is the internal variable characterizing inelastic material behavior. This observation suggests that the internal variable enters into the model through the term $\tilde{\bm E}(\tilde{\bm C}) - \bm E^{\mathrm v}(\bm \Gamma)$, closely aligning with the kinematic assumption of Green and Naghdi (\cite{Green1965,Naghdi1990}). With those observations, \cite{Liu2024} extended the model of \cite{Simo1987} to the fully nonlinear regime by incorporating the Green-Naghdi kinematic assumption and the generalized strain concept due to \cite{Hill1979}. In particular, the configurational free energy was formulated as a quadratic function of $\bm E(\bm C) - \bm E^{\mathrm v}(\bm \Gamma)$ (or $\tilde{\bm E}(\tilde{\bm C}) - \bm E^{\mathrm v}(\bm \Gamma)$), in which $\bm E$ and $\bm E^{\mathrm v}$ are defined as the generalized strains of the deformation tensor $\bm C$ and the internal variable $\bm \Gamma$, respectively. It was rigorously proved that the kinematic assumption ensures the relaxation property of the non-equilibrium stress. Another notable aspect of this approach is that it does not involve the intermediate configuration, bypassing the need to specify the kinematic laws for this imaginary configuration. This is particularly appealing in modeling anisotropy. 

Choosing $\bm \Gamma \in \mathrm{Sym}_{+}(3)$ as the internal variable is inspired and influenced by the model of \cite{Simo1987} and subsequent efforts to rationalize it (\cite{Holzapfel1996,Liu2021b}). It is also closely related to the concept of \textit{plastic metric} introduced by \cite{Miehe1998} (see also \cite{Miehe2000}). In this study, we develop a viscoelastic modeling framework by choosing $\bm E^{\mathrm v} \in \mathrm{Sym}(3)$ as the internal variable. This choice leads to a new class of constitutive theories, warranting further exploration for the following reasons.
\begin{itemize}
\item A family of linear viscoelastic evolution equations can be constructed when the configurational free energy is quadratic. In particular, the models of \cite{Green1946,Lubliner1985} and \cite{Simo1987} are recovered when the generalized strains are of the Euler-Almansi and Green-Lagrange type, respectively. In this regard, the proposed framework elegantly unifies and generalizes the existing finite deformation linear viscoelasticity theories.
\item With the advent of the generalized strains, the non-coerciveness issue of the classical Seth-Hill strain family is rectified, allowing us to introduce the concept of the elastic strain $\bm E^{\mathrm e}$ as $\bm E - \bm E^{\mathrm v}$. Notice that one cannot do so with the strain of the Seth-Hill family, and this is partly the reason that Green and Naghdi avoid interpreting their kinematic assumption as the additive decomposition (\cite{Naghdi1990}). A direct consequence is that this decomposition is compatible with the additive decomposition of strain in the infinitesimal inelasticity theory. This approach allows us to model the non-equilibrium material behavior using the general strategies of hyperelasticity modeling.
\item From the computational perspective, choosing a member in $\mathrm{Sym}(3)$ as the internal variable offers convenience in designing constitutive integration algorithms. In contrast, if we use $\bm \Gamma \in \mathrm{Sym}_{+}(3)$ as the internal variable, the constitutive integration needs to guarantee its positiveness, which can be quite non-trivial (see \cite[p.~74]{Simo1992e}).
\end{itemize}

In the remainder of this article, we begin by presenting a general theory in Section \ref{sec:theory}, focusing on the generalized strains and Green-Naghdi kinematic assumption. The relationship between the developed theory and existing models is discussed in Section \ref{sec:choice-of-ISV}. Section \ref{sec:quad-energy-model} discusses the quadratic configurational energy, which leads to a family of finite linear viscoelasticity. In Section \ref{sec:micromechanical-model}, we introduce the elastic strain based on the adoption of coercive strains and propose a model based on the micromechanically inspired energy. Section \ref{sec:results} presents the calibration and finite element analysis results of the proposed models, and we draw conclusive remarks in Section \ref{sec:conclusion}.

\section{Theory}
\label{sec:theory}
In this section, we present a nonlinear finite viscoelasticity modeling framework based on the Green-Naghdi kinematic assumption. We will review the recent progress in the development of generalized strains, which are utilized here as an effective tool for separating the inelastic deformation from the total deformation.  At the beginning of our discussion, we make the following clarifications.
\begin{enumerate}
\item Einstein summation notation is \textit{not} adopted, and all summations are written explicitly with a summation sign.
\item The rank-two identity and zero tensors are denoted by $\bm I$ and $\bm O$, respectively; the rank-four identity tensor is denoted by $\mathbb I$.
\item The norm of a generic rank-two tensor $\bm A$ is defined as $\left \lvert \bm A \right\rvert := \left( \mathrm{tr}[\bm A \bm A^T] \right)^{\frac12} = \left( \bm A : \bm A \right)^{\frac12}$.
\item In the construction of a general theory, we use the rank-four viscosity tensor $\mathbb V$ to describe the dissipative material behavior. In the instantiation of models in this framework, we assume the viscosity tensor is isotropic with identical deviatoric and volumetric viscosities, that is, $\mathbb V  = \eta \mathbb I$, where $\eta$ represents the viscosity.
\item The following notations are adopted. GL(3)$_{+}$ represents the group of rank-two tensors with positive determinant; Sym(3) denotes the set of symmetric rank-two tensors; Sym(3)$_{+}$ stands for the set of symmetric positive definite rank-two tensors.
\end{enumerate}

\subsection{Continuum basis}
\label{sec:Kinematics}
We consider the initial configuration of a continuum body $\Omega_{\bm X} \subset \mathbb R^3$ as the reference configuration and assume it to be stress-free. The initial position of a material point is labeled as $\bm X \in \Omega_{\bm X}$, and the motion is described by $\bm x = \bm \varphi(\bm X, t) =  \bm \varphi_t(\bm X)$. Parameterized by the time coordinate $t$, the map $\bm \varphi_t$ is a one-to-one correspondence between the original material point at $\bm X \in \Omega_{\bm X}$ and its current location $\bm x$ at time $t$. In this work, we use a superposed dot to represent the material time derivative when it does not cause ambiguity, and we use $d(\cdot)/dt$ explicitly to denote the material time derivative of $(\cdot)$ when necessary. The displacement field of the material particle $\bm X$ is defined as $\bm U(\bm X, t) := \bm \varphi(\bm X, t) - \bm X$, and we define the velocity as the total time derivative of $\bm U(\bm X, t)$, that is, $\bm V := \dot{\bm U}$. The deformation gradient and Jacobian are, respectively, defined as
\begin{align*}
\bm F := \frac{\partial \bm \varphi_t}{\partial \bm X} \quad \mbox{and} \quad J:=\mathrm{det}(\bm F).
\end{align*}
For materials exhibiting distinct bulk and shear responses, it is useful to decompose the deformation gradient into volumetric and isochoric parts as
\begin{align}
\label{eq:Flory_decomposition}
\bm F = \bm F_{\mathrm{vol}} \tilde{\bm F}, \qquad \bm F_{\mathrm{vol}} := J^{1/3}\bm I, \qquad \tilde{\bm F} := J^{-1/3}\bm F.
\end{align}
The right Cauchy-Green deformation tensor $\bm{C} := \bm{F}^T\bm{F}$ and its modified counterpart $\tilde{\bm C} := \tilde{\bm F}^{T} \tilde{\bm F}$ enjoy the following spectral representations,
\begin{align}
\label{eq:C_spectral}
\bm C = \sum_{a=1}^{3} \lambda_a^2 \bm M_a \quad \mbox{and} \quad \tilde{\bm C} = \sum_{a=1}^{3} \tilde{\lambda}_a^2 \bm M_a,
\end{align}
in which $\lambda_a$ ($\tilde{\lambda}_a := J^{-1/3} \lambda_a$) are the (modified) principal stretches, and $\bm M_a := \bm{N}_a \otimes \bm{N}_a$ are the self-dyads of the principal referential directions $\bm N_a$, for $a=1,2,3$. Using
\begin{align*}
\bm F \bm N_a = \lambda_a \bm n_a,
\end{align*}
the spectral representations of $\bm{F}$ and $\tilde{\bm F}$ take the form
\begin{align*}
\bm{F} = \sum_{a=1}^3 \lambda_a \bm{n}_a \otimes \bm{N}_a \quad \mbox{and} \quad \tilde{\bm{F}} = \sum_{a=1}^3 \tilde{\lambda}_a \bm{n}_a \otimes \bm{N}_a,
\end{align*}
where the unit-length spatial vectors $\bm n_a$ are known as the principal spatial directions.

In this study, generalized strains play a crucial role in characterizing inelastic deformations. It is a concept proposed in the pioneering work of \cite{Hill1968,Hill1979}. The generalized Lagrangian strain $\bm E$ is defined as
\begin{align}
\label{eq:Hill_strain}
\bm{E} := \sum_{a=1}^{3} E (\lambda_a) \bm{M}_a, 
\end{align}
with $E: (0,\infty) \rightarrow \mathbb R$ being a \textit{scale function}. Analogously, the strain associated with the isochoric part of the deformation is defined as
\begin{align}
\label{eq:Hill_strain_ich}
\tilde{\bm{E}} := \sum_{a=1}^{3} E (\tilde{\lambda}_a) \bm{M}_a.
\end{align}
Two rank-four tensors
\begin{align}
\label{eq:def_Q_tensor}
\mathbb Q  := 2 \frac{\partial \bm E}{\partial \bm C} \quad \mbox{and} \quad \tilde{\mathbb Q} := 2 \frac{\partial \tilde{\bm E}}{\partial \tilde{\bm C}}
\end{align}
and two rank-six tensors
\begin{align}
\label{eq:def_L_tensor}
\bm{\mathcal L} := 2 \frac{\partial \mathbb Q}{\partial \bm C} = 4 \frac{\partial^2 \bm E}{\partial \bm C \partial \bm C} \quad \mbox{and} \quad \tilde{\bm{\mathcal L}} := 2 \frac{\partial \tilde{\mathbb Q}}{\partial \tilde{\bm C}} = 4 \frac{\partial^2 \tilde{\bm E}}{\partial \tilde{\bm C} \partial \tilde{\bm C}}
\end{align}
are salubrious in the subsequent discussions. Their explicit representations in terms of principal stretches and principal directions are documented in \cite{Miehe2001b} and \cite{Liu2024}. It needs to be cautiously pointed out that $\mathbb Q$ and $\tilde{\mathbb Q}$ are not projectors, since $\mathbb Q : \mathbb Q \neq \mathbb Q$ and $\tilde{\mathbb Q} : \tilde{\mathbb Q} \neq \tilde{\mathbb Q}$, despite having occasionally been mistakenly referred to as projectors. Additionally, one may define the generalized Eulerian strains by replacing the self-dyads $\bm M_a$ in \eqref{eq:Hill_strain} and \eqref{eq:Hill_strain_ich} by $\bm n_a \otimes \bm n_a$. 

The crux of the above definitions is the scale function. \cite{Hill1979} demands the scale function to be at least twice differentiable and satisfy the following conditions
\begin{align}
\label{eq:E_property}
E' > 0, \quad E(1) = 0, \quad \mbox{and} \quad E'(1) = 1.
\end{align}
In this work, we use $E'$ to denote the first derivative for the univariate function $E$. The smoothness requirement is essential for the definition of the elasticity tensor. The monotonicity \eqref{eq:E_property}$_1$ is crucial for material stability and ensures the existence of the inverse function of $E$, which further guarantees the existence of $\mathbb Q^{-1}$ (see Lemma 1 of \cite{Liu2024}). \cite{Curnier1991} referred to this as the \textit{regularity condition}. The condition \eqref{eq:E_property}$_2$ states that the strain vanishes in the reference configuration. The condition \eqref{eq:E_property}$_3$ renders all curves tangent to each other at the undeformed state, making all measures coincide with the infinitesimal strain near the reference configuration. \cite{Ogden1997} referred to \eqref{eq:E_property}$_3$ as the \textit{normality condition}. As an example, the Green-Lagrange strain can be instantiated by choosing $E(\lambda) = (\lambda^2-1)/2$. A major drawback of this strain is that it achieves a finite value $-1/2$ when the stretch approaches zero. From the physical perspective, this is an anomalous behavior, and it is reasonable to expect strain to approach infinity under extreme deformation. To rectify this, an additional property, known as \textit{coerciveness}, is often demanded for newly proposed strains. It is mathematically represented as
\begin{align}
\label{eq:E_coerciveness}
\lim_{\lambda \rightarrow 0} E(\lambda)= -\infty \quad \mbox{and} \quad \lim_{\lambda \rightarrow \infty} E(\lambda) = +\infty.
\end{align}
With the condition \eqref{eq:E_coerciveness}, the scale function $E$ constitutes a bijection from $\mathbb R_{+}$ to $\mathbb R$. From the mathematical perspective, non-coercive strains suffer from a compatibility issue. Not every symmetric rank-two tensor can be represented as \eqref{eq:Hill_strain} if the range of the scale function $E$ is a strict subset of $\mathbb R$. For example, if a symmetric tensor has eigenvalues less than $-1/2$, it cannot be a Green-Lagrange strain. Therefore, coerciveness is a crucial property, and most generalized strains developed after the Seth-Hill strain family satisfy this (see Table \ref{table:list_of_gen_strains}). As will be discussed in Section \ref{sec:micromechanical-model}, this property plays a critical role in modeling inelasticity in our proposed framework. Certain generalized strains also satisfy the tension-compression symmetry condition,
\begin{align}
\label{eq:E_symmetry}
E(\lambda) = -E(\lambda^{-1}).
\end{align}
However, as noted by \cite{Moerman2016} and \cite{Du2020}, materials often exhibit asymmetric behavior under tension and compression. Therefore, the symmetry property \eqref{eq:E_symmetry} is not a universally appropriate requirement for strains, though it is a mathematically elegant feature for certain scale function choices.

\begin{table}[htbp]
\begin{center}
\tabcolsep=0.19cm
\renewcommand{\arraystretch}{1.6}
\begin{tabular}{P{2.0cm} P{4.0cm} P{2.5cm} P{1.8cm} P{1.5cm} P{3.0cm}  }
\hline
Name  & $E(\lambda)$ & Parameter(s) & Coerciveness & Symmetry & References  \\
\hline
Seth-Hill  & $\frac{1}{m}\left( \lambda^m -1\right)$ & $m\neq 0$ & N & N & \cite{Doyle1956,Seth1964} \\
Hencky & $\ln(\lambda)$ & - & Y & Y & \cite{Hencky1928} \\
Curnier-Rakotomanana & $\frac{1}{m+n}\left( \lambda^m - \lambda^{-n} \right)$ & $mn \neq 0$ & Y & N & \cite{Curnier1991}  \\
Curnier-Zysset & $\frac{2+m}{8}\lambda^2 - \frac{2-m}{8}\lambda^{-2} - \frac{m}{2}$  & $-2\le m \le 2$ & Y & N & \cite{Curnier2006}\\
Darijani-Naghdabadi &$\frac{1}{m+n} \left(e^{m(\lambda-1)} - e^{n(\lambda^{-1}-1)}\right)$ & $m,n>0$ & Y & Y & \cite{Darijani2013} \\
\hline
\end{tabular}
\end{center} 
\caption{Families of generalized strains.} 
\label{table:list_of_gen_strains}
\end{table}

Table \ref{table:list_of_gen_strains} presents several families of generalized strains, among which the most widely recognized is the Seth-Hill strain family, sometimes known as the Doyle-Ericksen family. It encompasses several widely used strains, including the Euler-Almansi ($m=-2$), Hill ($m=-1$), Biot ($m=1$), and Green-Lagrange ($m=2$) strains. The primary deficiency is its failure to satisfy the coercivity property \eqref{eq:E_coerciveness}, which has motivated the development of new generalized strains.

The Hencky strain, also known as the natural or logarithmic strain, is often considered a member of the Seth-Hill strain family by taking the limiting value of zero for the parameter. The use of the logarithmic function to define strain can be traced back to the works of Imbert and Ludwik in the 19th century. Introducing it into elasticity is nowadays credited to Hencky based on his 1928 article. Interested readers may refer to \cite{Martin2018} for further historical background.

In the original work of \cite{Curnier1991}, the parameters $m$ and $n$ were demanded to be integers and the formula $\left( \lambda^m - \lambda^{-n} \right)/(m+n)$ integrated various strains named after rubber elasticians (e.g. the Pelzer, Mooney, Wall, and Rivlin strains) into a unified strain family. It is thus also known as the rubber family. \cite{Darijani2010} extended the definition by allowing the two parameters to take real values with the same sign (i.e., $mn > 0$). The two parameters enable the separate characterization of material behavior under tension and compression, which makes this strain family particularly appealing. The study of \cite{Darijani2010b} also indicates that this generalized strain family provides a satisfactory quality of fit for a variety of soft materials. When $m=n$, the Ba\v{z}ant-Itskov strain family (\cite{Bazant1998,Itskov2004}) is recovered, with the scale function $\left(\lambda^m - \lambda^{-m} \right)/2m$.

More recently, the Curnier-Zysset and Darijani-Naghdabadi strain families have been introduced. The Curnier-Zysset strains in fact can be expressed directly in terms of $\bm C$, without calculating the spectral representation \eqref{eq:C_spectral}. \cite{Curnier2006} showed that their proposed generalized strain family can significantly expand the region where the model remains rank-one convex compared to the Green–Lagrange strain. The exponential strains proposed by \cite{Darijani2013} can be considered as an extension of the two-parameter Curnier–Rakotomanana strain family, utilizing the exponential function. It is potentially advantageous for characterizing the non-Gaussian behavior of rubber-like materials and the strain hardening observed in biological tissues (\cite{Dal2021}).

\subsection{Constitutive theory}
We start by deriving constitutive relations based on the Helmholtz free energy to establish our modeling strategy. This choice of the thermodynamic potential leads to a theory that works within the classical pure displacement formulation. We will then switch to using the Gibbs free energy, which leads to a mixed formulation and ensures the well-posedness of the model in both compressible and incompressible regimes. To simplify our presentation, we start by considering a single relaxation process. In this work, we adopt the following kinematic assumption, inspired by the work of \cite{Green1965} and \cite{Naghdi1990}.

\begin{assumption}
\label{as:Ev-existence}
There exists a symmetric rank-two tensor 
\begin{align*}
\bm E^{\mathrm v} \in \mathrm{Sym}(3)
\end{align*}
that serves as the internal variable characterizing the relaxation process. It has the same invariance property as the strain $\bm E$ under superimposed rigid body motion and enters into the thermodynamic potential through the term $\bm E - \bm E^{\mathrm v}$.
\end{assumption}

In the original proposal of \cite{Green1965}, the Green-Lagrange strain was invoked for the kinematic assumption $\bm E = \bm E^{\mathrm e} + \bm E^{\mathrm p}$, in which the plastic strain $\bm E^{\mathrm p} \in$ Sym(3) is the primitive quantity measuring the permanent deformation. Formally, it can be viewed as an additive decomposition of the strain $\bm E$ into the elastic part $\bm E^{\mathrm e}$ and plastic part $\bm E^{\mathrm p}$. This approach was inspired by the infinitesimal strain theory of elastoplasticity. \cite{Lee1969} pointed out that $\bm E - \bm E^{\mathrm p}$ is not necessarily a Green-Lagrange strain. In later works, \cite{Green1971} refrained from interpreting $\bm E^{\mathrm e}$ as a strain, and they constructed the theory by designing the thermodynamic potential as a function of $\bm E$ and $\bm E - \bm E^{\mathrm p}$ (\cite{Naghdi1990}). Therefore, the term $\bm E - \bm E^{\mathrm v}$ in Assumption \ref{as:Ev-existence} is understood in an abstract way that the internal variable $\bm E^{\mathrm v}$ enters into the constitutive theory in the particular form. This kinematic assumption, however, is abstract and lacks a clear physical interpretation, rendering it less popular than the multiplicative decomposition (\cite{Kroener1959,Lee1969}). Generalization of the Green-Naghdi plasticity theory has been made over the years by utilizing the Seth-Hill strain (\cite{Papadopoulos1998,Schroeder2002}) and the Hencky strain (\cite{Miehe1998,Miehe1998a}). The latter is also referred to as the additive plasticity in the logarithmic strain space.

\begin{remark}
According to the normality condition \eqref{eq:E_property}$_3$, all generalized strains $\bm E$ are compatible with the infinitesimal strain $\bm \varepsilon$ when the deformation is sufficiently small. If we formally introduce $\bm \varepsilon^{\mathrm v}$ as the infinitesimal internal variable corresponding to $\bm E^{\mathrm v}$, Assumption \ref{as:Ev-existence} recovers the kinematic setup for the standard linear solid model \cite[Chapter~10]{Simo2006}.
\end{remark}

Based on the isothermal condition, the Helmholtz free energy $\Psi$ is postulated as a function of $\bm C$ and $\bm E^{\mathrm v}$, that is, $\Psi = \Psi( \bm{C}, \bm{E}^{\mathrm v})$.  Given the definitions of generalized strains, the potential $\Psi$ can be equivalently expressed as a function of  $\bm E = \bm E(\bm C)$ and $\bm E^{\mathrm{v}}$. A number of experiments suggest the existence of distinct hyperelastic and time-dependent contributions of the energy stored in the material (\cite{Bergstroem1998,Wang2018}), leading to the following form of the free energy according to Assumption \ref{as:Ev-existence},
\begin{align}
\label{eq:Helmholtz_eq_neq}
\Psi(\bm C, \bm E^{\mathrm v}) = \Psi^{\infty}(\bm E(\bm C)) + \Upsilon(\bm E(\bm C) - \bm E^{\mathrm v}).
\end{align}
The superscript `$\infty$' represents the contribution arising from the equilibrium response of the material. The potential $\Upsilon$ is associated with the dissipative behavior and is known as the configurational free energy (\cite{Holzapfel2000,Liu2021b}) or the non-equilibrium part of the free energy (\cite{Reese1998}). We will henceforth use the superscripts `$\infty$' and `neq' to indicate quantities derived from $\Psi^{\infty}$ and $\Upsilon$, respectively. At this stage, we do not wish to interpret the term $\bm E - \bm E^{\mathrm v}$ as an elastic strain due to the same reason mentioned above and avoided the notation $\bm E^{\mathrm e}$ in \eqref{eq:Helmholtz_eq_neq}. The strain $\bm E$ can be any of the generalized strains given in Table \ref{table:list_of_gen_strains}. To derive the constitutive relation, we consider the Clausius-Planck inequality
\begin{align}
\label{eq:Clausius_Plank_inequality}
\mathcal D := \bm S : \frac12 \dot{\bm C} - \dot{\Psi} \geq 0,
\end{align}
which essentially demands the internal dissipation $\mathcal D$ to be non-negative. With the postulated free energy form \eqref{eq:Helmholtz_eq_neq}, the inequality can be expanded as
\begin{align*}
\mathcal{D} = \left( \bm S - 2\frac{\partial \Psi}{\partial \bm C} \right) : \frac12 \dot{\bm C} + \bm T^{\mathrm{neq}}  :  \dot{\bm E}^{\mathrm v} \geq 0,
\end{align*}
with
\begin{align}
\label{eq:def_Tneq}
\bm T^{\mathrm{neq}} :=  -\frac{\partial \Psi}{\partial \bm E^{\mathrm v}} =  -\frac{\partial \Upsilon}{\partial \bm E^{\mathrm v}} = \frac{\partial \Upsilon}{\partial \bm E},
\end{align}
which is known as the thermodynamic force or conjugate force associated with $\bm E^{\mathrm v}$. To ensure the satisfaction of the inequality \eqref{eq:Clausius_Plank_inequality} for arbitrary kinematic processes, we make the following choices,
\begin{align}
\label{eq:constitutive_S_Q}
\bm S &= 2 \frac{\partial \Psi}{\partial \bm C} = \bm S^{\infty} + \bm S^{\mathrm{neq}}, \qquad \mbox{with} \quad
\bm S^{\infty} = 2 \frac{\partial \Psi^{\infty}}{\partial \bm C} \quad \mbox{and} \quad
\bm S^{\mathrm{neq}} = 2 \frac{\partial \Upsilon}{\partial \bm C}.
\end{align}
Leveraging the rank-four tensor $\mathbb Q$ defined in \eqref{eq:def_Q_tensor}, we may alternatively represent the stresses as
\begin{align}
\label{eq:constitutive_S_eq_neq}
\bm S^{\infty} = \bm T^{\infty} : \mathbb Q \quad \mbox{and} \quad \bm S^{\mathrm{neq}} = \bm T^{\mathrm{neq}} : \mathbb Q,
\end{align}
in which $\bm T^{\infty}$ is defined as $\partial \Psi^{\infty}/\partial \bm E$. 

We postulate the existence of a dissipation potential $\Phi(\dot{\bm E}^{\mathrm v})$ such that
\begin{align}
\label{eq:dissipation_potential}
\Phi(\dot{\bm E}^{\mathrm v}) = \bm T^{\mathrm{neq}} : \dot{\bm E}^{\mathrm v}.
\end{align}
The dissipation potential is demanded to be non-negative and convex, and it characterizes the dissipative behavior of the material. Based on the maximum entropy production principle (\cite{Ziegler1983,Ziegler1987}), with a prescribed thermodynamic force $\bm T^{\mathrm{neq}}$, the arrangement of the internal variable maximize the dissipation $\bm T^{\mathrm{neq}} : \dot{\bm E}^{\mathrm v}$, subject to the condition $\bm T^{\mathrm{neq}} : \dot{\bm E}^{\mathrm v} = \Phi(\dot{\bm E}^{\mathrm v})$. This extreme value problem can be solved by considering
\begin{align}
\label{eq:maximum-dissipation}
\max_{(\dot{\bm E}^{\mathrm v}, \delta)} \left\lbrace \bm T^{\mathrm{neq}} : \dot{\bm E}^{\mathrm v} - \delta \left( \Phi(\dot{\bm E}^{\mathrm v}) -  \bm T^{\mathrm{neq}} : \dot{\bm E}^{\mathrm v} \right) \right\rbrace,
\end{align}
with $\delta$ introduced as a Lagrange multiplier. Through examining the optimality condition, we have
\begin{align}
\label{eq:Tneq_orthogonality_condition}
\bm T^{\mathrm{neq}} = \gamma \frac{\partial \Phi}{\partial \dot{\bm E}^{\mathrm v}}, \quad \gamma := \frac{\delta}{1+\delta}, \quad \mbox{and} \quad \Phi(\dot{\bm E}^{\mathrm v}) = \gamma \frac{\partial \Phi}{\partial \dot{\bm E}^{\mathrm v}} : \dot{\bm E}^{\mathrm v}.
\end{align}
The multiplier can be determined from \eqref{eq:Tneq_orthogonality_condition}$_3$. If the multiplier $\delta$ is positive, the convexity of $\Phi$ guarantees that the optimality condition \eqref{eq:Tneq_orthogonality_condition} gives the maximized entropy production. The condition \eqref{eq:Tneq_orthogonality_condition}$_1$ is geometrically interpreted as an orthogonality condition, stating the thermodynamic force $\bm T^{\mathrm{neq}}$ is orthogonal to the plane tangent to the dissipation surface $\Phi(\dot{\bm E}^{\mathrm v}) = \Phi_0$. Combining \eqref{eq:def_Tneq} and \eqref{eq:Tneq_orthogonality_condition}, the evolution equation of the internal variables is given by
\begin{align}
\label{eq:evolution-eqn-psi-phi}
\frac{\partial \Upsilon}{\partial \bm E^{\mathrm v}} + \gamma \frac{\partial \Phi}{\partial \dot{\bm E}^{\mathrm v}} = \bm O,
\end{align}
with a proper initial condition of $\bm E^{\mathrm v}$. Noticing the definition of $\bm T^{\mathrm{neq}}$ given in \eqref{eq:def_Tneq}, the above can be equivalently expressed as
\begin{align}
\label{eq:evolution-eqn-T-Phi}
\bm T^{\mathrm{neq}} = \gamma \frac{\partial \Phi}{\partial \dot{\bm E}^{\mathrm v}}.
\end{align}
The derivation of the constitutive relations relies on the introduction of two potential functions, that is, the free energy $\Psi$ and the dissipative potential $\Phi$. The evolution equation for the internal variable is derived based on the principle of maximum dissipation \eqref{eq:maximum-dissipation}. This approach, known as the framework of generalized thermodynamics, provides a unified rational approach for modeling dissipative material behaviors (\cite{Zhan2023,Houlsby2000,Martyushev2006,Flaschel2023}). We notice that the condition \eqref{eq:Tneq_orthogonality_condition}$_3$ essentially demands the dissipative potential as a homogeneous function of $\dot{\bm E}^{\mathrm v}$. Due to the convexity of the dissipation potential, one may introduce its dual $\Phi^{*}(\bm T^{\mathrm{neq}})$ by performing the Legendre transformation, and \eqref{eq:evolution-eqn-T-Phi} leads to
\begin{align}
\label{eq:evolution-eqn-phi-star}
\dot{\bm E}^{\mathrm v} = \gamma \frac{\partial \Phi^{*}}{\partial \bm T^{\mathrm{neq}}}.
\end{align}
This representation of the evolution equation can be viewed as a normality rule for the viscous strain $\bm E^{\mathrm v}$.

In this study, we restrict our discussion to the following quadratic form of the dissipation potential,
\begin{align}
\label{eq:dissipation_potential_quad_form}
\Phi(\dot{\bm E}^{\mathrm v}) = \dot{\bm E}^{\mathrm v} : \mathbb V : \dot{\bm E}^{\mathrm v}.
\end{align}
In the above, $\mathbb V$ is a rank-four viscosity tensor, which is symmetric and positive semi-definite. We postulate that its inverse $\mathbb V^{-1}$ exists, satisfying $\mathbb V : \mathbb V^{-1} = \mathbb V^{-1} : \mathbb V = \mathbb I$. Noticing that the above $\Phi$ is a homogeneous function of degree $2$, we immediately have $\gamma = 1/2$ and $\delta = 1$. Then the evolution equation \eqref{eq:evolution-eqn-T-Phi} simplifies to
\begin{align}
\label{eq:constitutive_evo_eqn}
\bm T^{\mathrm{neq}} = \mathbb V : \dot{\bm E}^{\mathrm v}.
\end{align}
Given the existence of $\mathbb V^{-1}$, the above relation is expressed as
\begin{align}
\label{eq:evolution_equation}
 \dot{\bm E}^{\mathrm v} = \mathbb V^{-1} : \bm T^{\mathrm{neq}},
\end{align}
which can also be derived from \eqref{eq:evolution-eqn-phi-star}. With the above constitutive relations, the internal dissipation can be explicitly represented as
\begin{align}
\label{eq:helmholtz-dissipation}
\mathcal D = \dot{\bm E}^{\mathrm v}  : \mathbb V : \dot{\bm E}^{\mathrm v}  = \bm T^{\mathrm{neq}} : \mathbb V^{-1} : \bm T^{\mathrm{neq}},
\end{align}
which remains non-negative. The quadratic form of the dissipation potential $\Phi$ as well as the constitutive relations \eqref{eq:constitutive_evo_eqn} and \eqref{eq:evolution_equation} can be viewed as an application of the reciprocal principle of \cite{Onsager1931}.

The \textit{thermodynamic equilibrium state} is defined as the state when the rate of change of the internal variable is zero, i.e.,
\begin{align}
\label{eq:def-thermdynamic-equilibrium-state}
\dot{\bm E}^{\mathrm v} = \bm O.
\end{align}
It is straightforward to obtain the vanishment of both $\bm T^{\mathrm{neq}}$ and $\bm S^{\mathrm{neq} }$ at the thermodynamic equilibrium state, according to the relations \eqref{eq:constitutive_evo_eqn} and \eqref{eq:constitutive_S_eq_neq}$_2$. This implies the stress on the non-equilibrium branch gets fully relaxed in the thermodynamic equilibrium state. This property is physically reasonable and guarantees the model to be well-behaved. Readers may refer to Remark 1 of \cite{Reese1998} and Section 2.4.3 of \cite{Liu2021b} for discussions on this property. Models not satisfying this property can exhibit pathological behavior, see Section 4.1 of \cite{Liu2021b}. In our modeling framework, it is Assumption \ref{as:Ev-existence} that guarantees this property.

\begin{remark}
The derivation of the constitutive relations starts with the prescription of the free energy $\Psi$ and the dissipation potential $\Phi$. The Clausius-Plank inequality and the maximum entropy production principle result in the constitutive relations. This approach falls into the framework of generalized standard materials. It offers a rational and elegant way for modeling different inelastic material behaviors, including plasticity (\cite{Hackl1997,Houlsby2000}), damage (\cite{Murakami2012,Zhan2023}), etc. For related development in viscoelasticity, one may refer to \cite{Kumar2016,Sadik2024}. The particular form of the dissipation potential $\Phi$ given in \eqref{eq:dissipation_potential_quad_form} is sufficient for the purposes of this work. It can be further generalized by postulating a more general form of $\Phi$, allowing it to depend on $\bm C$, $\dot{\bm C}$, and other state variables. Such generalization allows the incorporation of non-Newtonian effects and additional dissipative mechanisms.
\end{remark}

\subsection{Treatment of compressibility}
\label{sec:treatment-of-compressibility}
In many cases, materials exhibit distinct behaviors under isochoric and volumetric deformations. In this work, compressibility is taken into account through the following decoupled form,
\begin{align}
\label{eq:helmholtz-energy-additive-form-2}
\Psi(\bm C, \bm E^{\mathrm v}) =  \Psi_{\mathrm{a}}(\bm C, \bm E^{\mathrm v}) + \Psi_{\mathrm{vol}}(J), \quad \mbox{with} \quad \Psi_{\mathrm{a}}(\bm C, \bm E^{\mathrm v}) = \Psi^{\infty}_{\mathrm{a}}(\bm C) + \Upsilon(\bm E(\bm C) - \bm E^{\mathrm v}).
\end{align}
Notably, there is no involvement of the internal variable in $\Psi_{\mathrm{vol}}$ in the energy form \eqref{eq:helmholtz-energy-additive-form-2}, implying that the bulk response is purely elastic. This choice is based on the observation of typical soft materials, whose bulk viscous effect is several orders of magnitude narrower than the viscous effect in shear deformation (\cite{Ferry1980}). Readers may refer to \cite{Gouhier2024} for a discussion on modeling bulk viscous effect using the multiplicative decomposition. The two components, $\Psi_{\mathrm a}$ and $\Psi_{\mathrm{vol}}$, in \eqref{eq:helmholtz-energy-additive-form-2} have different physical origins. The former arises due to the configurational entropy, while the latter is due to the internal energy, respectively (\cite{Bischoff2001}). The energy $\Psi_{\mathrm{vol}}$ characterizes the stiff material behavior in bulk deformation. 

A drawback of the energy form \eqref{eq:helmholtz-energy-additive-form-2} is its restriction to the compressible regime. Specifically, the energy $\Psi_{\mathrm{vol}}$ becomes an infinity-times-zero indeterminate in the incompressible limit. To address this singularity problem, we perform a Legendre transformation on the bulk energy to inspire a saddle-point theory with the pressure or a pressure-like variable entering as an independent variable. The resulting energy $G_{\mathrm{vol}}$ is defined as
\begin{align}
\label{eq:legendre-transformation}
G_{\mathrm{vol}}(P) := \inf_{J} \left \{ \Psi_{\mathrm{vol}}(J) + PJ \right \}.
\end{align}
In the above, the \textit{pressure-like} variable $P$ is defined as a conjugate variable to $-J$ with respect to $\Psi_{\mathrm{vol}}$. We may further introduce the Gibbs free energy as
\begin{align}
\label{eq:Gibbs-a-b}
G(\bm C, P, \bm E^{\mathrm v}) :=  G_{\mathrm a}(\bm C, \bm E^{\mathrm v}) + G_{\mathrm{vol}}(P), \quad \mbox{with} \quad G_{\mathrm{a}} (\bm C, \bm E^{\mathrm v}) := G^{\infty}_{\mathrm{a}}(\bm C) + \Upsilon(\bm E - \bm E^{\mathrm v}),
\end{align}
in which $G^{\infty}_{\mathrm{a}}(\bm C) := \Psi^{\infty}_{\mathrm{a}}(\bm C)$ is introduced to maintain notation consistency. It is important to note that the energy $\Psi_{\mathrm a}$ is in fact unaffected by the Legendre transformation. Due to the transformed free energy, the representation of the internal dissipation is written as
\begin{align*}
\mathcal{D} &= \bm{S} : \frac12 \dot{\bm C} - \dot{G} + \dot{P}J + P\dot{J} = \left( \bm S - \bm S^{\infty}_{\mathrm{a}} - \bm S^{\mathrm{neq}}_{\mathrm{a}} + P J \bm C^{-1} \right) : \frac12 \dot{\bm C} +  \left( J - \frac{d{G_{\mathrm{vol}}}}{d P} \right) \dot{P} + \bm T^{\mathrm{neq}}  :  \dot{\bm E}^{\mathrm v},
\end{align*}
with 
\begin{align}
\label{eq:Gibbs_a_vol_Sa}
\bm{S}^{\infty}_{\mathrm{a}} := 2\frac{\partial G_{\mathrm{a}}^{\infty}}{\partial \bm{C}} \quad \mbox{and} \quad \bm{S}_{\mathrm{a}}^{\mathrm{neq}} := 2\frac{\partial \Upsilon}{\partial \bm{C}} = \bm T^{\mathrm{neq}} : \mathbb Q.
\end{align}
Inspired by the above representation of $\mathcal D$, we make the following choices
\begin{align}
\label{eq:Gibbs_a_vol_constitutive_relation}
\rho = \rho_0 \left( \frac{d{G_{\mathrm{vol}}}}{d P} \right)^{-1}, \quad
\bm S = \bm S_{\mathrm{a}} + \bm S_{\mathrm{vol}}  =  \bm S_{\mathrm{a}}^{\infty} + \bm S_{\mathrm{a}}^{\mathrm{neq}} + \bm S_{\mathrm{vol}}, \quad \bm S_{\mathrm{vol}} = -J P \bm C^{-1},
\end{align}
and the evolution equation remains identical to \eqref{eq:evolution-eqn-psi-phi}. We mention that $p := P \circ \bm \varphi_t^{-1}$ is only part of the hydrostatic pressure of the Cauchy stress since $\bm S_{\mathrm a} : \bm C$ is not necessarily zero. We therefore cautiously refer to $P$ as a pressure-like variable here. Table \ref{table:vol_energy} lists commonly used volumetric energies $\Psi_{\mathrm{vol}}$, their counterparts $G_{\mathrm{vol}}$, and related constitutive relations. Adopting the Gibbs free energy as the outset offers the advantage of a unified description for both incompressible and compressible models, as demonstrated in the first row of Table \ref{table:vol_energy}. Additional details on the relationship between $\Psi_{\mathrm{vol}}$ and $G_{\mathrm{vol}}$ are provided in Appendix \ref{appendix:vol-energy}.

\begin{table}[htbp]
\begin{center}
\tabcolsep=0.19cm
\renewcommand{\arraystretch}{1.6}
\begin{tabular}{P{2.0cm} P{3.0cm} P{3.0cm} P{2.0cm} P{2.0cm} P{3.0cm} }
\hline
Name & $\Psi_{\mathrm{vol}}$ & $G_{\mathrm{vol}}$ & $\rho$ & $\beta$ & Reference   \\
\hline
Incompressible model & - & $P$ & $\rho_0$ & $0$ & - \\
Quadratic model &  $\kappa (J-1)^2/2$ & $P - \frac{P^2}{2\kappa}$ & $\rho_0 \left( 1 - \frac{P}{\kappa} \right)^{-1}$ &  $(\kappa - P)^{-1}$ & \cite{Peng1975} \\
ST91 model & $\frac{\kappa}{4} \left( J^2 -2\ln(J) - 1 \right)$ & $\frac{P\sqrt{P^2+\kappa^2}-P^2}{2\kappa} - \frac{\kappa}{2} \ln( \frac{\sqrt{P^2+\kappa^2}-P}{\kappa})$ & $\rho_0 \frac{\sqrt{P^2+\kappa^2} + P}{\kappa}$ & $(P^2+\kappa^2)^{-\frac12}$ & \cite{Simo1991} \\
M94 model & $\kappa (J - \ln(J) - 1)$ & $ \kappa (\ln(P+\kappa) -\ln \kappa)$ & $\rho_0(1 + \frac{P}{\kappa})$ & $(P+\kappa)^{-1}$ & \cite{Miehe1994} \\
L94 model & $\kappa (J \ln(J)-J+1)$ & $\kappa(1-e^{-\frac{P}{\kappa}})$ & $\rho_0 e^{P/\kappa}$ & $1/\kappa$ & \cite{Liu1994} \\
\hline
\end{tabular}
\end{center} 
\caption{A list of volumetric energies $\Psi_{\mathrm{vol}}$, the corresponding Gibbs volumetric energies $G_{\mathrm{vol}}$, and the constitutive relations for the density $\rho$ and isothermal compressibility factor $\beta := (d\rho/dP)/\rho$.} 
\label{table:vol_energy}
\end{table}

Oftentimes, there is a more refined structure of the energy $G_{\mathrm{a}}$ in that it can be written in the following form,
\begin{align}
\label{eq:Gibbs-energy-additive-form}
G_{\mathrm{a}}(\bm C, \bm E^{\mathrm v}) = G^{\infty}_{\mathrm{ich}}(\tilde{\bm C}) + \Upsilon(\tilde{\bm E} - \bm E^{\mathrm v}).
\end{align}
Theoretical arguments have been presented in support of this energy split, particularly for isotropic elastic materials (\cite{Sansour2008,Liu2018}). With the energy form \eqref{eq:Gibbs-energy-additive-form}, the stresses are represented as
\begin{align}
\label{eq:Gibbs_ich_vol_constitutive_relation}
\bm S = \bm S_{\mathrm{ich}} + \bm S_{\mathrm{vol}} =  \bm S_{\mathrm{ich}}^{\infty} + \bm S_{\mathrm{ich}}^{\mathrm{neq}} + \bm S_{\mathrm{vol}}, \quad
\bm S_{\mathrm{ich}}^{\infty} = J^{-\frac{2}{3}} \mathbb{P} : \tilde{\bm{S}}^{\infty}_{\mathrm{iso}}, \quad
\bm S_{\mathrm{ich}}^{\mathrm{neq}} = J^{-\frac{2}{3}} \mathbb{P} : \tilde{\bm{S}}^{\mathrm{neq}}_{\mathrm{iso}}, \quad
\bm S_{\mathrm{vol}} = -J P \bm C^{-1},
\end{align}
with
\begin{align}
\label{eq:tilde_S_infty_ich}
\tilde{\bm{S}}^{\infty}_{\mathrm{ich}} := 2\frac{\partial G_{\mathrm{ich}}^{\infty}}{\partial \tilde{\bm{C}}},
\quad
\tilde{\bm{S}}_{\mathrm{ich}}^{\mathrm{neq}} := 2\frac{\partial \Upsilon}{\partial \tilde{\bm{C}}} = \bm T^{\mathrm{neq}} : \tilde{\mathbb Q},
\quad \mbox{and}\quad 
\mathbb P := \mathbb I - \frac13 \bm C^{-1} \otimes \bm C.
\end{align}
In \eqref{eq:tilde_S_infty_ich}$_2$, we invoked the definition of $\bm T^{\mathrm{neq}}$ in \eqref{eq:def_Tneq} and noticed that $\bm T^{\mathrm{neq}} = \partial \Upsilon / \partial \tilde{\bm E}$ due to the energy form \eqref{eq:Gibbs-energy-additive-form} here. It can be shown that now $p$ is indeed the hydrostatic pressure of the Cauchy stress. We hereafter refer to $G=G_{\mathrm a} + G_{\mathrm{vol}}$ as the isochoric-volumetric decoupled form of the energy if $G_{\mathrm a}$ takes the form of \eqref{eq:Gibbs-energy-additive-form}.

\begin{remark}
Despite its elegance, the form \eqref{eq:Gibbs-energy-additive-form} does not capture all material behaviors, such as those of anisotropic materials. A convenient resolution is to incorporate the full strain tensor in the anisotropic contribution (\cite{Gueltekin2019,Helfenstein2010,Nolan2014}). In the original work of \cite{Flory1961}, compressibility was in fact treated by appending a bulk energy term without invoking the unimodular part of deformation, which inherently incorporates the ideas similar to those introduced for handling anisotropic materials. We therefore consider both the refined energy form \eqref{eq:Gibbs-energy-additive-form} as well as a general energy form \eqref{eq:Gibbs-a-b} in this study. 
\end{remark}

\subsection{Generalization to multiple relaxation processes}
\label{sec:gen-multiple-relax-process}
We extend the previous discussion to accommodate multiple relaxation processes, specifically considering the isochoric-volumetric decoupled form due to \eqref{eq:Gibbs-energy-additive-form}. Assuming the existence of $M$ independent internal variables $\bm E^{\mathrm v}_{\alpha}$ with $1\leq \alpha \leq M$, each associated with a positive semi-definite viscosity tensor $\mathbb V_{\alpha}$, the configurational energy can be expressed as
\begin{align*}
\Upsilon(\tilde{\bm C}, \bm E^{\mathrm v}_{1}, \cdots, \bm E^{\mathrm v}_{M}) := \sum_{\alpha=1}^{M} \Upsilon_{\alpha}(\tilde{\bm E}_{\alpha} - \bm E^{\mathrm v}_{\alpha}).
\end{align*}
The rationale for this formulation stems from the rheological model illustrated in Figure \ref{fig:multiple_relaxation}, where multiple Maxwell elements are arranged in parallel with an elastic spring. The non-equilibrium stress is generalized to
\begin{align*}
\bm S^{\mathrm{neq}}_{\mathrm{ich}} = J^{-\frac23} \mathbb P : \tilde{\bm S}_{\mathrm{ich}}^{\mathrm{neq}}, \quad \mbox{with} \quad
\tilde{\bm S}_{\mathrm{ich}}^{\mathrm{neq}} = \sum_{\alpha=1}^{M} \bm T^{\mathrm{neq}}_{\alpha} : \tilde{\mathbb Q}_{\alpha}, \quad \bm T^{\mathrm{neq}}_{\alpha} := - \frac{\partial \Upsilon_{\alpha}}{\partial \bm E^{\mathrm v}_{\alpha}} = \frac{\partial \Upsilon_{\alpha}}{\partial \tilde{\bm E}_{\alpha}}, \quad \mbox{and} \quad \tilde{\mathbb Q}_{\alpha} := 2 \frac{\partial \tilde{\bm E}_{\alpha}}{\partial \tilde{\bm C}}.
\end{align*}
There are $M$ evolution equations associated with each non-equilibrium process,
\begin{align}
\label{eq:evo-multiple-relax-process}
\bm T^{\mathrm{neq}}_{\alpha} = \mathbb V_{\alpha} : \dot{\bm E}^{\mathrm v}_{\alpha}, \qquad \mbox{for} \quad 1 \leq \alpha \leq M.
\end{align}
The total dissipation is
\begin{align*}
\mathcal D = \sum_{\alpha=1}^{M} \dot{\bm E}_{\alpha}^{\mathrm v} : \mathbb V_{\alpha} : \dot{\bm E}_{\alpha}^{\mathrm v} = \sum_{\alpha=1}^{M} \bm T_{\alpha}^{\mathrm{neq}} : \mathbb V_{\alpha}^{-1} : \bm T_{\alpha}^{\mathrm{neq}},
\end{align*}
and the positive semi-definiteness of the viscosity tensors ensure the non-negativity of $\mathcal D$. At the thermodynamic equilibrium state, $\dot{\bm E}^{\mathrm v}_{\alpha} = \bm O$ for $1 \leq \alpha \leq M$, leading to the relaxation of $\bm S^{\mathrm{neq}}_{\mathrm{ich}}$ and zero dissipation.

\begin{figure}
  \begin{center}
    \includegraphics[angle=0, trim=180 115 180 70, clip=true, scale = 0.45]{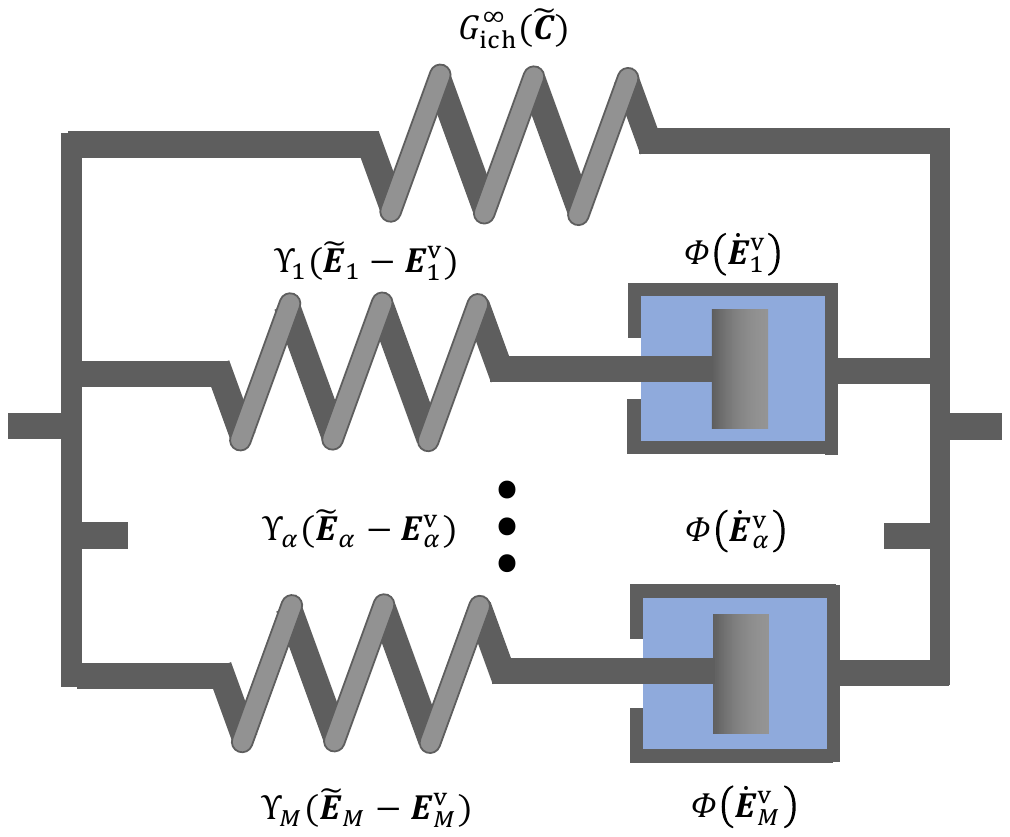}
  \end{center}
  \caption{An illustration of the multiple relaxation processes by the spring-dashpot device.}
  \label{fig:multiple_relaxation}
\end{figure}

\begin{remark}
This study develops a framework based on the spring-dashpot rheological model illustrated in Figure \ref{fig:multiple_relaxation}, commonly known as the generalized Maxwell model. An alternative approach, the generalized Kelvin-Voigt model, has not been widely studied in the literature (\cite{Huber2000,Laiarinandrasana2003}). Its limited popularity may be attributed to challenges in employing the multiplicative decomposition when incorporating multiple relaxation processes, as managing multiple intermediate configurations can be cumbersome. We feel the kinematic assumption proposed in this work offers an effective approach to model the generalized Kelvin-Voigt model at finite strains, and it remains to be thoroughly investigated in our future study.
\end{remark}

\section{Choice of internal variable and connection with existing models}
\label{sec:choice-of-ISV}
In the previous section, a general modeling framework has been established based on the Green-Naghdi kinematic assumption, with $\bm E^{\mathrm v} \in$ Sym(3) chosen as the internal variable. Here we discuss alternative models generated by different assumptions and examine their connection to the proposed framework. To simplify our discussion, we consider a single relaxation process within this section. 

\subsection{Two alternative internal variable choices}
In this part, we discuss the choice of the internal variables and emphasize their role in shaping the constitutive theory. One choice is based on the multiplicative decomposition with the internal variable being a member of GL(3)$_{+}$. Another choice is based on the Green-Naghdi assumption with the internal variable belonging to Sym(3)$_{+}$.

\paragraph{Multiplicative decomposition of \cite{Sidoroff1974}} We begin by discussing the theory based on the multiplicative decomposition, and we use the bar symbol $\bar{(\cdot)}$ to represent quantities associated with the multiplicative decomposition, e.g., $\bm F = \bar{\bm F}^{\mathrm e} \bar{\bm F}^{\mathrm v}$ and $\bar{\bm C}^{\mathrm e} = \bar{\bm F}^{\mathrm e \: T} \bar{\bm F}^{\mathrm e}$. \cite{Tallec1993} assumed the intermediate configuration does not undergo rotation and chose $\bar{\bm C}^{\mathrm v} := \bar{\bm F}^{\mathrm v \: T} \bar{\bm F}^{\mathrm v} = \bar{\bm F}^{\mathrm v \: 2}$ to characterize the viscous behavior. The Helmholtz free energy was represented as $\Psi(\bm C, \bar{\bm C}^{\mathrm e})$, while the dissipation potential was proposed in the form
\begin{align*}
\Phi(\bar{\bm C}^{\mathrm v}, \dot{\bar{\bm C}}^{\mathrm v}) = \eta \bar{\bm d}^{\mathrm v} : \bar{\bm d}^{\mathrm v}, \quad \mbox{with} \quad \bar{\bm d}^{\mathrm v} := \frac12 \bar{\bm F}^{\mathrm v \: -T} \dot{\bar{\bm C}}^{\mathrm v} \bar{\bm F}^{\mathrm v \: -1}
\end{align*}
where $\eta$ is a viscosity-like modulus. Interestingly, this dissipation potential is akin to the dissipation of viscous fluids. The evolution equation of \cite{Tallec1993} can be derived as
\begin{align}
\label{eq:Tallec_evolution_eqn}
\frac{\partial \Psi}{\partial \bar{\bm C}^{\mathrm v}} + \frac12 \frac{\partial \Phi}{\partial \dot{\bar{\bm C}}^{\mathrm v}} = 0.
\end{align}
Leveraging isotropy of the free energy $\Psi$, a more explicit form of the evolution equation can be obtained (\cite{Gouhier2024,Tallec1993}). \cite{Gouhier2024} also demonstrated that the models of \cite{Lion1997,Tallec1993,Reese1998} share the same evolution equations and are written on the intermediate, referential, and current configurations, respectively. Given the fact they adopted the same kinematic assumptions, it is not so surprising that the three ostensibly different models are identical. This observation highlights the importance of kinematic assumption, as it essentially dictates the form of the evolution equation.

\paragraph{Green-Naghdi assumption with an internal variable in Sym(3)$_{+}$} \cite{Liu2024} developed a constitutive theory by positing the existence of an internal variable 
\begin{align*}
\bm \Gamma \in \mathrm{Sym}(3)_{+},
\end{align*}
whose role is analogous to the deformation tensor $\bm C$. That choice drew inspiration from the work of \cite{Holzapfel1996} and the concept of \textit{plastic metric} proposed by Miehe (\cite{Miehe1998a,Miehe2000}). Based on the spectral decomposition of $\bm \Gamma$, a viscous strain $\bm E^{\mathrm v}$ can be defined by invoking the generalized strain concept \eqref{eq:Hill_strain}. The strain energy was constructed following an assumption analogous to Assumption \ref{as:Ev-existence}, with $\bm E(\bm C) - \bm E^{\mathrm v}(\bm \Gamma)$ entering into the configurational free energy. From the general framework, the work conjugate variable is defined as $\bm Q := -2 \partial \Psi / \partial \bm \Gamma = -2 \partial \Upsilon / \partial \bm \Gamma$ and the evolution equation is given by
\begin{align}
\label{eq:gamma_evolution_eqn}
\bm Q = \frac12 \mathbb V : \dot{\bm \Gamma}.
\end{align}
The relation between $\bm \Gamma$ and $\bm E^{\mathrm v}$ induces a rank-four tensor $\mathbb Q^{\mathrm v} := 2 \partial \bm E^{\mathrm v}/\partial \bm \Gamma$, and it can be shown that $\bm Q = \bm T^{\mathrm{neq}} : \mathbb Q^{\mathrm v}$, with $\bm T^{\mathrm{neq}}$ being defined consistently as \eqref{eq:def_Tneq}. \cite{Liu2024} showed that the non-equilibrium stress vanishes in the thermodynamic equilibrium state, i.e., the state with $\dot{\bm \Gamma} = \bm O$. Despite the similarities with the proposed model, the evolution equation \eqref{eq:gamma_evolution_eqn} in fact can be reformulated as
\begin{align}
\label{eq:gamma_evolution_eqn_detail}
\bm T^{\mathrm{neq}} = \mathbb V : \left( \mathbb Q^{\mathrm v \: -1} : \dot{\bm E}^{\mathrm v} : \mathbb Q^{\mathrm v \: -1} \right),
\end{align}
which differs from the evolution equation \eqref{eq:evolution_equation} proposed in this work. The distinction lies in the presence of two $\mathbb Q^{\mathrm v \: -1}$ acting on both sides of $\dot{\bm E}^{\mathrm v}$. Although both frameworks adopt the Green-Naghdi kinematic assumption, the difference stems from the choice of the internal variable. Unless $\mathbb Q^{\mathrm v} = \mathbb I$, the evolution equation \eqref{eq:gamma_evolution_eqn_detail} is inherently nonlinear. Additionally, adopting an internal variable in Sym(3)$_{+}$ demands special consideration in integrating the constitutive relation, as most integration schemes fail to preserve the positive definiteness (\cite[p.~74]{Simo1992e}).

\subsection{Connection with existing models}
\label{sec:connection-with-existing-models}
In the second part of this section, we discuss three finite viscoelasticity models in the literature and reveal how they are related to the framework developed in Section \ref{sec:Kinematics}. This reveals the intricate connection among different models and paves the way for our subsequent discussions. 

\paragraph{The identical polymer chain model}
Among different instantiations of Simo's model, a particularly relevant and important model is the identical polymer chain model (\cite{Govindjee1992,Holzapfel1996}), which uses a scaled hyperelastic stress to drive the dissipative evolution equation. It is inspired by the phenomenological observation that the viscous effect is induced in a medium composed of the same polymer chains. \cite{Liu2021b} analyzed its thermomechanical foundation and provided its configurational free energy as
\begin{align}
\label{eq:FLV-Upsilon}
\Upsilon(\bm C, \bm \Gamma) = \frac{1}{4\mu} \left\lvert 2 \frac{\partial \mathcal G(\tilde{\bm C})}{\partial \tilde{\bm C}} - \mu(\bm \Gamma - \bm I) - \hat{\bm S}_0 \right\rvert^2.
\end{align}
The constant tensor $\hat{\bm S}_0$ specifies the initial value of $\bm Q= -2 \partial \Upsilon / \partial \bm \Gamma$, and $\mu$ is the shear modulus associated with the relaxation process. The above configurational free energy does not fall into the form of $\Upsilon(\bm E - \bm E^{\mathrm v})$, and its internal variable is $\bm \Gamma \in$ Sym(3)$_{+}$. The relaxation property of the non-equilibrium stress deserves special analysis and was confirmed by \cite{Liu2021b}. Following the identical polymer chain assumption, the energy-like function $\mathcal G$ in \eqref{eq:FLV-Upsilon} is specified as $\mathcal G(\tilde{\bm C}) = \beta G^{\infty}_{\mathrm{ich}}(\tilde{\bm C})$,
with $\beta \in \mathbb R_{+}$ being a dimensionless scaling factor. It can be shown that the evolution equation for this model is given by $\eta \dot{\bm \Gamma} = \beta \tilde{\bm S}^{\infty}_{\mathrm{ich}} - \hat{\bm S}_0 - \mu (\bm \Gamma - \bm I)$, in which $\tilde{\bm S}^{\infty}_{\mathrm{ich}} = 2 \partial G^{\infty}_{\mathrm{ich}} / \partial \tilde{\bm C}$. As was pointed out, integrating this equation is non-trivial as one has to preserve the positiveness of the internal variable in the discrete scheme. Taking a time derivative at both sides of the equation results in
\begin{align}
\label{eq:FLV-evo-eqn}
\dot{\bm Q} + \frac{\bm Q}{\tau} = \beta \dot{\tilde{\bm S}}^{\infty}_{\mathrm{ich}}.
\end{align}
Given an initial condition $\bm Q|_{t=0} = \bm Q_0$, the solution can be represented in a hereditary integral form,
\begin{align}
\label{eq:FLV-hereditary-integral}
\bm Q = \exp\left( - t / \tau \right) \bm Q_0 + \int^t_{0^+} \exp\left( -(t-s)/\tau \right) \beta \frac{d}{ds} \left( \tilde{\bm S}^{\infty}_{\mathrm{ich}} \right) ds.
\end{align}
The hereditary integral inspires a recurrence formula for the constitutive integration, rendering the algorithm implementation quite convenient (see, e.g. \cite[Chapter~10]{Simo2006}). With the identical polymer chain assumption, one may account for anisotropy by introducing structural tensors into the definition of $\mathcal G$ (\cite{Holzapfel2001,Wollner2023}). 

\begin{remark}
If $\hat{\bm S}_0 = \bm O$ and the energy-like function $\mathcal G$ is modeled as $\mathcal G = \mu | (\tilde{\bm C} - \bm I)/2 |^2$, the configurational free energy \eqref{eq:FLV-Upsilon} can be reorganized as $\Upsilon(\tilde{\bm C}, \bm \Gamma) = \mu |\tilde{\bm E}(\tilde{\bm C}) - \bm E^{\mathrm v}(\bm \Gamma)|^2$, where both strains, $\tilde{\bm E}$ and $\bm E^{\mathrm v}$, are of the Green-Lagrange type. This corresponds exactly to the energy mentioned in Section \ref{sec:introduction} if we set $\mu = \mu^{\mathrm{neq}}/2$.
\end{remark}

\paragraph{The model of \cite{Green1946}} The outset of our derivation is based on the existence of the internal variable $\bm E^{\mathrm v}$, and we do not consider it to be defined as a tensorial function of a symmetric positive tensor. Nevertheless, here we temporarily regard $\bm E^{\mathrm v}$ defined as a generalized strain of $\bm \Gamma \in$ Sym(3)$_{+}$ and assume the configurational free energy in a quadratic form,
\begin{align*}
\Upsilon(\bm E, \bm E^{\mathrm v}) = \frac{\mu^\mathrm{neq}}{2} \left\lvert \bm E - \bm E^{\mathrm v} \right\rvert^2.
\end{align*}
Then, the evolution equation \eqref{eq:evolution_equation} can be written as
\begin{align*}
\dot{\bm E}^{\mathrm v} = \frac{\mu^{\mathrm{neq}}}{\eta} \left( \bm E - \bm E^{\mathrm v} \right).
\end{align*}
If both $\bm E$ and $\bm E^{\mathrm v}$ are chosen as the Euler-Almansi strain, the above evolution equation becomes
\begin{align*}
\frac{d}{dt}\bm \Gamma^{-1} = \frac{\mu^{\mathrm{neq}}}{\eta} \left( \bm C^{-1} - \bm \Gamma^{-1} \right),
\end{align*}
which is the evolution equation of the model proposed by \cite{Green1946}. \cite{Vernerey2017} proposed a kinetic theory that recovers the form of the above evolution equation and the internal variable can be interpreted as a tensor characterizing the chain distribution. \cite{Lubliner1985} invoked the multiplicative decomposition, adopted the above evolution equation, and heuristically interpreted the internal variable as $\bar{\bm C}_{\mathrm v}$. His evolution equation
\begin{align*}
\frac{d}{dt}\bar{\bm C}_{\mathrm v}^{-1} = \frac{\mu^{\mathrm{neq}}}{\eta} \left( \bm C^{-1} - \bar{\bm C}_{\mathrm v}^{-1} \right)
\end{align*}
has been used widely for soft materials (\cite{Hossain2012,Stewart2023}). \cite{Reese1998} showed that the above evolution equation can be recovered through linearization of \eqref{eq:Tallec_evolution_eqn}.

\paragraph{The Model based on the inelastic metric}
Following the above discussion, we consider $\bm E^{\mathrm v}$ defined as the Hencky strain of $\bm \Gamma \in$ Sym(3)$_{+}$ here. If $\bm E$ is also chosen as the Hencky strain, the configurational free energy is a function of $\bm E - \bm E^{\mathrm v} = \ln \bm C - \ln \bm \Gamma$. Noting that the Hencky strain is coercive, one may introduce an logarithmic elastic strain as $\bm E^{\mathrm e} := \bm E - \bm E^{\mathrm v}$, which can be represented as the Hencky strain defined by a positive-definite tensor $\bm C^{\mathrm e}$. Moreover, if $\bm C$ and $\bm \Gamma$ commute, the relation $\ln \bm C^{\mathrm e} = \ln \left( \bm C \bm \Gamma^{-1} \right)$ leads to
\begin{align*}
\bm C^{\mathrm e} = \bm C \bm \Gamma^{-1},
\end{align*}
which in fact is the decomposition proposed by \cite{Miehe1998}. In that work, $\bm \Gamma$ was introduced as a Lagrangian covariant plastic metric defined a priori. The coerciveness of the Hencky strain facilitates the definition of $\bm C^{\mathrm e}$, which conceptually describes the elastic deformation. This enables the characterization of elastic behavior using more flexible functional forms. \cite{Miehe2000} designed a model for rubbery polymers using the Ogden model based on the principal stretches of $\bm C^{\mathrm e}$ to describe the response in the viscoelastic and plastoelastic branches. On the numerical integration side, \cite{Miehe2000} formulated the evolution equation for the plastic metric in such a way that the exponential integrator can be applied. In doing so, one may guarantee the positiveness of the internal variable. Upon recognizing the one-to-one relationship between $\bm \Gamma$ and $\ln \bm \Gamma$, \cite{Miehe2002} switched to using the logarithmic strain of the plastic metric as the internal variable, and the design of the constitutive integration becomes more flexible.

\section{A family of finite linear viscoelasticity models}
\label{sec:quad-energy-model}
Our discussion in Section \ref{sec:connection-with-existing-models} reveals that the quadratic configurational free energy leads to a linear evolution equation for the internal variable, and different strain choices lead to different theories. For example, the Green-Lagrange strain results in the model of \cite{Simo1987} and the Euler-Almansi strain leads to the model of \cite{Green1946,Lubliner1985}. In the realm of hyperelasticity, the models of Hill's class are constructed with a quadratic strain energy, and the material nonlinearity is characterized by the use of generalized strains. In this section, we proceed to specify the energy forms based on the isochoric-volumetric decoupled Gibbs free energy, as given in \eqref{eq:Gibbs-energy-additive-form}. The equilibrium part of the energy $G^{\infty}_{\mathrm{ich}}$ can be constructed based on existing hyperelasticity theory, and our focus is on the quadratic form of the configurational free energy. To simplify the discussion, we consider a single relaxation process, with generalization to multiple relaxation processes as outlined in Section \ref{sec:gen-multiple-relax-process}. 

\subsection{Hyperelasticity of Hill's class}
\label{sec:hyperelasticity-of-Hill-class}
The modeling strategy is inspired by the hyperelasticity of Hill's class, whose strain energy is represented as a quadratic function form of a generalized strain. As an example, given a generalized strain $\tilde{\bm E}$, the equilibrium behavior is represented as 
\begin{align}
\label{eq:hyper-quadratic-energy}
G^{\infty}_{\mathrm{ich}}(\tilde{\bm C}) = \frac{\mu^\infty}{2} \left\lvert \tilde{\bm E}(\tilde{\bm C}) \right \rvert^2.
\end{align}
If one adopts the Seth-Hill strain family, the parameter $m$ of the strain and the shear modulus $\mu^\infty$ need to be calibrated to fit experimental data. Over the years, the advent of the generalized strains has significantly enriched the hyperelasticity of Hill's class (\cite{Korobeynikov2022}). Moreover, \cite{Beex2019} suggested extending model's capabilities by using multiple quadratic terms. For example, the strain energy can be expressed as
\begin{align}
\label{eq:gen-hyper-quadratic-energy}
G^{\infty}_{\mathrm{ich}}(\tilde{\bm C}) = \sum_{\beta=1}^{N} \frac{\mu^\infty_{\beta}}{2} \left\lvert \tilde{\bm E}_{\beta}(\tilde{\bm C}) \right \rvert^2,
\end{align}
by employing $N$ quadratic terms. 
\begin{remark}
It is worth pointing out that the term $\lvert \tilde{\bm E}_{\beta} \rvert^2$ can be viewed as an invariant of $\tilde{\bm E}_{\beta}$, i.e., $\mathrm{tr} \tilde{\bm E}_{\beta}^2$. Replacing $\mathrm{tr} \tilde{\bm E}_{\beta}^2$ by the invariant $\mathrm{tr} \tilde{\bm E}_{\beta}$ in \eqref{eq:gen-hyper-quadratic-energy}, we have
\begin{align*}
G^{\infty}_{\mathrm{ich}}(\bm C) = \sum_{\beta=1}^{N} \frac{\mu^\infty_{\beta}}{2} \mathrm{tr} \tilde{\bm E}_{\beta} \quad \mbox{with} \quad \mathrm{tr} \tilde{\bm E}_{\beta} = \sum_{a=1}^{3} \frac{\lambda_a^{m_{\beta}}-1}{m_{\beta}},
\end{align*}
which is in fact the Ogden model. It is intriguing to notice that the above establishes an interesting relation between the Ogden model and the model of Hill's class through the invariants of $\tilde{\bm E}_{\beta}$.
\end{remark}

\begin{figure}
\begin{center}
\begin{tabular}{ccc}
\includegraphics[angle=0, trim=20 25 160 110, clip=true, scale=0.075]{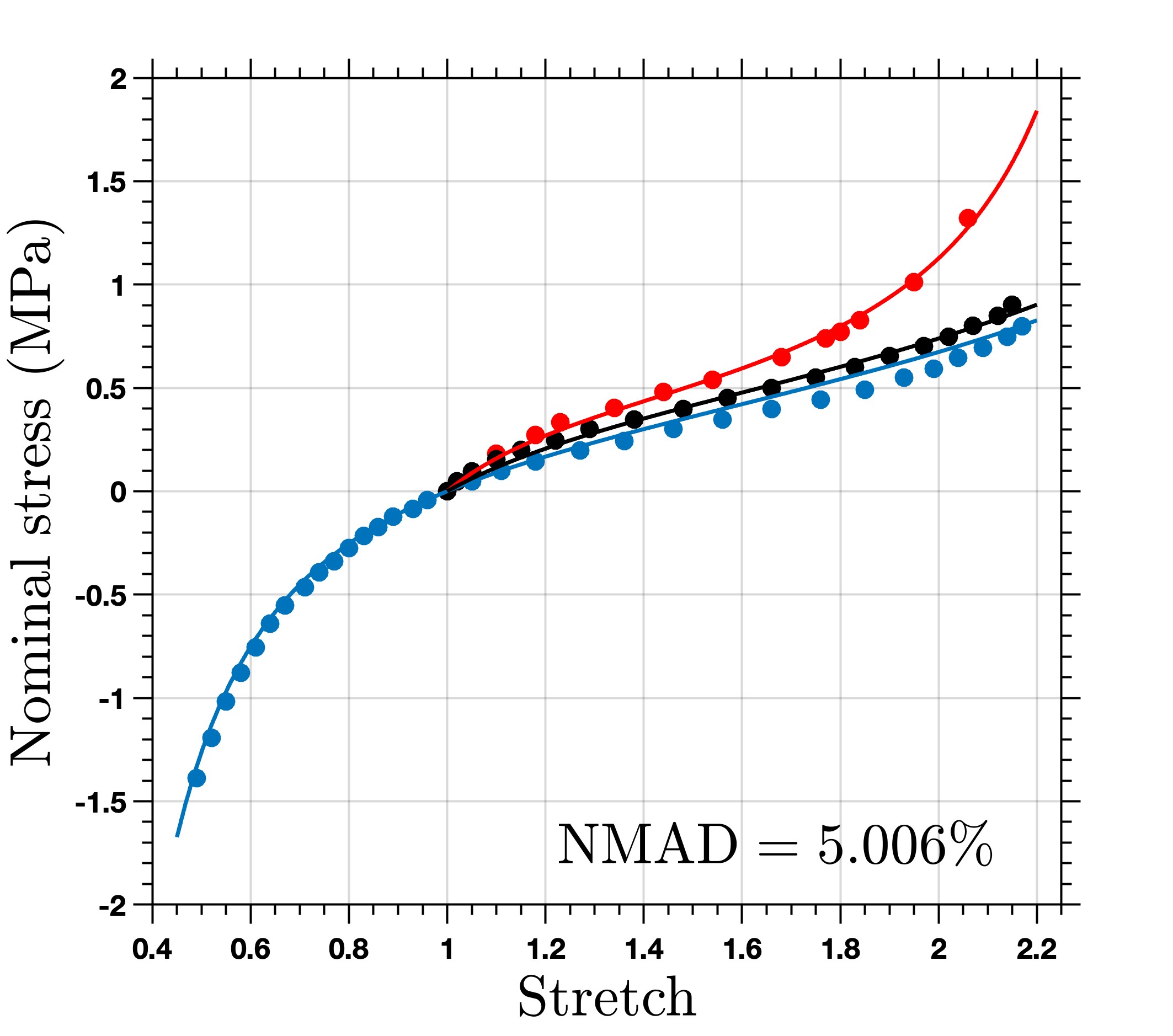} &
\includegraphics[angle=0, trim=130 25 160 110, clip=true, scale=0.075]{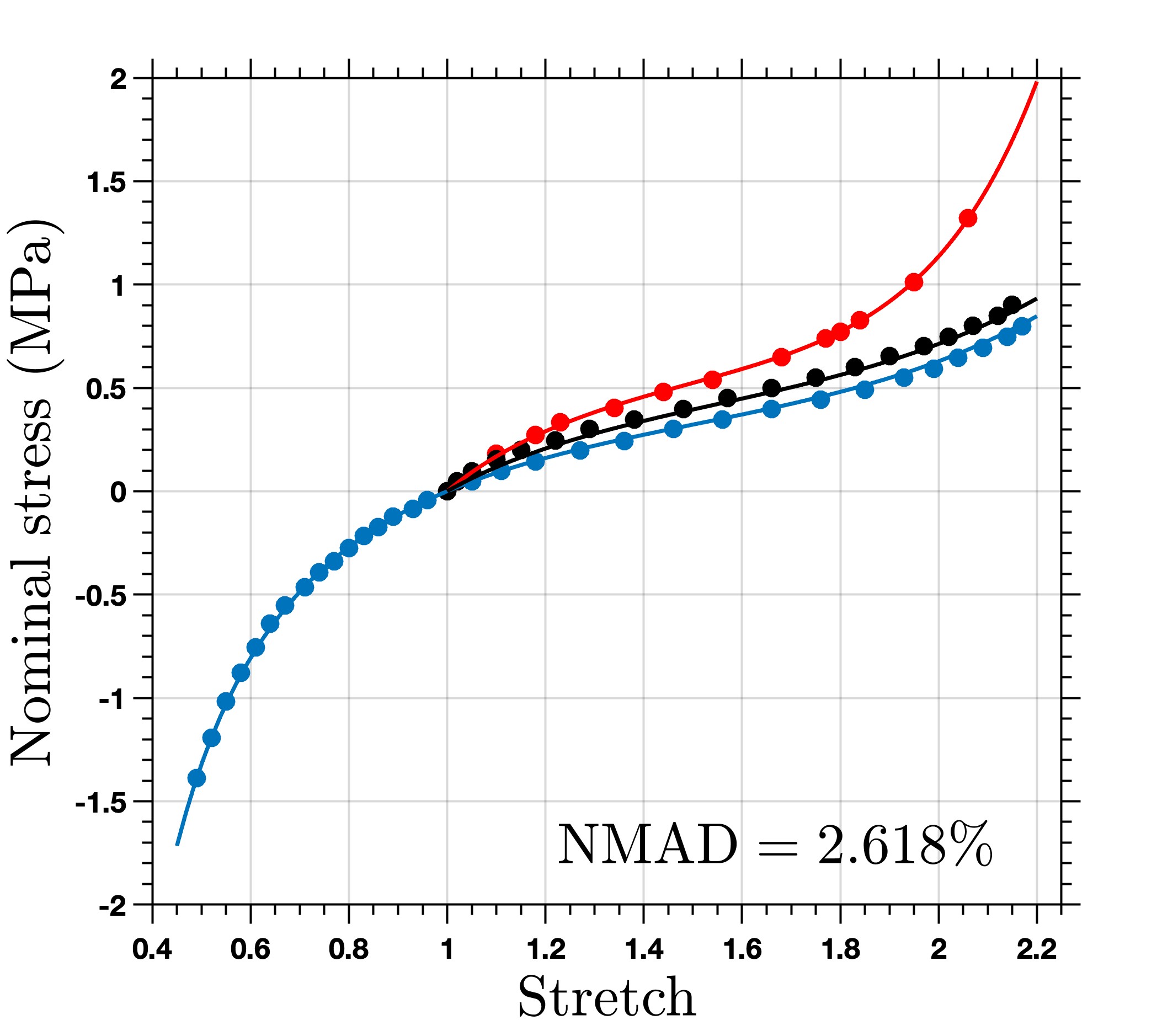} &
\includegraphics[angle=0, trim=130 25 160 110, clip=true, scale=0.075]{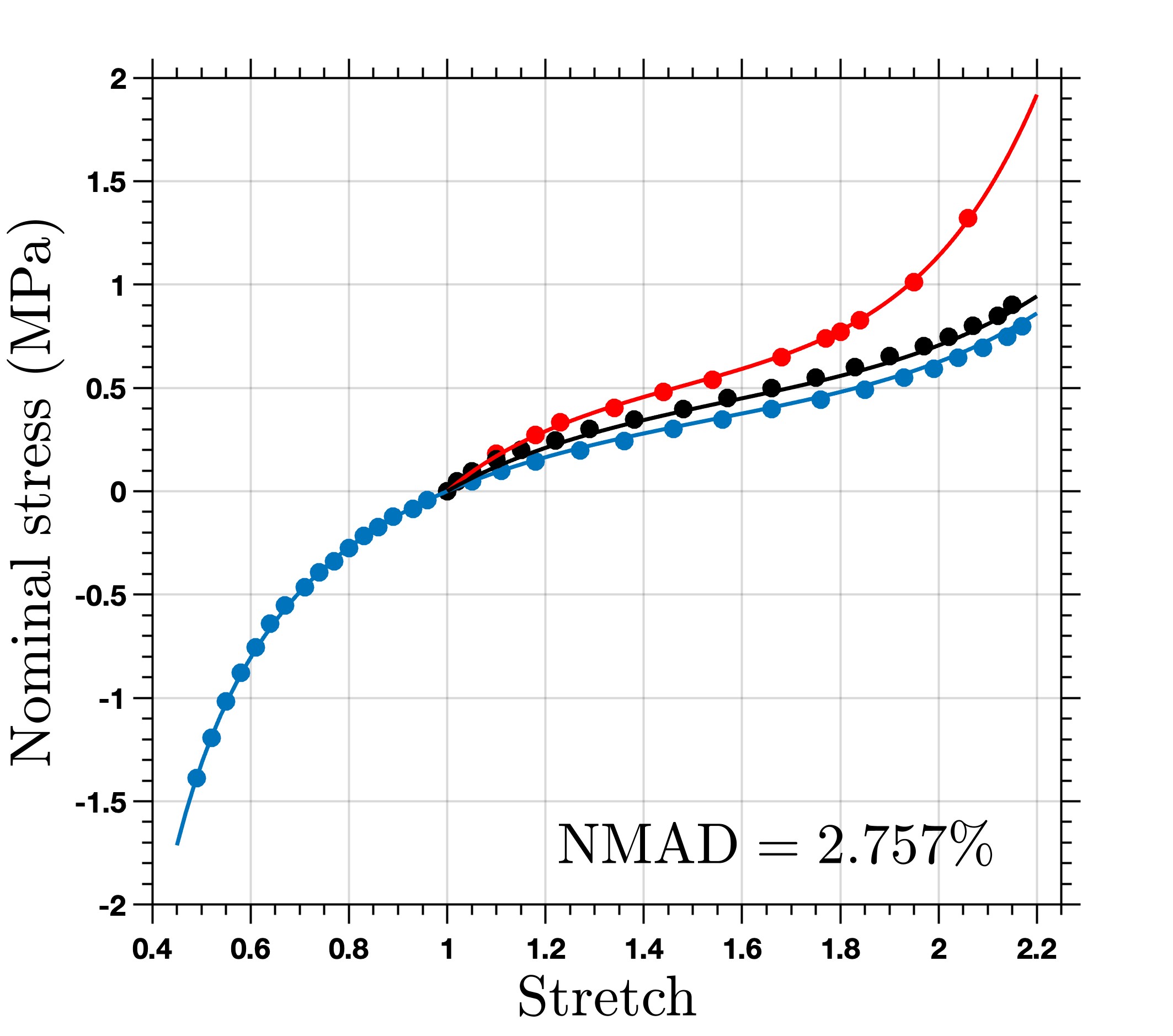}
\end{tabular}
\end{center}
\caption{Fitting results of the unfilled silicone rubber experiment of \cite{Meunier2008}: the eight-chain model (left), the three-term Ogden model (middle), and the model of Hill's class with two Curnier-Rakotomanana strains (right), under the uniaxial tension (blue), equi-biaxial tension (red) and pure shear (black) test conditions.}
\label{fig:hyperelasticity_fitting}
\end{figure}

To identify the material parameters and assess the models, we invoke the normalized mean absolute difference (NMAD) as the metric, which evaluates the error of the model in percentage relative to the experimental data. Let 
\begin{align*}
\bm e := \left\lbrace P_{\mathrm{exp}}(\lambda_i) \right\rbrace_{i=1}^{\mathcal N}, \quad \bm p := \left\lbrace P_{\mathrm{sim}}(\lambda_i) \right\rbrace_{i=1}^{\mathcal N}, \quad \mbox{and} \quad \bm e - \bm p := \left\lbrace P_{\mathrm{exp}}(\lambda_i) - P_{\mathrm{sim}}(\lambda_i) \right\rbrace_{i=1}^{\mathcal N}
\end{align*}
represent the set of the experimental data, the model predicted data, and their difference, respectively. Here $P_{\mathrm{exp}}(\lambda_i)$ and $P_{\mathrm{sim}}(\lambda_i)$ represent the experimentally measured and predicted nominal stress values at the stretch value $\lambda_i$, and there are a total of $\mathcal N$ data points in the data set. There may exist non-unique parameters that can all fit one single set of experimental data, especially for material models with multiple parameters (\cite{Dal2021,Ogden2004}). Therefore, more than one loading condition is utilized during the process of the model calibration, which is known as simultaneous fitting. For the $j$-th loading condition, with $j=1,\cdots, \mathcal M$, the experimental and predicted data are denoted by $\{\bm e^j \}$ and $\{ \bm p^j\}$, respectively, both with $\mathcal N_j$ data points. Here $\mathcal M$ denotes the total number of different experimental data sets. The NMAD of the simultaneous fitting is defined as
\begin{align}
\label{eq:NMAD}
\text{NMAD} := \frac{1}{\mathcal{M}} \sum_{j=1}^{\mathcal{M}} \frac{ \langle \bm{e}^j - \bm{p}^j \rangle_j }{\max(\langle \bm{e}^j \rangle_j, \langle \bm{p}^j \rangle_j)} 100 \%
\quad
\mbox{with}
\quad
\langle (\cdot) \rangle_j := \frac{1}{\mathcal N_j} \sum_{i=1}^{\mathcal N_j} | (\cdot)^{j}_{i} |.
\end{align}
In the above, $ \langle (\cdot) \rangle_j $ is the mean absolute value of the $j$-th generic data set. In this work, the above NMAD serves as the metric for evaluating the quality of fit. A smaller NMAD value indicates better model fitting performance.

To demonstrate the capability of the model \eqref{eq:gen-hyper-quadratic-energy}, we use the experimental data from \cite{Meunier2008}, which includes uniaxial, equi-biaxial, and pure shear tests for the unfilled silicone rubber. The calibration of three distinct material models is shown in Figure \ref{fig:hyperelasticity_fitting}. In the figure, the dots represent the experimental data, while the solid lines depict the fitted model curves. The eight-chain model, with only two parameters, fails to simultaneously capture the responses of uniaxial and biaxial tensile tests. In contrast, the Ogden model and the quadratic strain energy model based on two Curnier-Rakotomanana strains, both utilizing six parameters, accurately capture all three fundamental tests. Both models achieve comparable and notably lower NMAD values, demonstrating the effectiveness of the model of Hill's class.

\subsection{Quadratic configurational free energy}
\label{sec:quadratic_configurational_free_energy}
Consider the configurational free energy expressed as
\begin{align}
\label{eq:Upsilon-Hill-energy}
\Upsilon(\tilde{\bm C}, \bm E^{\mathrm v}) = \frac{\mu^\mathrm{neq}}{2} \left\lvert \tilde{\bm E} - \bm E^{\mathrm v} \right\rvert^2,
\end{align}
where $\mu^{\mathrm{neq}}$ denotes the shear modulus. In this design, the generalized strains serve two roles. They offer a flexible way for separating the inelastic deformation from the total deformation. At the same time, they characterize the nonlinear material behavior within the framework of hyperelasticity of Hill's class. The functional form of the energy \eqref{eq:Upsilon-Hill-energy} leads to the evolution equation 
\begin{align}
\label{eq:evo-eqn-hill-original-form}
\eta \dot{\bm E}^{\mathrm v} = \mu^\mathrm{neq} \left( \tilde{\bm E} - \bm E^{\mathrm v} \right).
\end{align}
Recalling that the thermodynamic equilibrium state is defined by $\dot{\bm E}^{\mathrm v} = \bm O$, it can be equivalently characterized by $\tilde{\bm E} = \bm E^{\mathrm v}$ from the evolution equation above. If we introduce $\tau := \eta / \mu^{\mathrm{neq}}$ as the relaxation time of the non-equilibrium process, the evolution equation \eqref{eq:evo-eqn-hill-original-form} is simplified to
\begin{align}
\label{eq:evo-eqn-hill-Ev}
\dot{\bm E}^{\mathrm v} + \frac{1}{\tau} \bm E^{\mathrm v} = \frac{1}{\tau} \tilde{\bm E}.
\end{align}
Taking a time derivative on \eqref{eq:evo-eqn-hill-original-form} results in an evolution equation for $\bm T^{\mathrm{neq}}$, which reads as
\begin{align}
\label{eq:evo-eqn-hill-T}
\dot{\bm T}^{\mathrm{neq}} + \frac{1}{\tau} \bm T^{\mathrm{neq}} = \mu^\mathrm{neq} \dot{\tilde{\bm E}}.
\end{align}
In the derivation of \eqref{eq:evo-eqn-hill-T}, we assume the moduli remain constant over time. Otherwise, there will be an extra term associated with $\dot{\mu}^\mathrm{neq}$ on the right-hand side of \eqref{eq:evo-eqn-hill-T}. Given the initial condition
\begin{align*}
\bm E^{\mathrm v} \Big|_{t=0} = \bm E^{\mathrm v}_{0},
\end{align*}
the solution of \eqref{eq:evo-eqn-hill-Ev} can be represented in the hereditary form
\begin{align}
\label{eq:exp_integration_Ev}
\bm E^{\mathrm v} = \exp\left( - \frac{t}{\tau} \right) \bm E^{\mathrm v}_{0} + \int^t_{0^+} \exp\left( -\frac{t-s}{\tau} \right) \frac{1}{\tau} \tilde{\bm E}  ds.
\end{align}
Similarly, integrating \eqref{eq:evo-eqn-hill-T} directly for $\bm T^{\mathrm{neq}}$ with the initial data $\bm T^{\mathrm{neq}}_{0}$ yields
\begin{align}
\label{eq:exp_integration_T}
\bm T^{\mathrm{neq}} = \exp\left( - \frac{t}{ \tau } \right) \bm T^{\mathrm{neq}}_{0} + \int^t_{0^+} \exp\left( -\frac{t-s}{\tau} \right) \mu^\mathrm{neq} \frac{d\tilde{\bm E}}{ds}  ds.
\end{align}
An interesting observation is the similarity between the evolution equations \eqref{eq:evo-eqn-hill-Ev}-\eqref{eq:evo-eqn-hill-T} and \eqref{eq:FLV-evo-eqn} from the identical polymer chain model. The similarity lies in that the evolution equations are linear with respect to the internal variables, leading to the analogous hereditary integral forms for their solutions. In both cases, the evolution equations \eqref{eq:evo-eqn-hill-T} and \eqref{eq:FLV-evo-eqn} are driven by nonlinear stress-like tensors. However, in \eqref{eq:FLV-evo-eqn} the right-hand side comes from the identical polymer chain assumption and the specific form of the configurational free energy \eqref{eq:FLV-Upsilon}. In the model developed here, the right-hand side of the evolution equation comes from the generalized strains $\tilde{\bm E}$ adopted in the configurational free energy.

\begin{remark}
To emphasize the impact of the choice of the internal variable, we consider the configurational free energy
\begin{align*}
\Upsilon = \frac{\mu^\mathrm{neq}}{2} \left\lvert \tilde{\bm E}(\tilde{\bm C}) - \bm E^{\mathrm v}(\bm \Gamma) \right\rvert^2,
\end{align*}
where the internal variable $\bm \Gamma \in$ Sym(3)$_{+}$ as discussed in Section \ref{sec:choice-of-ISV}. It results in the evolution equation
\begin{align*}
\frac12 \dot{\bm\Gamma} = \frac{1}{\tau}(\tilde{\bm E} - \bm E^{\mathrm v}) : \mathbb Q^{\mathrm v}.
\end{align*}
This equation is nonlinear in general and requires local Newton-Raphson iteration for constitutive integration.
\end{remark}

\begin{remark}
Our proposed model is based on $\bm E^{\mathrm v} \in$ Sym(3), without further specifying its dependence on $\bm \Gamma \in$ Sym(3)$_{+}$. Nevertheless, if we postulate the existence of $\bm \Gamma$ that defines $\bm E^{\mathrm v}$ as a generalized viscous strain, the proposed model recovers several existing models, following our discussion made in Section \ref{sec:connection-with-existing-models}. In particular, if $\bm E^{\mathrm v}$ is a Seth-Hill strain, the proposed model corresponds to the models of \cite{Green1946}, \cite{Miehe2000}, and \cite{Simo1987}, if the strain parameter is set to $m=-2$, $0$, and $2$, respectively.
\end{remark}

\subsection{Constitutive integration}
\label{sec:quadratic_constutitve_integration}
In this section, a detailed derivation of the stress and elasticity tensor is provided for the material model, incorporating both equilibrium and non-equilibrium components. Let $(0,T)$ be the time interval of interest. It is divided into $n_\mathrm{ts}$ subintervals of size $\Delta t_n = t_{n+1} - t_n$, bounded by a discrete time vector $\left\{ t_n \right\}_{n=0}^{n_\mathrm{ts}}$. In the following, $\left. (\cdot)\right|_{t_n}$ denotes the value of $(\cdot)$ evaluated at time $t_n$, and $(\cdot)_n$ denotes the discrete approximation of $(\cdot)$ at the time instance $t_n$. The elasticity tensor at time $t_{n+1}$ is given by the following explicit expression \cite[pp.~255-256]{Holzapfel2000},
\begin{align*}
\mathbb C_{\mathrm{ich}\:n+1} :=& 2\frac{\partial \bm{S}_{\text{ich}\:n+1}}{\partial \bm{C}_{n+1}} \displaybreak[2] \\
=& \mathbb P_{n+1} : \tilde{\mathbb C} _{\mathrm{ich}\:n+1} : \mathbb P_{n+1}^\mathrm{T} + \frac 23 \mathrm{Tr} \left(J_{n+1}^{-2/3} \tilde{\bm S}_{\mathrm{ich}\:n+1} \right) \tilde{\mathbb P}_{n+1} - \frac 23 \left(\bm C_{n+1}^{-1} \otimes\bm S_{\mathrm{ich}\:n+1}  + \bm S_{\mathrm{ich}\:n+1} \otimes \bm C_{n+1}^{-1} \right),
\end{align*}
with
\begin{gather}
\mathbb P_{n+1} := \mathbb I - \frac13 \bm C^{-1}_{n+1} \otimes \bm C_{n+1}, \quad  \mathrm{Tr}(\cdot):=(\cdot):\bm C_{n+1}, \quad \mbox{and} \quad \tilde{\mathbb P}_{n+1} := \bm C_{n+1}^{-1} \odot \bm C_{n+1}^{-1} - \frac 13 \bm C_{n+1}^{-1} \otimes\bm C_{n+1}^{-1}, \\
\label{eq:FLV-elasticity-tensor}
\tilde{\mathbb C}_{\mathrm{ich}\:n+1} = \tilde{\mathbb C}^{\infty}_{\mathrm{ich}\:n+1} + \tilde{\mathbb C}^{\mathrm{neq}}_{\mathrm{ich}\:n+1}, \quad \tilde{\mathbb C}^{\infty}_{\mathrm{ich}\:n+1}:= 2 J_{n+1}^{-\frac43} \frac{\partial \tilde{\bm S}^{\infty}_{\mathrm{ich}\:n+1}}{\partial \tilde{\bm C}_{n+1}}, \quad \tilde{\mathbb C}^{\mathrm{neq}}_{\mathrm{ich}\:n+1}:= 2 J_{n+1}^{-\frac43} \frac{\partial \tilde{\bm S}^{\mathrm{neq}}_{\mathrm{ich}\:n+1}}{\partial \tilde{\bm C}_{n+1}},
\end{gather}
We may start by briefly discussing the algorithmic components of the equilibrium branch of the model.

\paragraph{Stress and elasticity tensor of the equilibrium part}
Based on the energy form of $G^{\infty}_{\mathrm{ich}}$ given in \eqref{eq:hyper-quadratic-energy}, the isochoric second Piola-Kirchhoff stress can be expressed as
\begin{align*}
\bm S^\infty_\mathrm{ich \: n+1} = J^{-2/3}_{n+1} \mathbb P_{n+1} :  \tilde{\bm S}^\infty_\mathrm{ich \: n+1}
\end{align*}
with
\begin{align*}
\tilde{\bm S}^\infty_\mathrm{ich \: n+1}=\bm T^\infty_{n+1} : \tilde{\mathbb Q}_{n+1}, \quad \bm T^\infty_{n+1} =  \mu^\infty \tilde{\bm E}_{n+1}, \quad \mbox{and} \quad \tilde{\mathbb Q}_{n+1}=2 \frac{\partial \tilde{\bm E}_{n+1}}{\partial \tilde{\bm C}_{n+1}}.
\end{align*}
The fictitious elasticity tensor $\tilde{\mathbb C}^\infty_{\mathrm{ich}\:n+1}$ in \eqref{eq:FLV-elasticity-tensor}$_2$ is explicitly formulated as
\begin{align*}
\tilde{\mathbb C}^\infty_{\mathrm{ich}\:n+1}=2J^{-4/3}_{n+1} \frac{\partial \left(\bm T^\infty_{n+1}:\tilde{\mathbb Q}_{n+1}\right)}{\partial \tilde{\bm C}_{n+1}}=J_{n+1}^{-4/3} \left( \mu^\infty \tilde{\mathbb{Q}}^{\:\mathrm{T}}_{n+1}:\tilde{\mathbb{Q}}_{n+1} + \bm T^\infty_{n+1}:\tilde{\bm{\mathcal L}}_{n+1}\right),
\end{align*}
with
\begin{align*}
\tilde{\bm{\mathcal{L}}}_{n+1} := 2 \frac{\partial \tilde{\mathbb Q}_{n+1}}{\partial \tilde{\bm C}_{n+1}} = 4\frac{\partial^2 \tilde{\bm E}_{n+1}}{\partial \tilde{\bm C}_{n+1} \partial \tilde{\bm C}_{n+1}}.
\end{align*}
The rank-six tensor $\tilde{\bm{\mathcal L}}_{n+1}$ can be calculated based on the eigen-decomposition of $\tilde{\bm C}_{n+1}$ and the scale function of the strain $\tilde{\bm E}_{n+1}$. Its explicit formula is documented in \cite{Miehe2001b} or \cite{Liu2024}. As discussed in Section \ref{sec:hyperelasticity-of-Hill-class}, the strain energy can be extended by adopting multiple quadratic terms, as shown in \eqref{eq:gen-hyper-quadratic-energy}. This leads to the modification of the stress and elasticity tensors accordingly. Readers may refer to \cite[Section~3.2]{Liu2024} for more details.

\paragraph{Stress and elasticity tensor of the non-equilibrium part}
The integration of the constitutive relation is based on developing a recursive formula for the internal variables. According to  \eqref{eq:exp_integration_T}, we have
\begin{align*}
\bm T^{\mathrm{neq}}|_{t_{n+1}} = \exp\left( - \frac{\Delta t_n}{ \tau } \right) \left. \bm T^{\mathrm{neq}}\right|_{t_n} + \int^{t_{n+1}}_{t_n} \exp\left( -\frac{t_{n+1}-s}{\tau} \right) \mu^\mathrm{neq} \frac{d\tilde{\bm E}}{ds} ds.
\end{align*}
The integral in the above can be approximated numerically by a variety of different strategies \cite[p.~353-355]{Simo2006}. Here we may use the value $\exp(-\Delta t_n / 2\tau)$ to approximate the exponential function in the time integral, leading to
\begin{align}
\label{eq:exponential_integration_T}
\left.\bm T^\mathrm{neq} \right|_{t_{n+1}} \approx \bm T^\mathrm{neq}_{n+1}= \exp\left(-\frac{\Delta t_n}{\tau}\right)\bm T^\mathrm{neq}_n+ \exp\left(-\frac{\Delta t_n}{2\tau}\right) \mu^\mathrm{neq} \left(\tilde{\bm E}_{n+1}-\tilde{\bm E}_n\right).
\end{align}
This gives the constitutive integration algorithm. We notice that this formula is one-step, demanding only the information of the previous time instance. The non-equilibrium stress is then defined as
\begin{align*}
\bm S^\mathrm{neq}_\mathrm{ich\:n+1} = J^{-2/3}_{n+1}\mathbb{P}_{n+1}:\tilde{\bm S}^\mathrm{neq}_\mathrm{ich\:n+1} \quad \mbox{with} \quad \tilde{\bm S}^\mathrm{neq}_\mathrm{ich\:n+1}=\bm T^\mathrm{neq}_{n+1}:\tilde{\mathbb{Q}}_{n+1}.
\end{align*}
The fictitious elasticity tensor $\tilde{\mathbb C}^{\mathrm{neq}}_{\mathrm{ich}\:n+1}$ in \eqref{eq:FLV-elasticity-tensor}$_3$ is given by
\begin{align*}
\tilde{\mathbb C}^\mathrm{neq}_{\mathrm{ich}\:n+1}:=2J_{n+1}^{-4/3} \frac{\partial \tilde{\bm S}^\mathrm{neq}_{\mathrm{ich}\:n+1}}{\partial \tilde {\bm C}_{n+1}}=J_{n+1}^{-4/3}\left( \tilde{\mathbb Q}^{T}_{n+1}:\tilde{\mathbb S}_{n+1}+\bm T^\mathrm{neq}_{n+1}:\tilde{\bm{\mathcal L}}_{n+1}\right)
\quad \mbox{with} \quad
\tilde{\mathbb S}_{n+1} := 2\frac{\partial \bm T^\mathrm{neq}_{n+1}}{\partial \tilde{\bm C}_{n+1}}.
\end{align*}
The expression of $\tilde{\mathbb S}_{n+1}$ depends on the constitutive integration algorithm. According to \eqref{eq:exponential_integration_T}, one has
\begin{align*}
\tilde{\mathbb S}_{n+1} = \mu^\mathrm{neq}\exp\left(-\frac{\Delta t_n}{2\tau}\right) \frac{\partial \tilde{\bm E}_{n+1}}{\partial  \tilde{\bm C}_{n+1}}=\mu^\mathrm{neq}\exp\left(-\frac{\Delta t_n}{2\tau}\right)\tilde{\mathbb{Q}}_{n+1}.
\end{align*}
Remarkably, the formula of $\tilde{\mathbb C}^\mathrm{neq}_{\mathrm{ich}\:n+1}$ is identical to that of $\tilde{\mathbb C}^\infty_{\mathrm{ich}\:n+1}$, except that $\mu^{\infty}$ replaced by $\mu^{\mathrm{neq}}\exp(-\Delta t_n/2\tau)$. The energy of the non-equilibrium part can be characterized by multiple relaxation processes as well. Each process evolves independently with distinct parameters and generalized strains, as shown in \eqref{eq:evo-multiple-relax-process}. The modified isochoric second Piola-Kirchhoff stress $\bm S^\mathrm{neq}_{\mathrm{ich}\:n+1}$ and the algorithmic tensor $\tilde{\mathbb C}^\mathrm{neq}_{\mathrm{ich}\:n+1}$ can therefore be represented as the sum of these multiple processes.

\begin{remark}
We may alternatively design the constitutive integration by evaluating \eqref{eq:exp_integration_Ev} at time $t_{n+1}$, leading to
\begin{align*}
\left.\bm E^\mathrm{v}\right|_{t_{n+1}}=\exp\left( -\frac{\Delta t_n}{\tau}\right) \left. \bm E ^\mathrm{v}\right|_{t_n}+\int_{t_n}^{t_{n+1}}  \exp\left(-\frac{t_{n+1}-s}{\tau}\right) \frac{\tilde{\bm E}}{\tau} ds.
\end{align*}
An unconditionally stable and second-order accurate formula can be obtained by approximating $\tilde{\bm E}$ in the integral using its midpoint value, yielding
\begin{align*}
\left.\bm E^\mathrm{v}\right|_{t_{n+1}} \approx \bm E^\mathrm{v}_{n+1}=\exp\left( -\frac{\Delta t_n}{\tau}\right) \bm E ^\mathrm{v}_n+\frac{\left(1-\exp\left(-\Delta t_n/\tau\right)\right)}{2}\left(\tilde{\bm E}_n+\tilde{\bm E}_{n+1}\right).
\end{align*}
In the model developed here, integrating the viscoelastic constitutive relation based on $\bm E^{\mathrm v}$ is completely feasible because our theory merely demands $\bm E^{\mathrm v}$ as a symmetric tensor. Yet, if one introduces $\bm E^{\mathrm v}$ as a non-coercive generalized strain of $\bm \Gamma \in \mathrm{Sym}(3)_{+}$, the eigenvalues of $\bm E^{\mathrm v}$ need to lie within a strict subset of $\mathbb R$. The above integration formula does not warrant the particular eigenstructure of $\bm E^{\mathrm v}$ and may thus fail. This is likely the reason that \cite{Simo1987} chose to integrate the stress-like tensor instead of integrating the internal variables $\bm \Gamma$.
\end{remark}

\section{Beyond the linear evolution equation}
\label{sec:micromechanical-model}
In this section, we extend the model by introducing the \textit{elastic strain} within the framework outlined in Section \ref{sec:theory}. This concept was briefly touched in the discussion of the model based on the inelastic metric in Section \ref{sec:connection-with-existing-models}. The concept of the elastic strain facilitates the construction of the configurational free energy $\Upsilon$ using the general strategies of hyperelasticity modeling, without being confined to a quadratic form. This approach paves the way for deriving the nonlinear viscoelasticity model in a systematic manner.

\subsection{On the notion of the elastic strain}
In Section \ref{sec:quad-energy-model}, we developed a family of models based on Assumption \ref{as:Ev-existence} and a quadratic form of the energies. For the Seth-Hill strain family, the non-coerciveness renders $\bm E - \bm E^{\mathrm v}$ not necessarily representable as a strain. It is thus avoided to interpret $\bm E - \bm E^{\mathrm v}$ as the additivity of elastic and inelastic strains. Despite the fact that several inelasticity theories have been established based on the Green-Naghdi kinematic assumption, the majority of them describe the elastic response by a quadratic energy function, closely related to the class of Hill's hyperelasticity. An exception is the model by \cite{Miehe2000}, in which the authors designed the elastic response using the Ogden model. The Hencky strain adopted in \cite{Miehe2000} allows the authors to retrieve the elastic stretches in the inelastic branches. It thus becomes evident that the coerciveness \eqref{eq:E_coerciveness} is not only physically sound but also mathematically relevant. It ensures that any symmetric tensor can be represented in the form of a generalized strain. This observation is stated in the following proposition.
\begin{proposition}
\label{prop:Sym3-spectral-rep}
Given $\bm W \in$ Sym(3) and a coercive scale function $E$, there exist $w_a \in \mathbb R_{+}$ and orthonormal eigenvectors $\hat{\bm N}_a$ for $a=1,2,3$, such that
\begin{align*}
\bm W = \sum_{a=1}^{3} E(w_a) \hat{\bm N}_a \otimes \hat{\bm N}_a.
\end{align*}
\end{proposition}
\begin{proof}
As a symmetric tensor, $\bm W$ enjoys a spectral decomposition 
\begin{align*}
\bm W = \sum_{a=1}^{3} W_a \hat{\bm N}_a \otimes \hat{\bm N}_a,
\end{align*}
where $\lbrace \hat{\bm N}_a \rbrace$ are mutually orthonormal eigenvectors, and $W_a \in \mathbb R$ are the eigenvalues. The monotonicity \eqref{eq:E_property}$_1$ ensures the inverse of the scale function $E^{-1}$ exists, while the coerciveness \eqref{eq:E_coerciveness} implies the domain of the inverse function is $\mathbb R$. Therefore, the value of $w_a$ can be uniquely determined as $E^{-1}(W_a)$. 
\end{proof}

The above proposition implies that $\bm E - \bm E^{\mathrm v}$, as a symmetric tensor, can be expressed as a coercive generalized strain denoted by $\bm E^{\mathrm e}$, which is intuitively interpreted as the elastic strain. For convenience, it is natural to require that the scale function of $\bm E^{\mathrm e}$ is of the same type as that of the strain $\bm E$. We formalized this by stating the following kinematic assumption in addition to Assumption \ref{as:Ev-existence}.

\begin{assumption}
\label{as:Ee-existence}
Given a coercive strain $\bm E$, the elastic strain $\bm E^{\mathrm e} := \bm E - \bm E^{\mathrm v}$ is defined by the same coercive scale function used in the definition of $\bm E$.  
\end{assumption}

The coerciveness enables the introduction of the elastic strain, and Assumption \ref{as:Ee-existence} leads to the additive split of the strain $\bm E = \bm E^{\mathrm e} + \bm E^{\mathrm v}$. We have to cautiously mention that the definition of $\bm E^{\mathrm e}$ is indeed local, as there does not necessarily exist a continuous global deformation that generates it. The same situation arises in the multiplicative decomposition and even in the theory of infinitesimal strain inelasticity. Nevertheless, we still refer to it as a `strain' with the underlying issue recognized. Following Proposition \ref{prop:Sym3-spectral-rep}, we have the elastic stretches $\{ \lambda^{\mathrm e}_a \}$ and elastic principal directions $\{ \bm N_a^{\mathrm e} \}$ defined, such that
\begin{align*}
\bm E^{\mathrm e} = \sum_{a=1}^{3} E(\lambda^{\mathrm e}_a) \bm N_a^{\mathrm e} \otimes \bm N_a^{\mathrm e}.
\end{align*}
With this eigenstructure, a positive semi-definite tensor
\begin{align}
\label{eq:def-GN-C-e}
\bm C^{\mathrm e} := \sum_{a=1}^{3} \lambda^{\mathrm e \: 2}_a \bm N_a^{\mathrm e} \otimes \bm N_a^{\mathrm e},
\end{align}
is introduced and is termed as the \textit{elastic deformation tensor}.  With the introduction of $\bm E^{\mathrm e}$ and $\bm C^{\mathrm e}$, we may construct the Gibbs free energy \eqref{eq:Gibbs-a-b} as
\begin{align}
\label{eq:Gibbs-vol-a-Ce}
G(\bm C, P, \bm E^{\mathrm v}) = G_{\mathrm{vol}}(P) + G^{\infty}_{\mathrm a}(\bm C) + \Upsilon(\bm C^{\mathrm e}).
\end{align}
The configurational free energy $\Upsilon(\bm C^{\mathrm e})$, interpreted as the energy stored in the Maxwell element, is consistent with Assumption \ref{as:Ev-existence}, as $\bm C^{\mathrm e}$ is in fact a tensorial function depending on $\bm E(\bm C) - \bm E^{\mathrm v}$. Therefore, the constitutive theory developed in Section \ref{sec:theory} remain valid. Due to the additional structure induced by Assumption \ref{as:Ee-existence}, we can present a more refined representation of the constitutive relations. Recalling \eqref{eq:Gibbs_a_vol_constitutive_relation}, the stress is given by $\bm S = \bm S_{\mathrm{a}} + \bm S_{\mathrm{vol}}  =  \bm S_{\mathrm{a}}^{\infty} + \bm S_{\mathrm{a}}^{\mathrm{neq}} + \bm S_{\mathrm{vol}}$, with
\begin{align*}
\bm S^{\infty}_{\mathrm a} = 2 \frac{\partial G^{\infty}_{\mathrm a}}{\partial \bm C}, \quad
\bm S^{\mathrm{neq}} = 2 \frac{\partial \Upsilon}{\partial \bm C} = \bm T^{\mathrm{neq}} : \mathbb Q, \quad \bm T^{\mathrm{neq}} = -\frac{\partial \Upsilon}{\partial \bm E^{\mathrm v}}, \quad \mbox{and} \quad \bm S_{\mathrm{vol}} = -J P\bm C^{-1}.
\end{align*}
The evolution equation follows \eqref{eq:evolution_equation} and is restated here as
\begin{align*}
\dot{\bm E}^{\mathrm v} = \mathbb V^{-1} : \bm T^{\mathrm{neq}}.
\end{align*}
Regarding the configurational free energy constructed as a function of $\bm C^{\mathrm e}$, an explicit form of $\bm T^{\mathrm{neq}}$ can be derived as
\begin{align}
\label{eq:T_neq_explicit_form}
\bm T^{\mathrm{neq}} = - \frac{\partial \Upsilon}{\partial \bm E^{\mathrm v}} = - \frac{\partial \Upsilon}{\partial \bm C^{\mathrm e}} : \frac{\partial \bm C^{\mathrm e}}{\partial \bm E^{\mathrm e}} : \frac{\partial \bm E^{\mathrm e}}{\partial \bm E^{\mathrm v}} = \bm S^{\mathrm e} : \mathbb Q^{\mathrm e \: -1}, \quad \mbox{with} \quad \bm S^{\mathrm e} := 2\frac{\partial \Upsilon}{\partial \bm C^{\mathrm e}} \quad \mbox{and} \quad \mathbb Q^{\mathrm e\:-1} := \frac12 \frac{\partial \bm E^{\mathrm e}}{\partial \bm C^{\mathrm e}}.
\end{align}
From the above, we conclude that the constitutive relations and evolution equations depend on the form of $\bm S^{\mathrm e}$ essentially. The benefit of introducing the elastic deformation tensor is that we may exploit the existing phenomenological or micromechanical models to characterize the elastic response in the non-equilibrium branch. In addition to the Hencky strain, the coercive strains listed in Table \ref{table:list_of_gen_strains} offer more flexibility in the separation of the inelastic behavior from the total deformation. Before delving into the micromechanical models, we first consider the two general options for phenomenological modeling by examining the explicit forms of $\bm S^{\mathrm e}$.

\paragraph{Invariant-based models}
One general modeling approach constructs the energy as the polynomial combinations of the principal invariants of the deformation tensor. Sometimes, more complex forms, such as logarithmic or exponential functions, are introduced to enhance the model. We consider the configurational free energy $\Upsilon$ constructed in terms of the principal invariants of $\bm C^{\mathrm e}$, that is,
\begin{align*}
\Upsilon(\bm C^{\mathrm e}) = \Upsilon( I_1^{\mathrm e}, I_2^{\mathrm e}, I_3^{\mathrm e}),
\end{align*}
where the invariants are given by
\begin{align*}
I_1^{\mathrm e} := \bm C^{\mathrm e} : \bm I, \quad I_2^{\mathrm e} := \frac12 \left( I_1^{\mathrm e \: 2} - \bm C^{\mathrm e \: 2} : \bm I \right), \quad I_3^{\mathrm e} := \mathrm{det}\bm C^{\mathrm e}.
\end{align*}
With this form of $\Upsilon$, the tensor $\bm S^{\mathrm e}$ can be expressed as
\begin{align*}
\bm S^{\mathrm e} := 2\frac{\partial \Upsilon}{\partial \bm C^{\mathrm e}} = 2 \left( \frac{\partial \Upsilon}{\partial I_1^{\mathrm e}} + I_1^{\mathrm e} \frac{\partial \Upsilon}{\partial I_2^{\mathrm e}} \right) \bm I - 2\frac{\partial \Upsilon}{\partial I^{\mathrm e}_2} \bm C^{\mathrm e} + 2 I^{\mathrm e}_3 \frac{\partial \Upsilon}{\partial I^{\mathrm e}_3} \bm C^{\mathrm e \: -1}.
\end{align*}
An appealing aspect of the invariant-based approach is that it allows for generalization to anisotropic material responses, with the aid of structural tensors (\cite{Holzapfel2000}). However, the treatment of anisotropic materials is beyond the scope of this work.

\paragraph{Principal stretch-based models}
Instead of using the principal invariants, the isotropic hyperelastic material response can also be conveniently represented in terms of principal stretches. The Valanis-Landel hypothesis enables a more concise form of the configurational free energy, i.e.,
\begin{align*}
\Upsilon(\bm C^{\mathrm e}) = \sum_{a=1}^{3} \varpi(\lambda_a^{\mathrm e}),
\end{align*}
where the response in each direction is characterized by an identical scalar function $\varpi$. With the particular form of the energy, we have
\begin{align*}
\bm S^{\mathrm e} = \sum_{a=1}^{3} \frac{1}{\lambda^{\mathrm e}_a} \varpi'(\lambda^{\mathrm e}_{a}) \bm N_a^{\mathrm e} \otimes \bm N_a^{\mathrm e}.
\end{align*}
Since the elastic deformation tensor $\bm C^{\mathrm e}$ is defined based on the elastic stretches and elastic principal directions, the calculation of $\bm S^{\mathrm e}$ is straightforward in practice.

\begin{remark}
Assumption \ref{as:Ee-existence} assumes the same scale function is used for defining both the strain $\bm E$ and the elastic strain $\bm E^{\mathrm e}$. It is possible to further relax this condition by allowing these two strains to be defined with different scale functions.
\end{remark}

\begin{remark}
Regarding the theory based on the isochoric-volumetric split of the energy form \eqref{eq:Gibbs-energy-additive-form}, we may introduce $\tilde{\bm E}^{\mathrm e} := \tilde{\bm E} - \bm E^{\mathrm v}$, and the energy can be constructed as $G(\bm C, P, \bm E^{\mathrm v}) = G_{\mathrm{vol}}(P) + G^{\infty}_{\mathrm{ich}}(\tilde{\bm E}(\bm C)) + \Upsilon(\tilde{\bm C}^{\mathrm e})$. The resulting constitutive relations follow the discussion made in Section \ref{sec:treatment-of-compressibility}.
\end{remark}

\subsection{A micromechanically inspired model}
\label{sec:a_micro_inspired_model}
The strength of introducing the elastic strain lies in its ability to leverage general models for the elastic response. Micromechanical models, rooted in the microstructure of materials, often demonstrate exceptional accuracy in capturing the stress-strain response for rubber-like materials. In this work, we utilize the eight-chain model for our nonlinear viscoelasticity model (\cite{Arruda1993,Bischoff2001}). It is worth noting that the model of \cite{Arruda1993} has been adopted for nonlinear viscoelasticity based on the multiplicative decomposition $\bm F = \bar{\bm F}^{\mathrm e} \bar{\bm F}^{\mathrm v}$ by \cite{Bergstroem1998,Dal2020}, and others.

In this section, we propose a micromechanical model based on the Gibbs free energy in the form of \eqref{eq:Gibbs-vol-a-Ce}, with the contributions of the equilibrium and non-equilibrium parts given by
\begin{align*}
G^{\infty}_{\mathrm{a}}(\bm C) = \mu^\infty N^\infty \left( \lambda^{\infty} \mathfrak L^{-1}(\lambda^{\infty}) + \ln \frac{\mathfrak L^{-1}(\lambda^{\infty})}{\sinh(\mathfrak L^{-1}(\lambda^{\infty}))} \right) - \frac{\mu^\infty\sqrt{N^\infty}}{3} \mathfrak L^{-1}(1/\sqrt{N^\infty}) \ln(J),
\end{align*}
and
\begin{align*}
\Upsilon(\bm C^{\mathrm e}) = \mu^\mathrm{neq} N^\mathrm{neq} \left( \lambda^{\mathrm{neq}} \mathfrak L^{-1}(\lambda^{\mathrm{neq}}) + \ln \frac{\mathfrak L^{-1}(\lambda^{\mathrm{neq}})}{\sinh(\mathfrak L^{-1}(\lambda^{\mathrm{neq}}))} \right) - \frac{\mu^\mathrm{neq} \sqrt{N^\mathrm{neq}}}{3} \mathfrak L^{-1}(1/\sqrt{N^\mathrm{neq}}) \ln(J^{\mathrm e}),
\end{align*}
respectively. Here, $\mu^{\infty}$ ($\mu^{\mathrm{neq}}$) and $N^{\infty}$ ($N^{\mathrm{neq}}$) represent the shear modulus and the number of molecular chain segments for the (non-)equilibrium part, respectively; we use $\mathfrak L^{-1}$ to denote the inverse of the Langevin function $\mathfrak L(x) := \coth x - 1/x$, which accounts for the limited chain extensibility; the relative average network stretches for the equilibrium part $\lambda^{\infty}$ and the non-equilibrium part $\lambda^\mathrm{neq}$ are given by
\begin{align*}
\lambda^{\infty} := \sqrt{\frac{I_1}{3N^\infty}}, \quad \text{and} \quad\lambda^{\mathrm{neq}} := \sqrt{\frac{I^{\mathrm e}_1}{3N^\mathrm{neq}}},
\end{align*}
respectively; the principal invariant $I_1$ defined as $\bm C : \bm I$, and $J^{\mathrm e} := \lambda^{\mathrm e}_1\lambda^{\mathrm e}_2\lambda^{\mathrm e}_3$. The forms of the energies are based on the statistical mechanics of polymer networks and a mapping between the macroscopic deformation and the microscopic network stretch. In the meantime, they can also be viewed as energy functions written in terms of the first principal invariants $I_1$ and $I_1^{\mathrm e}$. Recalling \eqref{eq:Gibbs_a_vol_constitutive_relation}, the second Piola-Kirchhoff stress is given by:
\begin{align*}
\bm S = \bm S_{\mathrm a} + \bm S_{\mathrm{vol}}, \quad \mbox{with} \quad \bm S_{\mathrm a} = \bm S_\mathrm{a}^\infty + \bm S_\mathrm{a}^\mathrm{neq} \quad \mbox{and} \quad \bm S_{\mathrm{vol}} = -J P \bm C^{-1}.
\end{align*}
The equilibrium part of the stress is expressed as
\begin{align}
\label{eq:michro-S-infty-a}
\bm S^\infty_\mathrm{a}=2\frac{\partial G^\infty_\mathrm{a}}{\partial \bm C}=2\frac{\partial G^\infty_\mathrm{a}}{\partial \lambda^{\infty}}\frac{\partial \lambda^{\infty}}{\partial \bm C}=\frac{\mu^\infty}{3\lambda^{\infty}}\mathfrak{L}^{-1}(\lambda^{\infty}) \bm I - \frac{\mu^\infty \sqrt{N^\infty}}{3}\mathfrak{L}^{-1}( 1/\sqrt{N^\infty} ) \bm C^{-1}.
\end{align}
The non-equilibrium part of the stress is determined from $\bm S^{\mathrm{neq}}_{\mathrm{a}} = \bm T^{\mathrm{neq}} : \mathbb{Q}$, with the form of $\bm T^{\mathrm{neq}}$ given by \eqref{eq:T_neq_explicit_form}, that is,
\begin{align*}
\bm T^\mathrm{neq} = \bm S^\mathrm{e} : {\mathbb{Q}^\mathrm{e}}^{-1}.
\end{align*}
The form of $\bm S^\mathrm{e}$ is similar to that of $\bm S^\infty_\mathrm{a}$, which is given by
\begin{align*}
\bm S^\mathrm{e} = 2\frac{\partial \Upsilon}{\partial \bm C^\mathrm{e}} = 2\frac{\partial \Upsilon}{\partial \lambda^\mathrm{neq}}\frac{\partial \lambda^\mathrm{neq}}{\partial \bm C^\mathrm{e}} = \frac{\mu^\mathrm{neq}}{3\lambda^\mathrm{neq}}\mathfrak{L}^{-1}(\lambda^\mathrm{neq}) \bm I - \frac{\mu^\mathrm{neq} \sqrt{N^\mathrm{neq}}}{3}\mathfrak{L}^{-1} (1/\sqrt{N^\mathrm{neq}} ) {\bm C^\mathrm{e}}^{-1}.
\end{align*}
Recalling the evolution equation \eqref{eq:evolution_equation}, the evolution equation for the micromechanical model can be expressed as
\begin{align}
\label{eq:micro_evolution_equation}
\eta \dot{ \bm E}^{\mathrm{v}} = \bm T^{\mathrm{neq}} = \bm S^\mathrm{e} : {\mathbb{Q}^\mathrm{e}}^{-1}.
\end{align}
Apparently, due to the nonlinear dependency of $\bm T^{\mathrm{neq}}$ on the internal variable $\bm E^{\mathrm v}$, the above evolution equation demands a local iterative process for the constitutive integration.

\subsection{Constitutive integration}
In this section, we discuss the integration algorithms of the proposed micromechanical model for finite element analysis. Given the nonlinear nature of the evolution equation, the overall algorithm involves the integration of the constitutive equations as well as the integration of the global balance equations. We begin by considering the integration scheme for the evolution equation \eqref{eq:micro_evolution_equation}. The computation of the elasticity tensor, which is essential for integrating the momentum balance equation, will be detailed subsequently. 

\paragraph{Integration of the evolution equation}
The nonlinear evolution equation \eqref{eq:micro_evolution_equation} is discretized through the mid-point scheme
\begin{align}
\label{eq:mid-point}
\bm E^\mathrm{v}_{n+1} = \bm E^\mathrm{v}_n + \frac{\Delta t_n}{\eta}\bm S^\mathrm{e}_{n+\frac12} :  \mathbb Q^{\mathrm{e}\:-1}_{n+\frac{1}{2}},
\end{align}
where
\begin{gather*}
\bm S^{\mathrm e}_{n+\frac12} = \frac{\mu^\mathrm{neq}}{3\lambda^{\infty}_{n+\frac12}}\mathfrak{L}^{-1}(\lambda^{\infty}_{n+\frac12}) \bm I - \frac{\mu^\mathrm{neq} \sqrt{N^\mathrm{neq}}}{3}\mathfrak{L}^{-1} (1/\sqrt{N^\mathrm{neq}} ) \bm C^\mathrm{e \: -1}_{n+\frac12}, \qquad
\mathbb Q^{\mathrm e \: -1}_{n+\frac12} = \frac12 \frac{\partial \bm C^{\mathrm e}_{n+\frac12}}{\partial \bm E^{\mathrm e}_{n+\frac12}}, \displaybreak[2] \\
\bm E^{\mathrm e}_{n+1} = \bm E_{n+1} - \bm E^{\mathrm v}_{n+1}, \quad \bm E^{\mathrm e}_{n} = \bm E_{n} - \bm E^{\mathrm v}_{n}, \displaybreak[2] \\
\bm E^{\mathrm e}_{n} = \sum_{a=1}^{3} E(\lambda^{\mathrm e}_{a \: n}) \bm N^{\mathrm e}_{a \: n} \otimes \bm N^{\mathrm e}_{a \: n}, \quad \bm E^{\mathrm e}_{n+1} = \sum_{a=1}^{3} E(\lambda^{\mathrm e}_{a \: n+1}) \bm N^{\mathrm e}_{a \: n+1} \otimes \bm N^{\mathrm e}_{a \: n+1}, \displaybreak[2] \\
\bm C^{\mathrm e}_{n} = \sum_{a=1}^{3} \lambda^{\mathrm e \: 2}_{a \: n} \bm N^{\mathrm e}_{a \: n} \otimes \bm N^{\mathrm e}_{a \: n}, \quad \bm C^{\mathrm e}_{n+1} = \sum_{a=1}^{3} \lambda^{\mathrm e \: 2}_{a \: n+1} \bm N^{\mathrm e}_{a \: n+1} \otimes \bm N^{\mathrm e}_{a \: n+1}, \displaybreak[2] \\
\bm C^{\mathrm e}_{n+\frac12} = \frac12 \left( \bm C^{\mathrm e}_{n} + \bm C^{\mathrm e}_{n+1} \right), \quad \lambda^{\mathrm{neq}}_{n+\frac12} = \left(\bm C^{\mathrm e}_{n+\frac12} : \bm I  / 3 N^{\mathrm{neq}}\right)^{\frac12},
\end{gather*}
and $\Delta t_n = t_{n+1}-t_n$. Given the deformation state characterized by $\bm E_{n+1}$, the internal variable $\bm E^{\mathrm v}_{n+1}$ is determined by solving the equation \eqref{eq:mid-point} through a Newton-Raphson iterative procedure at each quadrature point. The residual of the local Newton-Raphson iteration is defined as
\begin{align*}
\bm R_{n+1} := \bm E^\mathrm{v}_{n+1} - \bm E^\mathrm{v}_n - \frac{\Delta t_n}{\eta} \bm S^\mathrm{e}_{n+\frac12}:\mathbb Q_{n+\frac12}^{\mathrm{e}\:-1}.
\end{align*}
The linearization of $\bm R_{n+1}$ with respect to $\bm E^{\mathrm v}_{n+1}$ can be performed as
\begin{align}
\mathbb K_{n+1} := \frac{\partial \bm R_{n+1}}{\partial \bm E^\mathrm{v}_{n+1}} &=\mathbb I - \frac{\Delta t_n}{\eta} \Big(
\mathbb Q_{n+\frac12}^{\mathrm{e}\:-\mathrm{T}}:\frac{\partial \bm S^\mathrm{e}_{n+\frac12}}{\partial \bm E^\mathrm{v}_{n+1}}
+ \bm S^\mathrm{e}_{n+\frac12} : \frac{\partial \mathbb Q_{n+\frac12}^{\mathrm{e}\:-1}}{\partial \bm E^\mathrm{v}_{n+1}}
\Big) \displaybreak[2] \nonumber \\
\label{eq:mid-point-KK-def}
&=\mathbb I +  \frac{\Delta t_n}{\eta} \Big( \mathbb Q_{n+\frac12}^{\mathrm{e}\:\mathrm{-T}}:\frac12 \mathbb C^\mathrm{e}_{n+\frac12} : \mathbb Q_{n+1}^{\mathrm{e}\:-1} + \bm S^\mathrm{e}_{n+\frac12}: \bm{\mathcal K}^{\mathrm{e}}_{n+\frac12}:\mathbb Q^\mathrm{e}_{n+\frac12} : \mathbb Q ^{\mathrm{e}\:-1}_{n+1} \Big),
\end{align}
where
\begin{align}
\label{eq:micro-model-C-e-mid}
\mathbb C^\mathrm{e}_{n+\frac12}:= 2\frac{\partial \bm S^\mathrm{e}_{n+\frac12}}{\partial \bm C^\mathrm{e}_{n+\frac12}}
= \frac{\mu^\mathrm{neq}}{9N^\mathrm{neq}\lambda^\mathrm{neq}_{n+\frac12}}\frac{\partial \big(\mathfrak{L}^{-1}(\lambda_{n+\frac12}^\mathrm{neq})/\lambda_{n+\frac12}^\mathrm{neq}\big)}{\partial \lambda^\mathrm{neq}_{n+\frac12}} \bm I \otimes \bm I + \frac{2\mu^\mathrm{neq}\sqrt{N^\mathrm{neq}}}{3}\mathfrak{L}^{-1}(1/\sqrt{N^\mathrm{neq}}) \bm C_{n+\frac12}^{\mathrm{e}\:-1} \odot \bm C_{n+\frac12}^{\mathrm{e}\:-1}
\end{align}
and
\begin{align*}
\bm{\mathcal K}^\mathrm{e}_{n+\frac12}:= \frac{\partial \mathbb Q^{\mathrm{e}\:-1}_{n+\frac12}}{2\partial \bm E^\mathrm{e}_{n+\frac12}}= \frac{\partial^2 \bm C_{n+\frac12}^\mathrm{e}}{4\partial \bm E_{n+\frac12}^\mathrm{e}\partial \bm E^\mathrm{e}_{n+\frac12}}.
\end{align*}
The calculation of the rank-six tensor $\bm{\mathcal K}^{\mathrm e}_{n+1/2}$ is based on the eigendecomposition of $\bm C^{\mathrm e}_{n+1/2}$, and we give its explicit formula in Appendix \ref{appendix:mathcal-K}. We determine the internal variable at quadrature points by solving the nonlinear equation $\bm R_{n+1} = \bm O$. If we use the subscript $(i)$ to denote the quantities at the $i$-th Newton-Raphson iteration, the local iterative procedure can be summarized as follows.

\begin{myenv}{Local Newton-Raphson iteration}
\noindent \textbf{Predictor stage}: Set $\bm{E}_{n+1\:(0)}^{\mathrm v} = \bm{E}_{n}^{\mathrm v}$.
 
\noindent \textbf{Multi-Corrector stage}: Repeat the following steps for $i=0, \cdots, i_{\mathrm{max}}$.
\begin{enumerate}
\item Construct the local residual $\bm R_{n+1 \: (i)}$ with $\bm E^\mathrm{v}_{n+1 \: (i)}$. If one of the stopping criteria
\begin{align*}
\frac{\left \vert \bm{R}_{n+1\:(i)} \right \vert}{\left \vert \bm{R}_{n+1\:(0)} \right \vert} \leq \mathrm{tol}_{\mathrm{r}} \quad \mbox{and} \quad \left \vert \bm{R}_{n+1\:(i)} \right \vert \leq \mathrm{tol}_{\mathrm{a}}
\end{align*} 
is satisfied for two prescribed tolerances $\mathrm{tol}_{\mathrm r}$ and $\mathrm{tol}_{\mathrm a}$, set $\bm{E}_{n+1} = \bm{E}_{n+1\:(i)}$ and exit the multi-corrector stage. Otherwise, continue to Step 2.
\item Perform linearization of $\bm{R}_{n+1\:(i)}$ with respect to $\bm E^\mathrm{v}_{n+1 \: (i)}$ by constructing $\mathbb K_{n+1 \: (i)}$ using the internal variable $\bm E^\mathrm{v}_{n+1 \: (i)}$.
\item Solve the incremental solution
\begin{align}
\label{eq:local-NR-linear-tensor-eqn}
\mathbb K_{n+1\: (i)} : \Delta \bm{E}^{\mathrm v}_{n+1\:(i)} = -\bm{R}_{n+1\:(i)}.
\end{align}
\item Update the discrete internal state variable as $\bm{E}^{\mathrm v}_{n+1\:(i+1)} = \bm{E}^{\mathrm v}_{ n+1\:(i)} + \Delta \bm{E}^{\mathrm v}_{n+1\:(i)}$.
\end{enumerate}
\end{myenv}

\begin{remark}
We employ the mid-point rule for the evolution equation to achieve second-order temporal accuracy in the overall scheme. It can be demonstrated that the rank-four tensor $\mathbb K$, as given in \eqref{eq:mid-point-KK-def}, possesses minor symmetries. When the time step is sufficiently small, $\mathbb K$ is close to the identity tensor. Therefore, the equation \eqref{eq:local-NR-linear-tensor-eqn} is expected to be uniquely solvable at least for small time step sizes. Moreover, if the backward Euler scheme is applied to the evolution equation, the linearization of the local residual possesses both minor and major symmetries, which renders an even better-structured incremental equation in the local Newton-Raphson iteration.
\end{remark}

\paragraph{Integration of the linear momentum balance equations}
The integration of the global balance equation is also performed with the Newton-Raphson iterative procedure, which necessitates the specification of the elasticity tensor. The elasticity tensor due to $\bm S_{\mathrm a} = \bm S_\mathrm{a}^\infty + \bm S_\mathrm{a}^\mathrm{neq}$ at time $t_{n+1}$ is given by
\begin{align*}
\mathbb C_\mathrm{a \: n+1}:=2\frac{\partial \bm S_{\mathrm{a}\:n+1}}{\partial \bm C_{n+1}}=\mathbb C_{\mathrm{a}\: n+1}^\infty + \mathbb C_{\mathrm{a} \: n+1}^\mathrm{neq}
\quad \mbox{with} \quad
\mathbb C^\infty_{\mathrm{a} \: n+1} := 2\frac{\partial \bm S^\infty_{\mathrm{a} \: n+1}}{\partial \bm C_{n+1}}
\quad \mbox{and} \quad
\mathbb C^\mathrm{neq}_{\mathrm{a} \: n+1}:=2\frac{\partial \bm S^\mathrm{neq}_{\mathrm{a} \: n+1}}{\partial \bm C_{n+1}}.
\end{align*}
Based on the stress representation \eqref{eq:michro-S-infty-a}, the elasticity tensor of the equilibrium part can be derived directly as
\begin{align*}
\mathbb{C}_{\mathrm{a}\:n+1}^\infty = \frac{\mu^\infty}{9N^\infty\lambda^{\infty}_{n+1}}\frac{\partial \left(\mathfrak{L}^{-1}(\lambda^{\infty}_{n+1})/\lambda^{\infty}_{n+1}\right)}{\partial \lambda^{\infty}_{n+1}} \bm I \otimes \bm I + \frac{2\mu^\infty\sqrt{N^\infty}}{3}\mathfrak{L}^{-1}(1/\sqrt{N^\infty}) \bm C_{n+1}^{-1} \odot \bm C_{n+1}^{-1},
\end{align*}
where 
\begin{align*}
\lambda^{\infty}_{n+1} := \left( \bm C_{n+1} : \bm I / 3 N^{\infty} \right)^{\frac12}. 
\end{align*}
The inverse Langevin function $\mathfrak L^{-1}$ does not have an analytical formula, and in practice it is evaluated using approximated functions. In this work, the inverse Langevin function is approximated by the following rational function,
\begin{align*}
\mathfrak{L}^{-1}(x) \approx x\frac{15-\left(6x^2+x^4-2x^6\right)}{5\left(1-x^2\right)},
\end{align*}
which achieves a balance between accuracy and complexity (\cite{Kroeger2015}). The elasticity tensor of the non-equilibrium part is given by
\begin{align}
\label{eq:micro-model-CC-neq}
\mathbb{C}_{\mathrm{a}\:n+1}^\mathrm{neq} = 2\frac{\partial \left(\bm T_{n+1}^\mathrm{neq}:\mathbb Q_{n+1}\right)}{\partial \bm C_{n+1}} =  \mathbb Q_{n+1}^\mathrm{T}: \mathbb S_{n+1} + \bm T_{n+1}^\mathrm{neq} : \bm{\mathcal L}_{n+1},
\end{align}
with
\begin{align*}
\bm{\mathcal L}_{n+1} &:= 2\frac{\partial \mathbb Q_{n+1}}{\partial \bm C_{n+1}}=4\frac{\partial^2 \bm E_{n+1}}{\partial \bm C_{n+1} \partial \bm C_{n+1}}, \\
\mathbb S_{n+1} &:= 2\frac{\partial \bm T_{n+1}^\mathrm{neq}}{\partial \bm C_{n+1}} = 2\frac{\partial \left(\bm S_{n+1}^\mathrm{e}:\mathbb Q_{n+1}^{\mathrm{e}\:-1}\right)}{\partial \bm C_{n+1}}=\mathbb Q^{\mathrm{e}\:-\mathrm{T}}_{n+1}:2\frac{\partial \bm S^\mathrm{e}_{n+1}}{\partial \bm C_{n+1}} + \bm S_{n+1}^\mathrm{e} : 2\frac{\partial \mathbb Q^{\mathrm{e}\:-1}_{n+1}}{\partial \bm C_{n+1}}.
\end{align*}
The rank-six tensor $\bm{\mathcal L}_{n+1}$ can be calculated with the aid of the eigen-decomposition of $\bm C_{n+1}$. To derive a more explicit formula for $\mathbb S_{n+1}$, we introduce $\mathbb C^\mathrm{e}_{n+1}:=2\partial \bm S_{n+1}^\mathrm{e} / \partial \bm C_{n+1}^\mathrm{e}$, whose explicit expression is similar to \eqref{eq:micro-model-C-e-mid}, except that the network stretch is evaluated at $\lambda^{\mathrm{neq}}_{n+1}$. Consequently, we have
\begin{align}
\label{eq:micro-model-SS_1}
2\frac{\partial \bm S_{n+1}^\mathrm{e}}{\partial \bm C_{n+1}}=2\frac{\partial \bm S_{n+1}^\mathrm{e}}{\partial \bm C_{n+1}^\mathrm{e}}:\frac{\partial \bm C_{n+1}^\mathrm{e}}{2\partial \bm E_{n+1}^\mathrm{e}}:2\frac{\partial \bm E_{n+1}^\mathrm{e}}{\partial \bm C_{n+1}}=\mathbb C_{n+1}^\mathrm{e}:\mathbb{Q}_{n+1}^{\mathrm{e}\:-1}:\left(\mathbb Q_{n+1} - \mathbb H_{n+1} \right)
\end{align}
and
\begin{align}
\label{eq:micro-model-SS_2}
2\frac{\partial \mathbb Q_{n+1}^{\mathrm{e}\:-1}}{\partial \bm C_{n+1}}=\frac{\partial \mathbb Q^{\mathrm{e}\:-1}_{n+1}}{\partial \bm E_{n+1}^\mathrm{e}}:2\frac{\partial \bm E^\mathrm{e}_{n+1}}{\partial \bm C_{n+1}}=2 \bm{\mathcal K}_{n+1}^{\mathrm{e}}:\left(\mathbb Q_{n+1} - \mathbb H_{n+1} \right).
\end{align}
where $\mathbb H_{n+1} := 2\partial \bm E^\mathrm{v}_{n+1}/\partial \bm C_{n+1}$ is determined based on the constitutive integration scheme. In our proposed method, the discrete evolution equation \eqref{eq:mid-point} leads to
\begin{align*}
\mathbb H_{n+1} &:= 2\frac{\partial \bm E^\mathrm{v}_{n+1}}{\partial \bm C_{n+1}}=   \frac{2\Delta t_n}{\eta}\frac{\partial \left(\bm S^\mathrm{e}_{n+\frac12}:\mathbb Q^{\mathrm{e}\:-1}_{n+\frac12}  \right)}{\partial \bm C_{n+1}}\\
&=\frac{\Delta t_n}{\eta} \left(
\mathbb Q^\mathrm{e\:-T}_{n+\frac12}:\frac12\mathbb C^\mathrm{e}_{n+\frac12}:\mathbb Q^{\mathrm{e}\:-1}_{n+1}+\bm S^\mathrm{e}_{n+\frac12}: \bm{\mathcal K}^{\mathrm{e}}_{n+\frac12}:\mathbb Q^\mathrm{e}_{n+\frac12}:\mathbb Q^{\mathrm{e}\:-1}_{n+1} \right)
: \left(\mathbb Q_{n+1}-\mathbb H_{n+1}\right) \\
&= \left( \mathbb K_{n+1} - \mathbb I \right) : \left( \mathbb Q_{n+1}-\mathbb H_{n+1} \right).
\end{align*}
Re-organizing the above relation leads to
\begin{align*}
\mathbb K_{n+1} : \mathbb H_{n+1} = \left( \mathbb K_{n+1} - \mathbb I \right) :  \mathbb Q_{n+1}.
\end{align*}
The rank-four tensor $\mathbb H_{n+1}$ is determined through solving the above equation. Interestingly, the formulation of the above equation is based on $\mathbb K_{n+1}$, which has already been obtained in \eqref{eq:mid-point-KK-def}. With $\mathbb H_{n+1}$ determined, the formula of $\mathbb C^\mathrm{neq}_{\mathrm{a}\:n+1}$ can be re-organized as
\begin{align*}
\mathbb C^\mathrm{neq}_{\mathrm{a}\:n+1}&= \mathbb Q_{n+1}^\mathrm{T}: \left(
\mathbb Q^\mathrm{e\:-T}_{n+1}:\mathbb C^\mathrm{e}_{n+1}:\mathbb Q^{\mathrm{e}\:-1}_{n+1}+2\bm S^\mathrm{e}_{n+1}: \bm{\mathcal K}^{\mathrm{e}}_{n+1} \right)
: \left(\mathbb Q_{n+1}-\mathbb H_{n+1}\right) + \bm T_{n+1}^\mathrm{neq}: \bm{\mathcal L}_{n+1}
\end{align*}
by substituting \eqref{eq:micro-model-SS_1} and \eqref{eq:micro-model-SS_2} into the expression of $\mathbb S_{n+1}$ in \eqref{eq:micro-model-CC-neq}. With the above expression of $\mathbb C^\mathrm{neq}_{\mathrm{a}\:n+1}$, we complete the derivation of the elasticity tensor to be used in the solution procedure of the linear momentum balance equation.

\section{Results}
\label{sec:results}
This section presents the results of model calibration and finite element analysis to examine, validate, and compare the constitutive theories proposed in this work. In Section \ref{sec:exp_VHB_4910}, the proposed models are calibrated using the experimental data of VHB 4910 from \cite{Hossain2012}. The results demonstrate the capability and potential advantages of the proposed modeling framework in characterizing the representative viscoelastomer. In Section \ref{sec:results_FEM}, finite element analysis is performed on the proposed models to investigate the influence of generalized strains and kinematic decomposition on the dissipation behavior. In both calibration and simulations, we consider the finite linear viscoelasticity (FLV) models presented in Section \ref{sec:quad-energy-model} and the micromechanical models developed in Section \ref{sec:micromechanical-model}. As for the former, the model is given by the energy in the form $G = G_{\mathrm{ich}}^{\infty} + \Upsilon + G_{\mathrm{vol}}$, with
\begin{align}
\label{eq:sec-results-flv-model}
G^{\infty}_{\mathrm{ich}}(\tilde{\bm C}) = \sum_{\beta = 1}^{N}\frac{\mu^\infty_\beta}{2}\left|\tilde{\bm E}_\beta^\infty(\tilde{\bm C}) \right|^2, \quad
\Upsilon(\tilde{\bm C}, \bm E^{\mathrm v}_1, \cdots, \bm E^{\mathrm v}_{M}) = \sum_{\alpha =1 }^{M}\frac{\mu^\mathrm{neq}_\alpha}{2}\left| \tilde{\bm E}^\mathrm{neq}_{\alpha}(\tilde{\bm C}) -\bm E^\mathrm{v}_\alpha \right|^2.
\end{align}
Here, $\tilde{\bm E}^{(\cdot)}_{(\cdot)}$ can be instantiated by different generalized strains. For the micromechanical models, the energy is given by $G = G_{\mathrm a}^{\infty} + \Upsilon + G_{\mathrm{vol}}$, with
\begin{align}
\label{eq:sec-results-mm-1}
G^{\infty}_{\mathrm{a}}(\bm C) = \mu^\infty N^\infty \left( \lambda^{\infty} \mathfrak L^{-1}(\lambda^{\infty}) + \ln \frac{\mathfrak L^{-1}(\lambda^{\infty})}{\sinh(\mathfrak L^{-1}(\lambda^{\infty}))} \right) - \frac{\mu^\infty\sqrt{N^\infty}}{3} \mathfrak L^{-1}(1/\sqrt{N^\infty}) \ln(J)
\end{align}
and
\begin{align}
\label{eq:sec-results-mm-2}
\Upsilon(\bm C^{\mathrm e}) = \mu^\mathrm{neq} N^\mathrm{neq} \left( \lambda^{\mathrm{neq}} \mathfrak L^{-1}(\lambda^{\mathrm{neq}}) + \ln \frac{\mathfrak L^{-1}(\lambda^{\mathrm{neq}})}{\sinh(\mathfrak L^{-1}(\lambda^{\mathrm{neq}}))} \right) - \frac{\mu^\mathrm{neq} \sqrt{N^\mathrm{neq}}}{3} \mathfrak L^{-1}(1/\sqrt{N^\mathrm{neq}}) \ln(J^{\mathrm e}).
\end{align}
In particular, we adopt a single relaxation process for the micromechanical models in this study. For the micromechanical model, we also invoke the multiplicative decomposition for comparison purposes.

\subsection{Model calibration based on VHB 4910}
\label{sec:exp_VHB_4910}
VHB 4910 is an acrylic polymer capable of undergoing large deformations with pronounced dissipative behavior. Mechanical characterization of this material is critical for understanding its performance in novel devices. In this study, experimental data are obtained from \cite{Hossain2012}. A suite of uniaxial tests was conducted on specimens with a length-to-width ratio of 10:1, ensuring uniaxial stress conditions throughout the experiment. Tensile loading was applied at three strain rates of $0.01$s$^{-1}$, $0.03$s$^{-1}$, and $0.05$s$^{-1}$, with final strains extending from $50\%$ to $200\%$, followed by unloading to zero stress. Nominal stress values are calculated through dividing the applied force by the specimen’s original cross-sectional area. The experimental data reveals significant rate-dependent hysteresis and stress relaxation, which are used for examining the proposed material models.

\begin{table}[htbp]
\centering
\begin{tabular}{>{\centering\arraybackslash}m{4cm} >{\centering\arraybackslash}m{2cm} >{\centering\arraybackslash}m{2cm} >{\centering\arraybackslash}m{2cm}}
\toprule
& \multicolumn{3}{c}{\textbf{Stretch rate (s$^{-1}$)}} \\
\cmidrule(lr){2-4}
\textbf{The maximum stretch} & \textbf{0.01} & \textbf{0.03} & \textbf{0.05} \\
\midrule
\textbf{1.5} & 80 & 83 & 86 \\
\textbf{2.0} & 80 & 83 & 86 \\
\textbf{2.5} & 86 & 88 & 91 \\
\textbf{3.0} & 85 & - & 91 \\
\bottomrule
\end{tabular}
\caption{The number of data points corresponding to different maximum stretches and strain rates of the uniaxial tests extracted from \cite{Hossain2012}. }
\label{table:Hossain_data}
\end{table}

The metric used for calibration and evaluation is NMAD, as defined in \eqref{eq:NMAD}. The numbers of data points in different loading cases are listed in Table \ref{table:Hossain_data}. We mention that for the case with a maximum stretch of $3.0$, the data for the strain rate of $0.03$ s$^{-1}$ is unavailable. Within the FLV framework, three different types of generalized strains are employed within the quadratic energy formulation. For the micromechanical models, fittings are conducted based on two distinct kinematic assumptions: the multiplicative decomposition and the additive decomposition. For the latter, there remain options for the choice of generalized strains, and we opt for the Hencky and Curnier-Rakotomanana strains.

\begin{figure}
\begin{center}
\begin{tabular}{ccc}
\includegraphics[angle=0, trim=50 30 160 120, clip=true, scale=0.075]{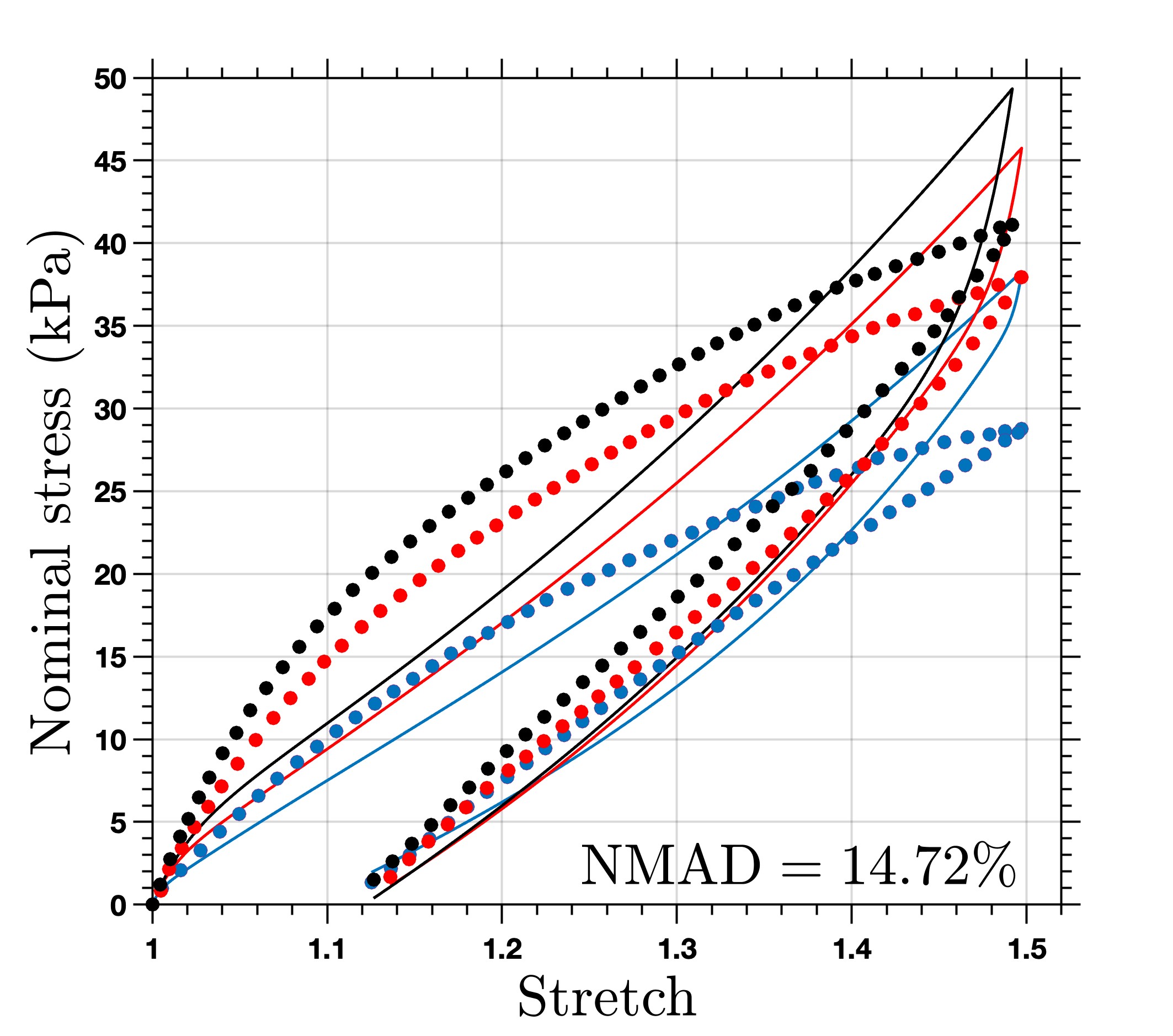} &
\includegraphics[angle=0, trim=50 30 160 120, clip=true, scale=0.075]{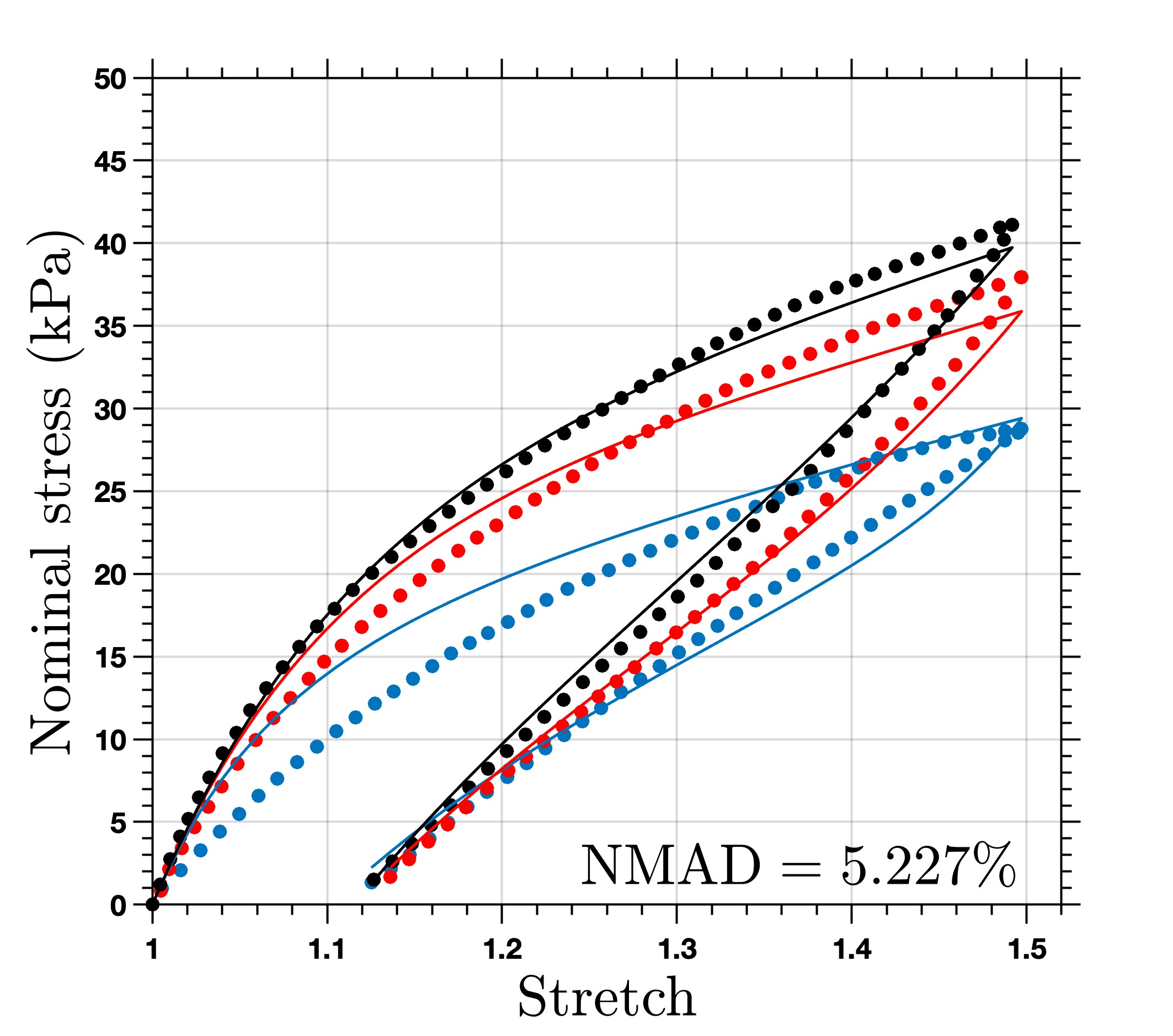} &
\includegraphics[angle=0, trim=50 30 160 120, clip=true, scale=0.075]{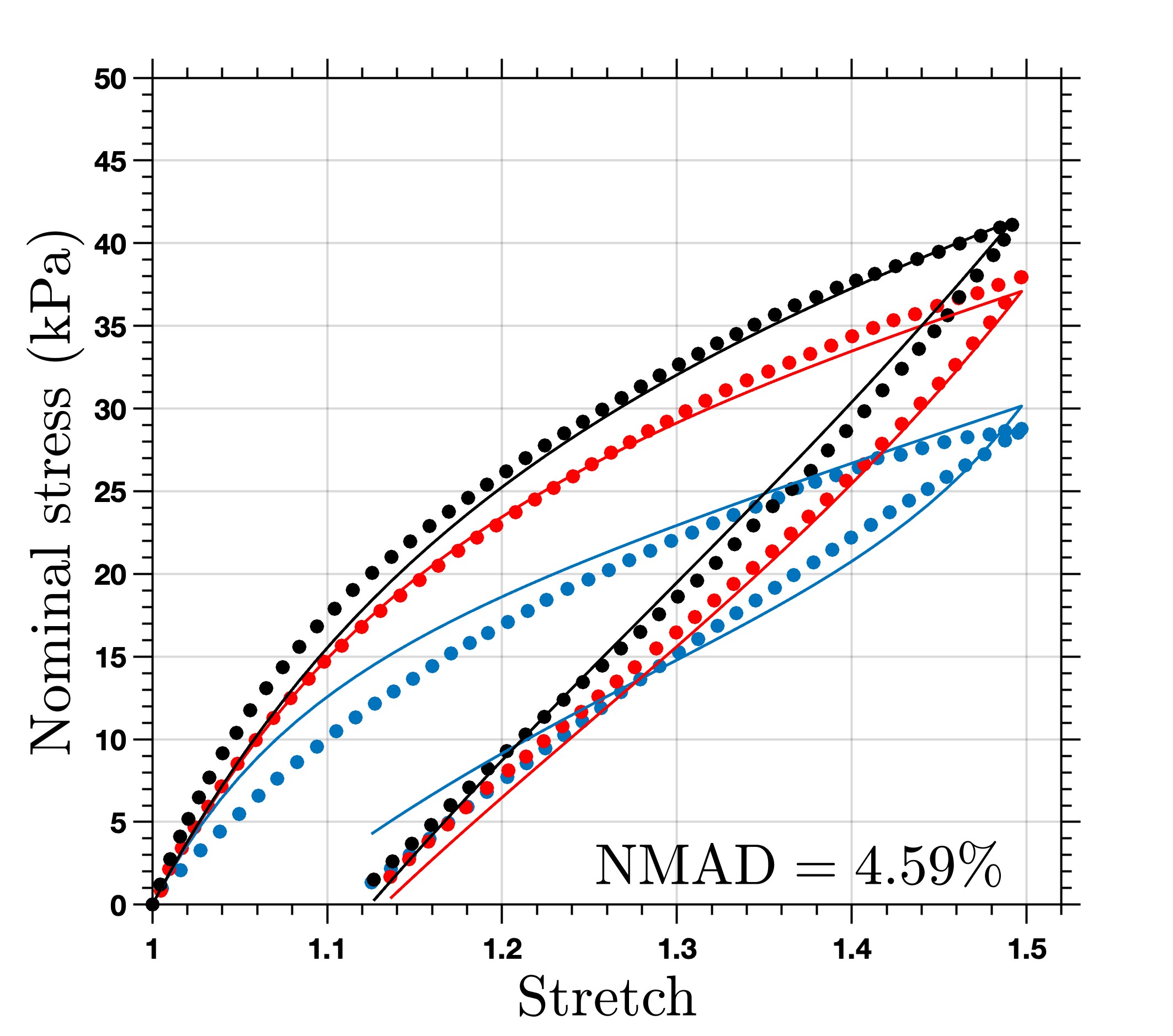} \\
\includegraphics[angle=0, trim=50 30 160 120, clip=true, scale=0.075]{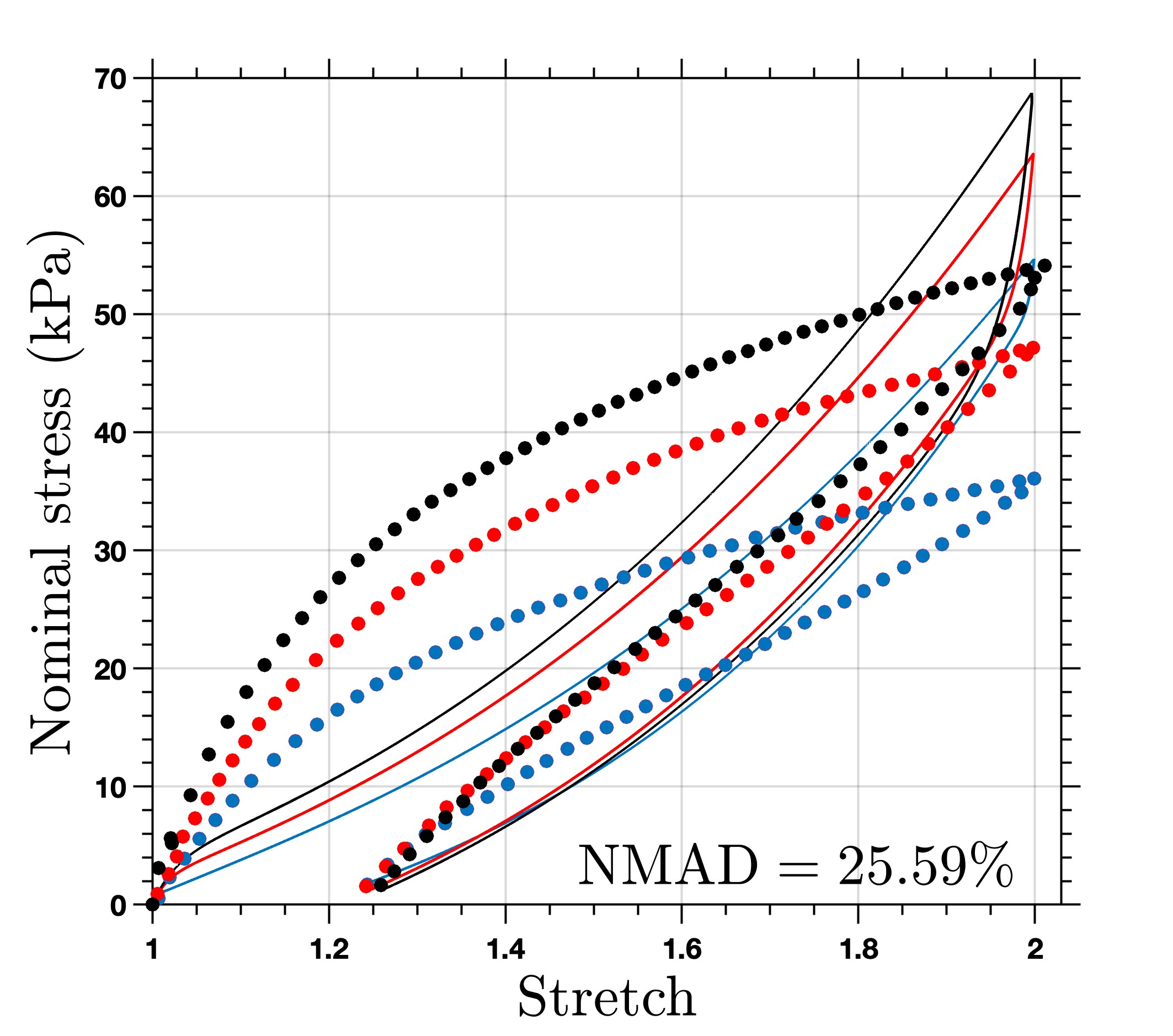} &
\includegraphics[angle=0, trim=50 30 160 120, clip=true, scale=0.075]{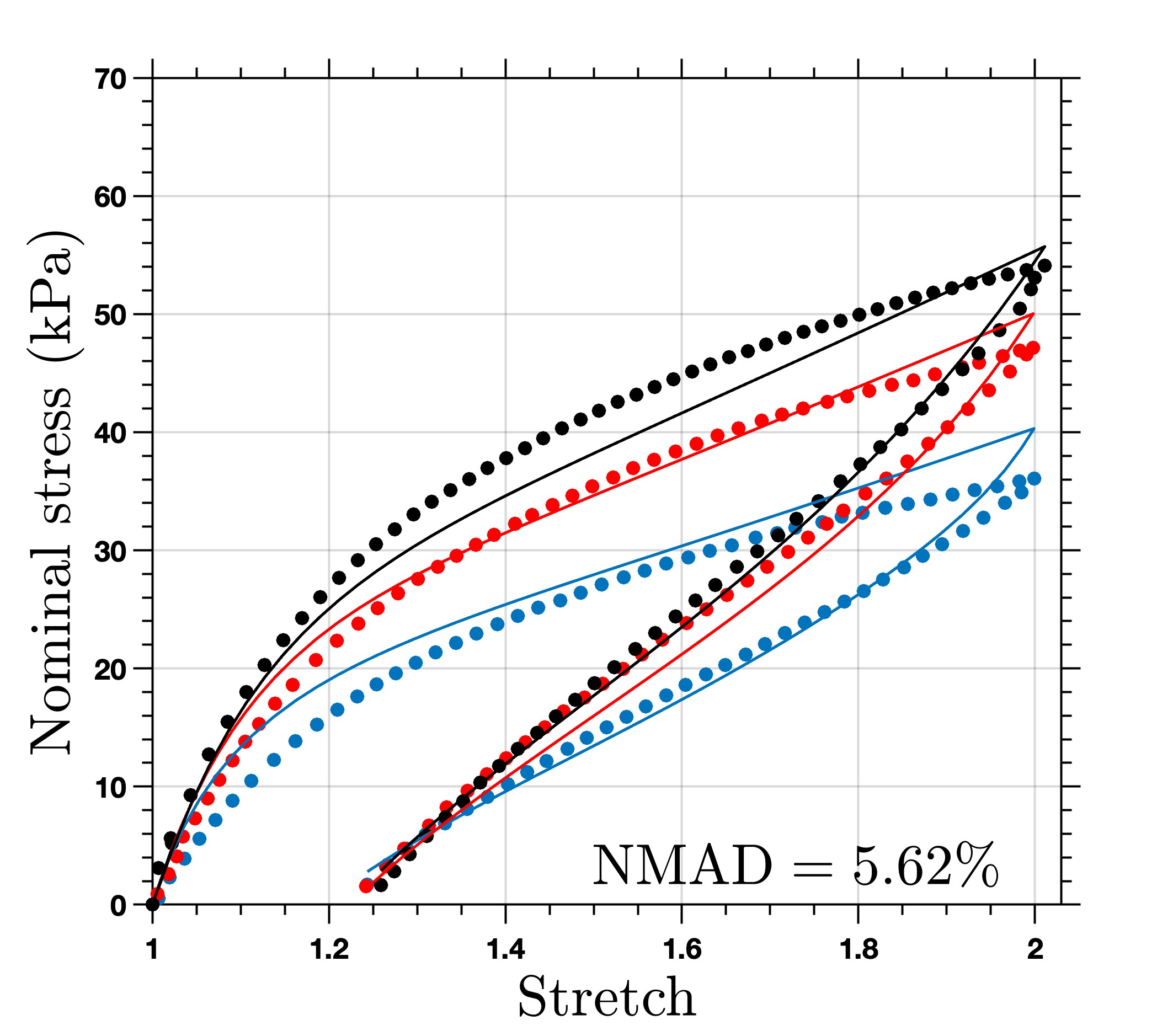} &
\includegraphics[angle=0, trim=50 30 160 120, clip=true, scale=0.075]{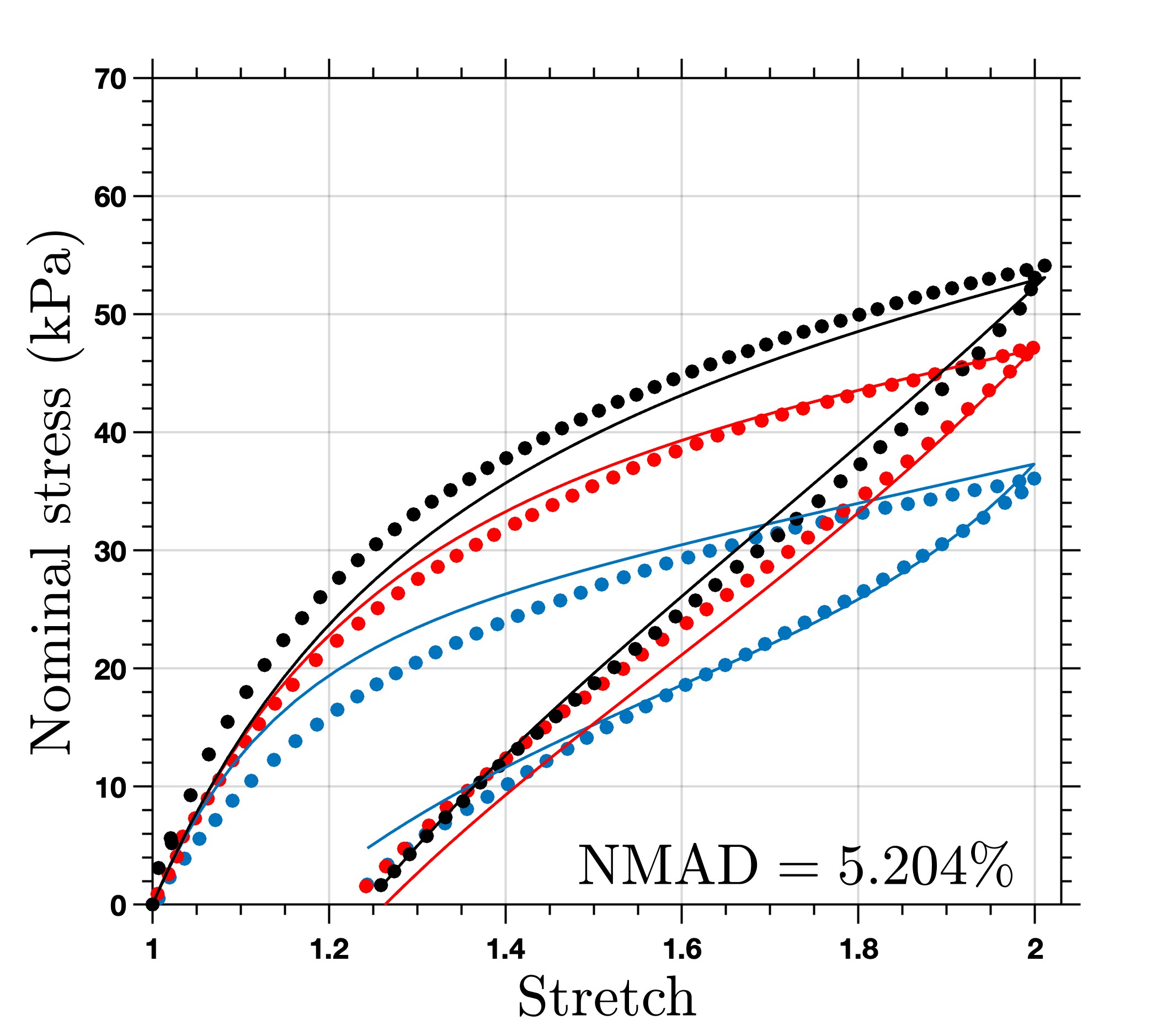} \\
\includegraphics[angle=0, trim=50 30 160 120, clip=true, scale=0.075]{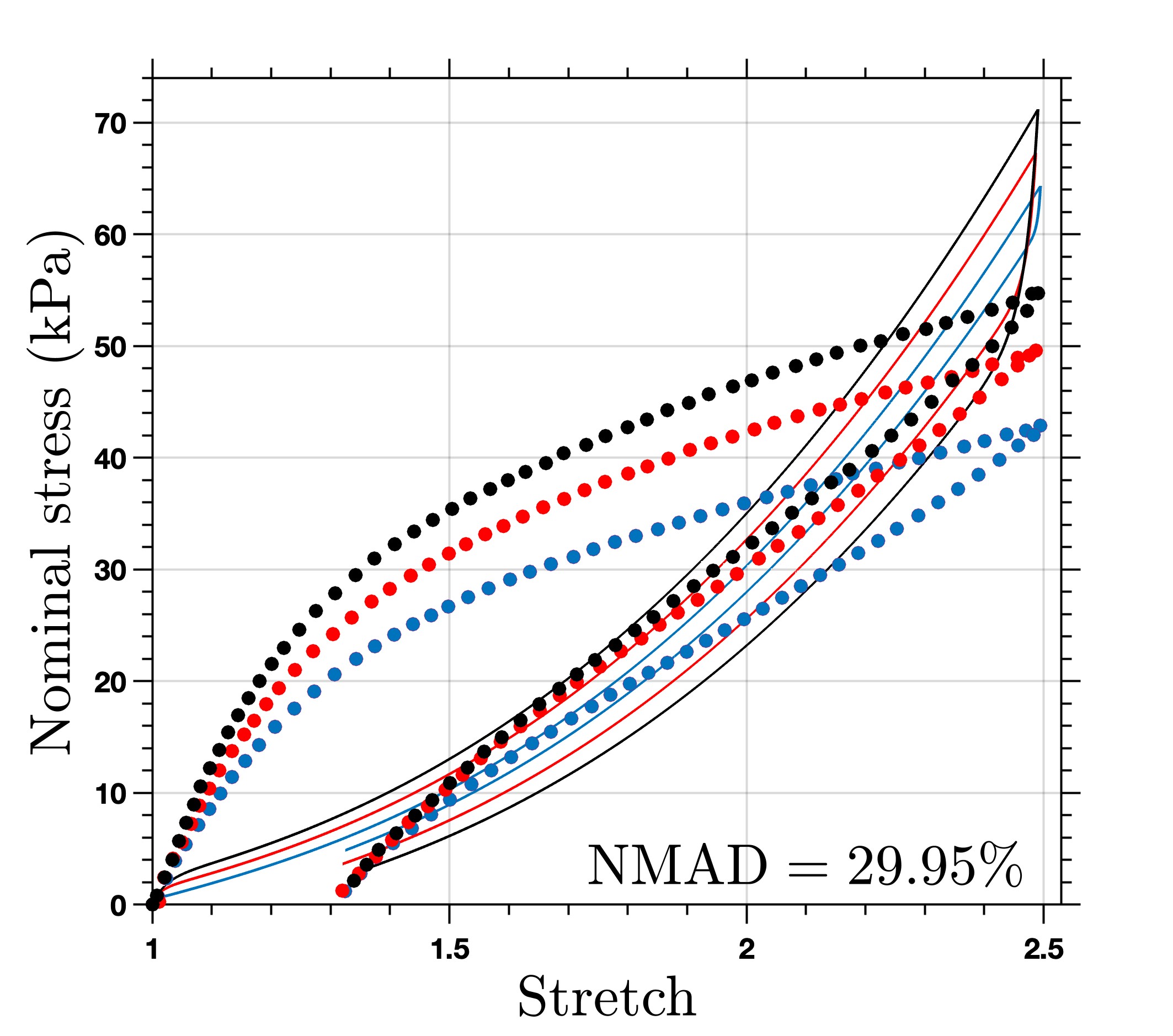} &
\includegraphics[angle=0, trim=50 30 160 120, clip=true, scale=0.075]{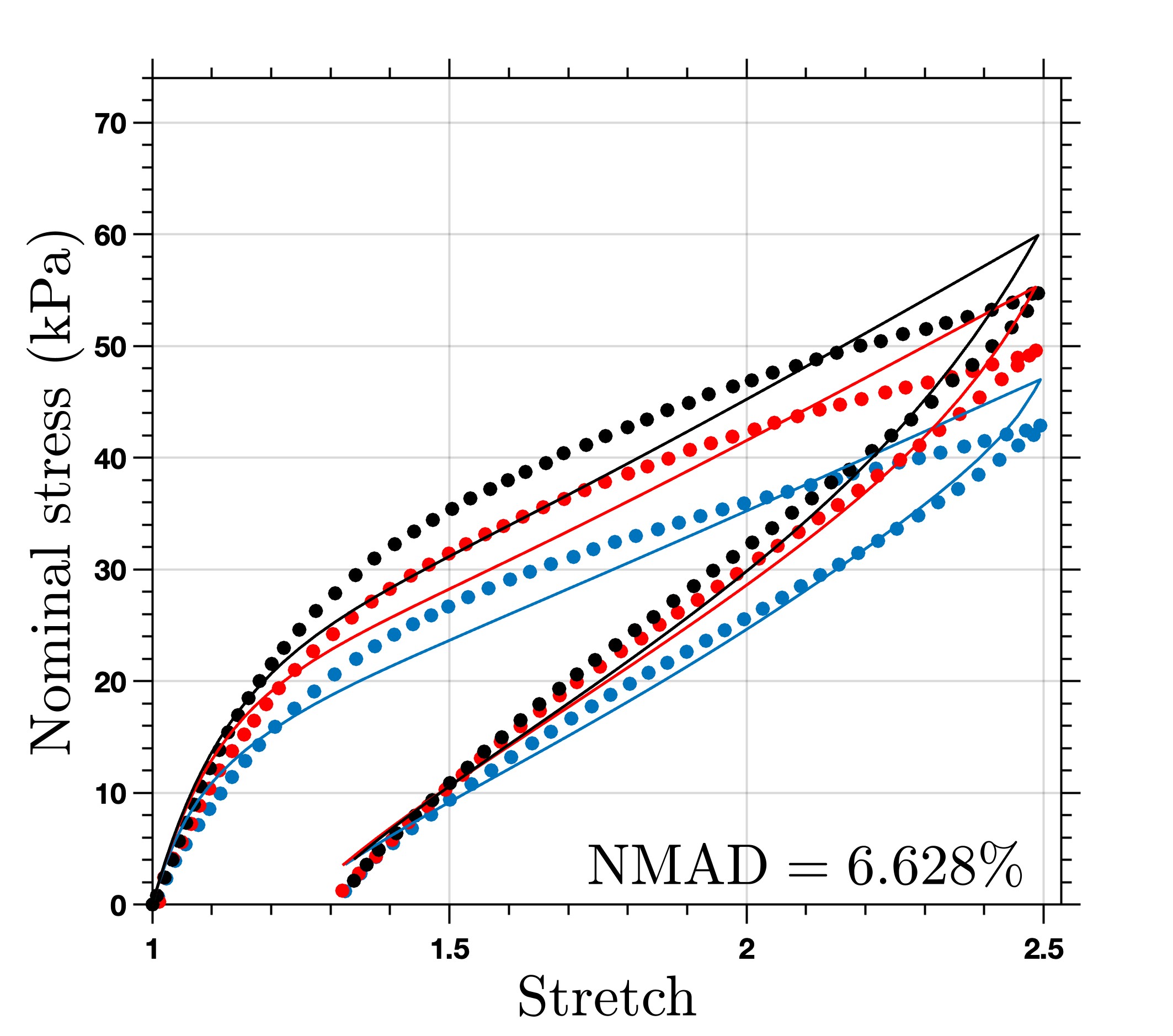} &
\includegraphics[angle=0, trim=50 30 160 120, clip=true, scale=0.075]{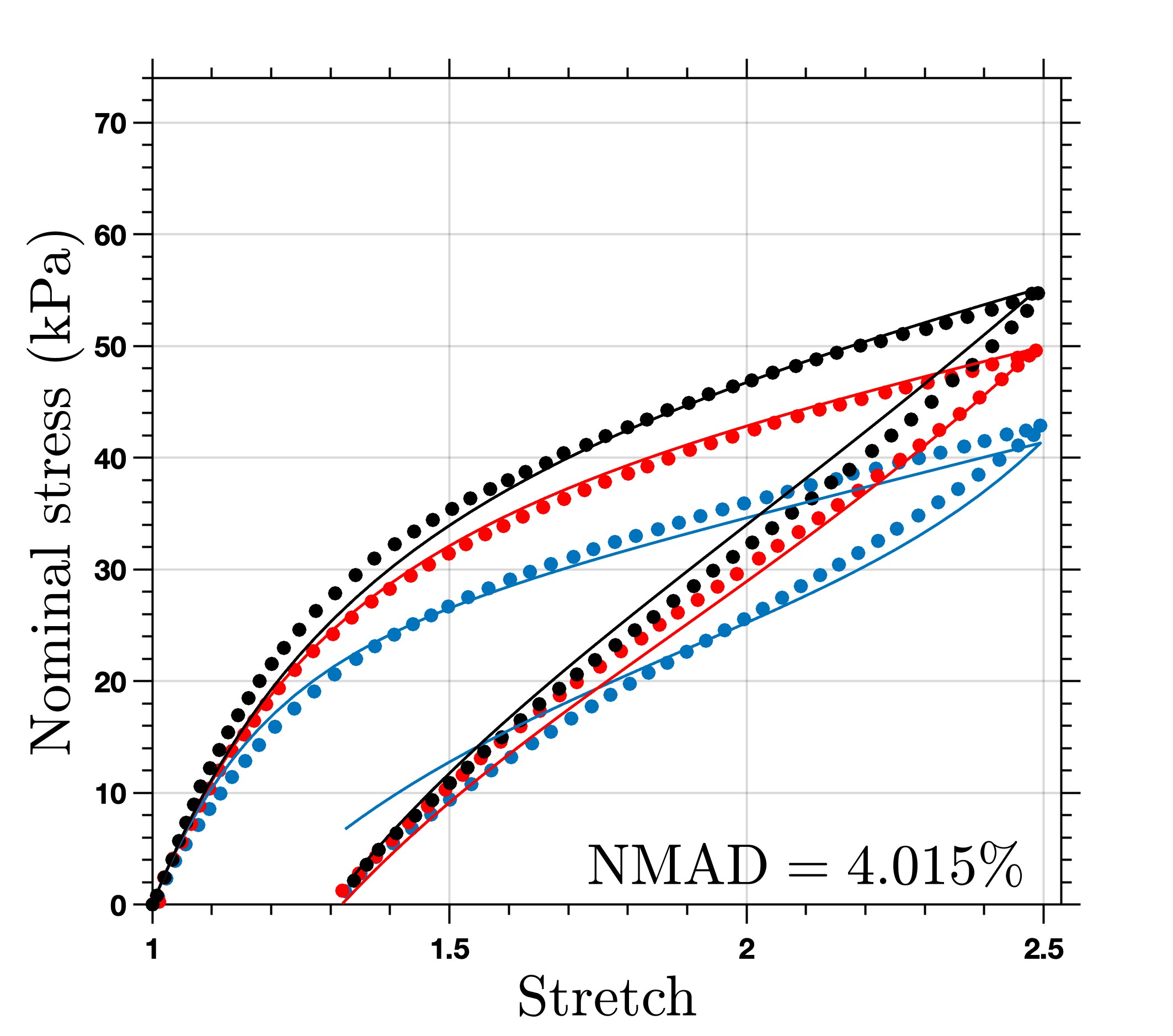} \\
\includegraphics[angle=0, trim=50 30 160 120, clip=true, scale=0.075]{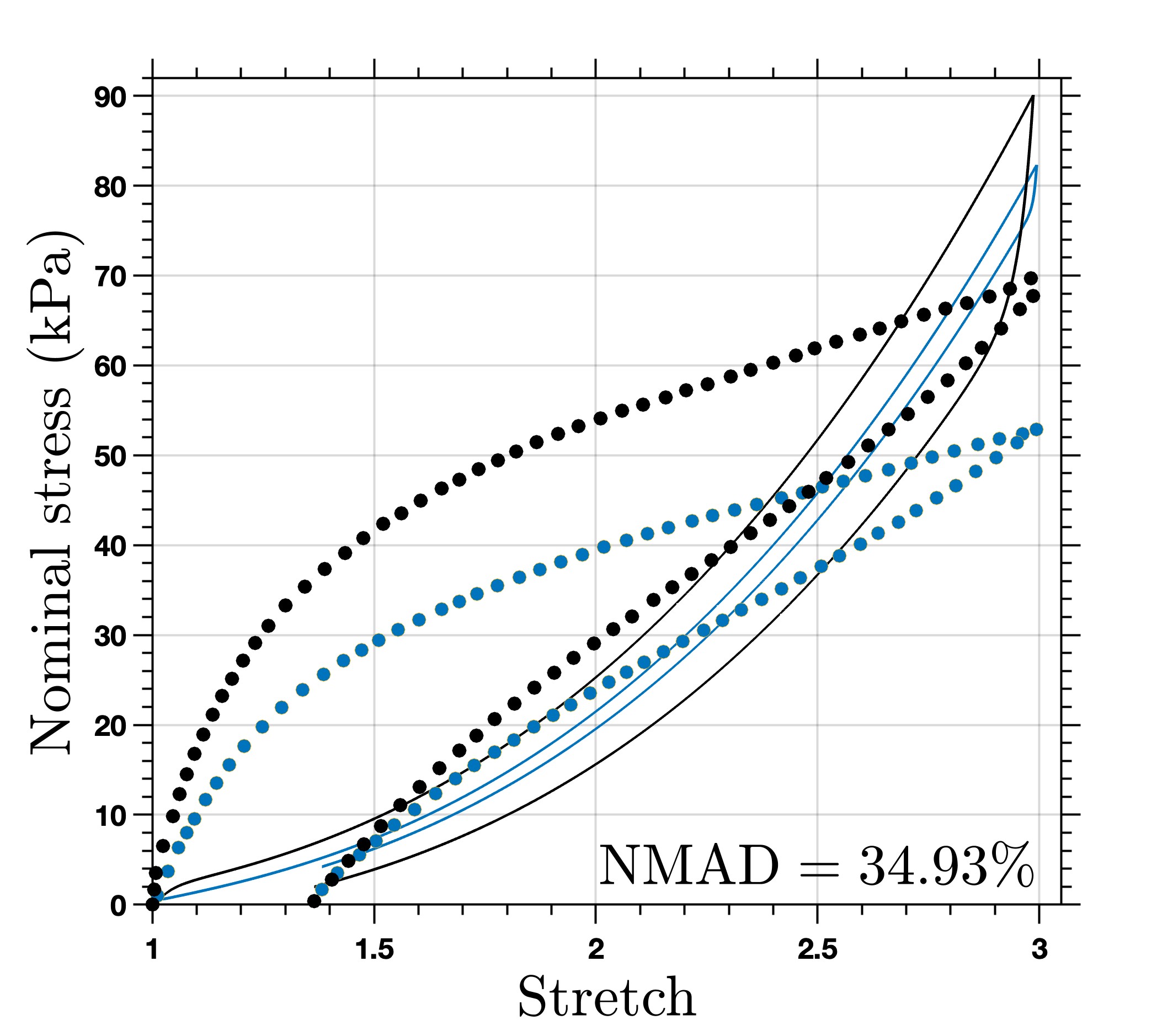} &
\includegraphics[angle=0, trim=50 30 160 120, clip=true, scale=0.075]{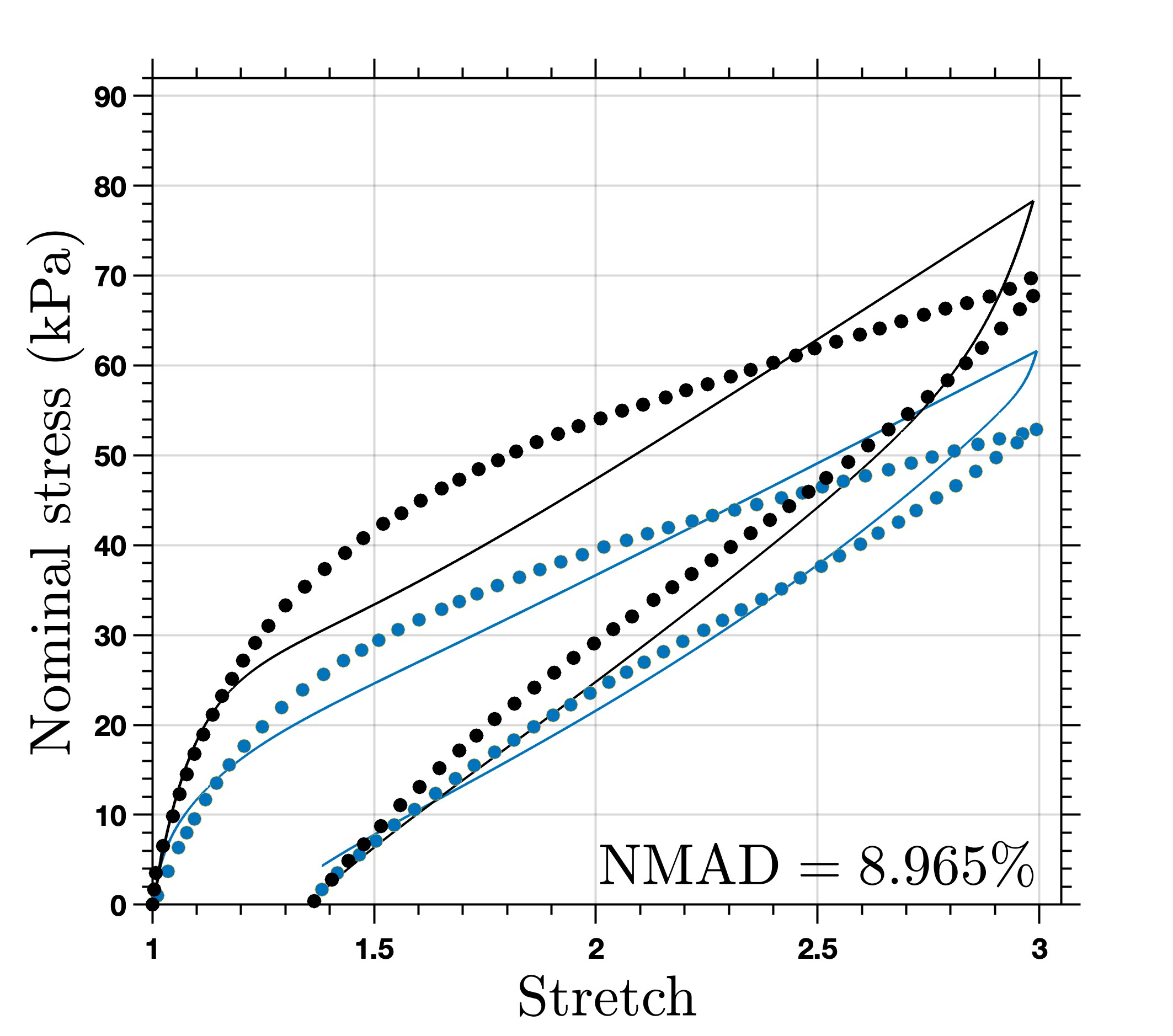} &
\includegraphics[angle=0, trim=50 30 160 120, clip=true, scale=0.075]{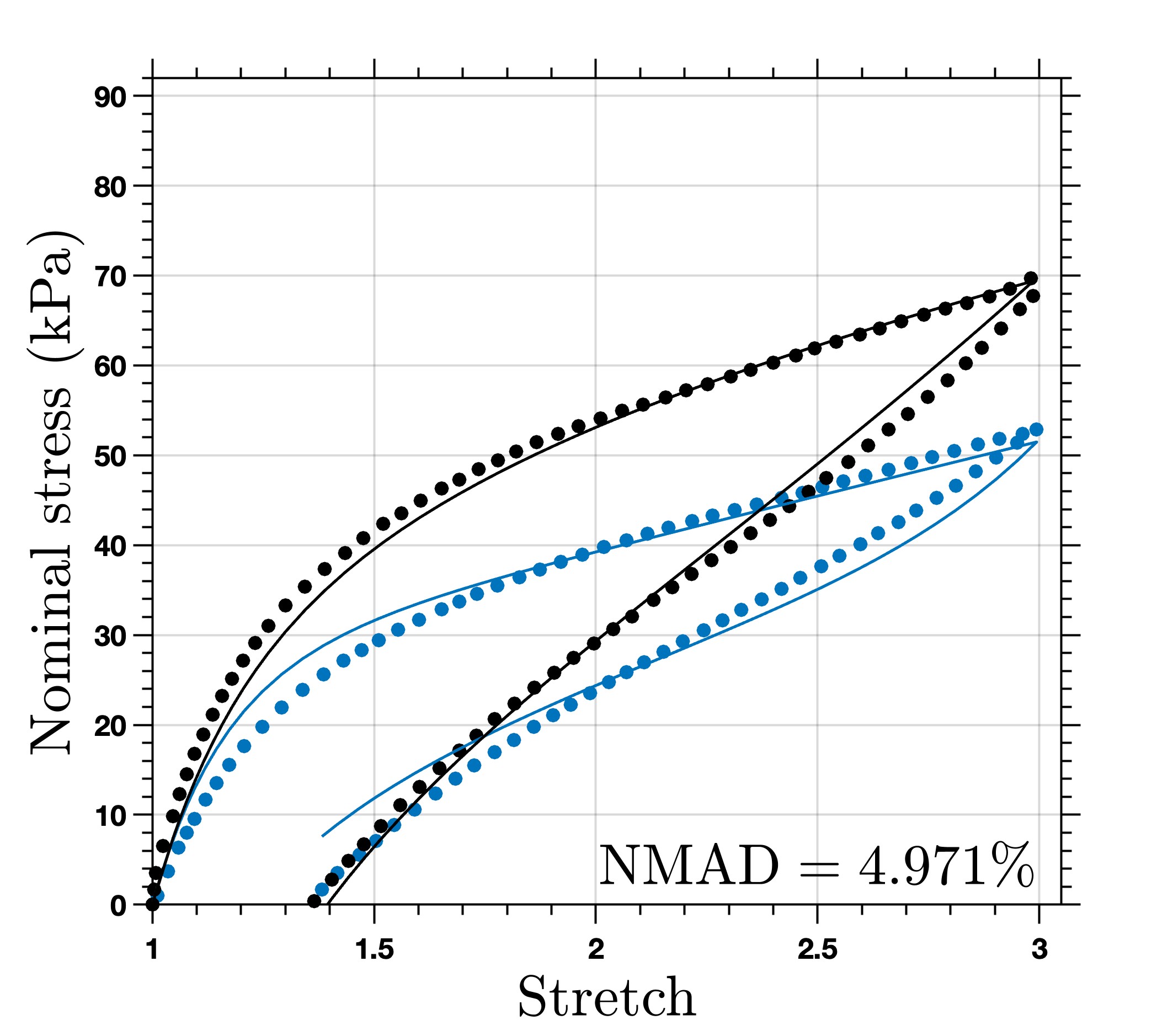} \\
Green-Lagrange strain & Euler-Almansi strain & Curnier-Rakotomanana strain
\end{tabular}
\end{center}
\caption{The results of simultaneous fitting for different maximum stretches using the FLV model with the Green-Lagrange strain (left), Euler-Almansi strain (middle), and Curnier-Rakotomanana strain (right). The fittings are performed on uniaxial test data from \cite{Hossain2012}, corresponding to stretch rates of $0.01$ s$^{-1}$ (blue), $0.03$ s$^{-1}$ (red), and $0.05$ s$^{-1}$ (black).}
\label{fig:results_quad_Hossain}
\end{figure}

\paragraph{Calibration of the finite linear viscoelasticity models}
We evaluate the modeling capability of the quadratic energy form with a variety of generalized strains. Simultaneous fittings are performed for each given maximum stretch, using data of distinct strain rates. In the first case, the energy is constructed with $N=1$ and $M=2$, using the Green-Lagrange strain. This choice is related to the model of \cite{Simo1987} and involves a total of $5$ parameters. The fitting results, presented in the left column of Figure \ref{fig:results_quad_Hossain}, indicate that the model demonstrates rate dependence across different maximum stretches. However, using the Green-Lagrange strain fails to accurately capture the actual material behavior and exhibits a significant hardening near the maximum stretch.

In the second case, the energy is also constructed with $N=1$ and $M=2$, using the Euler-Almansi strain. This choice aligns with the model of \cite{Green1946} and \cite{Lubliner1985} and also involves $5$ parameters in total. The fitting results, shown in the middle column of Figure \ref{fig:results_quad_Hossain}, demonstrate that the model effectively captures the rate-dependent behavior and more accurately characterizes the stress-strain curve. Nevertheless, as the stretch increases to $3.0$, the model exhibits hardening, deviating from the experimental data significantly. For the quadratic energy considered here, the type of the generalized strain dictates the overall stress-strain curve, while the remaining material parameters (such as the shear moduli and viscosities) influence characteristics like the slope of the curve. With this observation, we proceed to calibrate the model with the strain parameters taken into account.

In the third case, the energy is constructed with $N=M=1$, and the generalized strains are characterized using the Curnier-Rakotomanana strain, with the strain parameters subjected to optimization. Therefore, in addition to $\mu_1^{\infty}$, $\mu_1^{\mathrm{neq}}$, and $\eta$, there are two additional strain parameters, leading to a total of $5$ parameters. The fitting results, presented in the right column of Figure \ref{fig:results_quad_Hossain}, demonstrate that the model accurately captures the rate-dependent behavior of VHB 4910 during uniaxial loading-unloading tests conducted at three different stretch rates. The optimized parameters reveal that the generalized strain parameters effectively improve the overall stress-strain curve, especially when the maximum stretch exceeds $2.0$.

For the first and second cases, we also conducted calibrations with $M=3$. Yet, it is observed that increasing the number of relaxation processes does not improve the performance of the FLV models if a well-defined generalized strain is absent. Therefore, incorporating suitable generalized strains into the calibration process is crucial for accurately capturing the mechanical behavior of the material. Additionally, the calibrated viscosity steadily increases with larger maximum stretch values, indicating the viscosity of VHB 4910 exhibits non-Newtonian behavior.

\begin{figure}
\begin{center}
\begin{tabular}{cc}
\includegraphics[angle=0, trim=50 30 160 120, clip=true, scale=0.075]{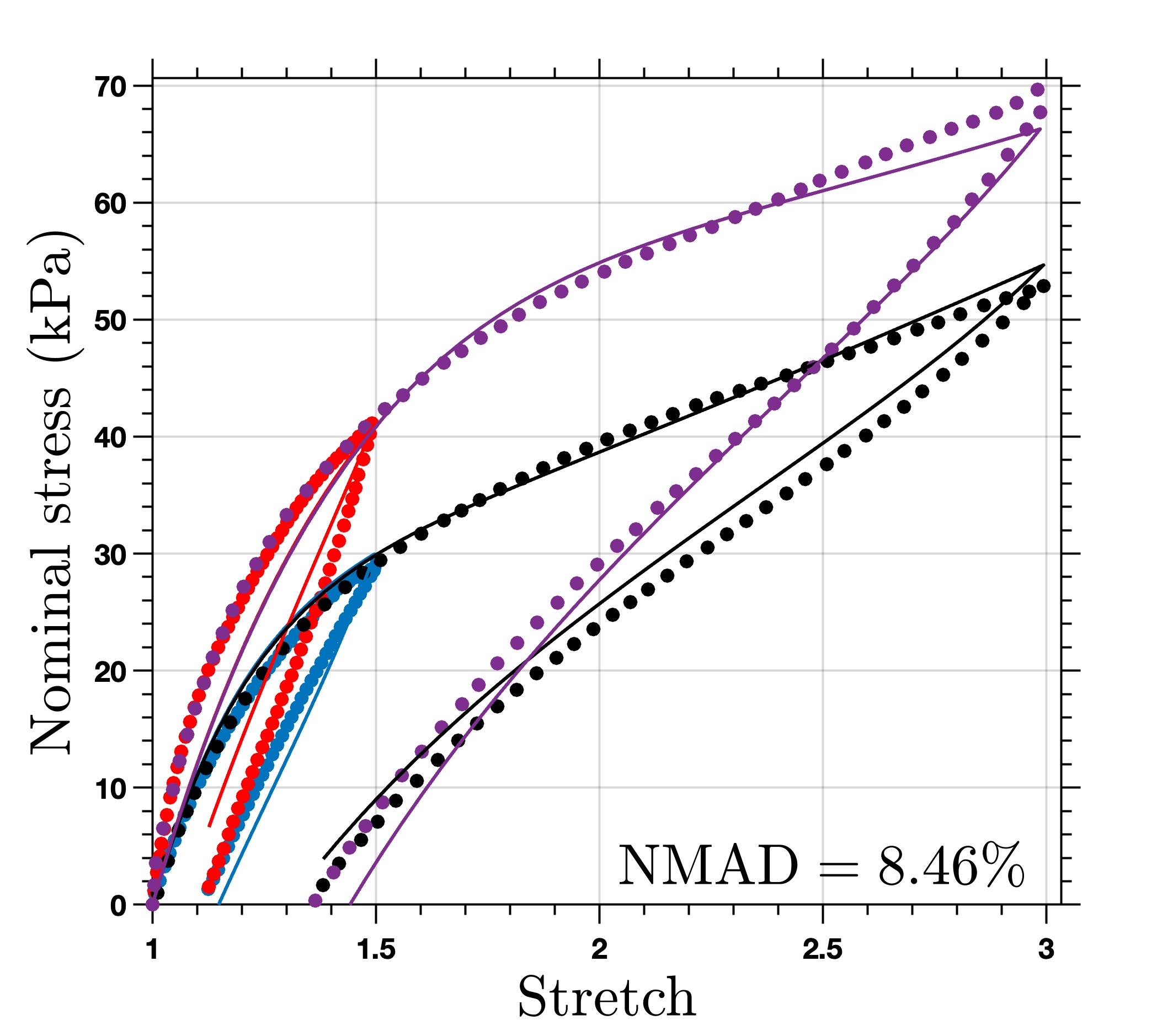} &
\includegraphics[angle=0, trim=50 30 160 120, clip=true, scale=0.075]{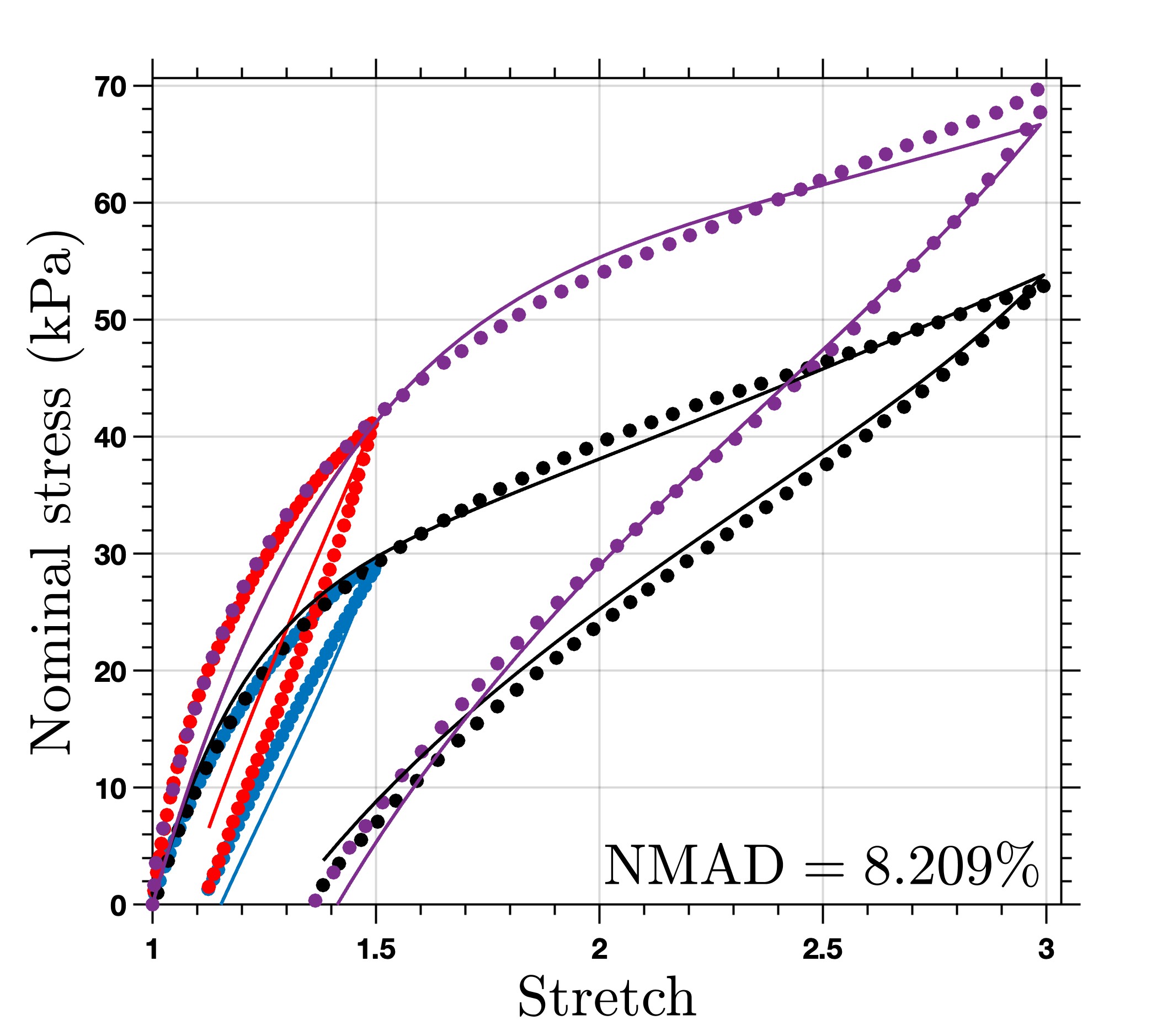} \\
(a) & (b) \\
\includegraphics[angle=0, trim=50 30 160 120, clip=true, scale=0.075]{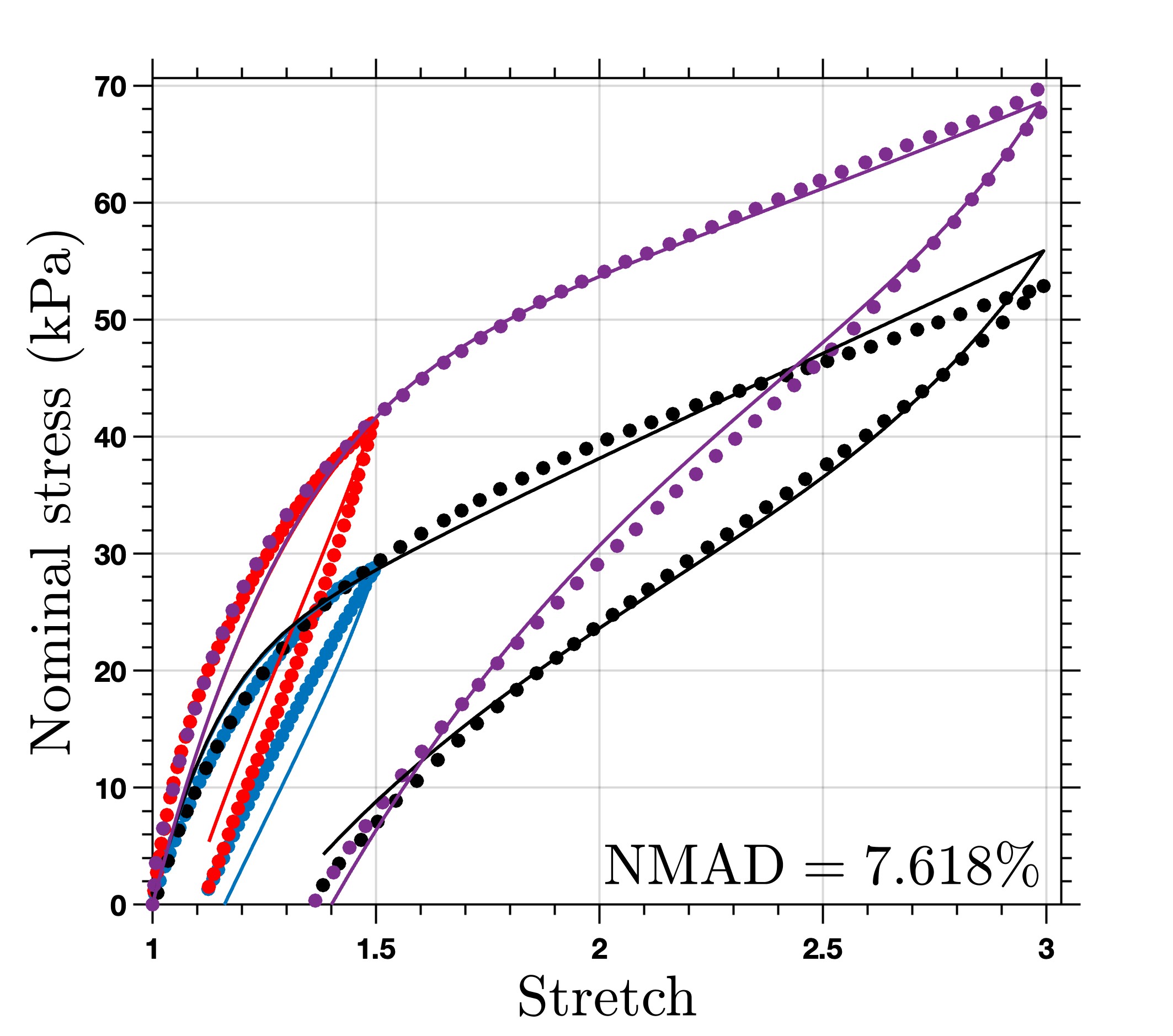} &
\includegraphics[angle=0, trim=50 30 160 120, clip=true, scale=0.075]{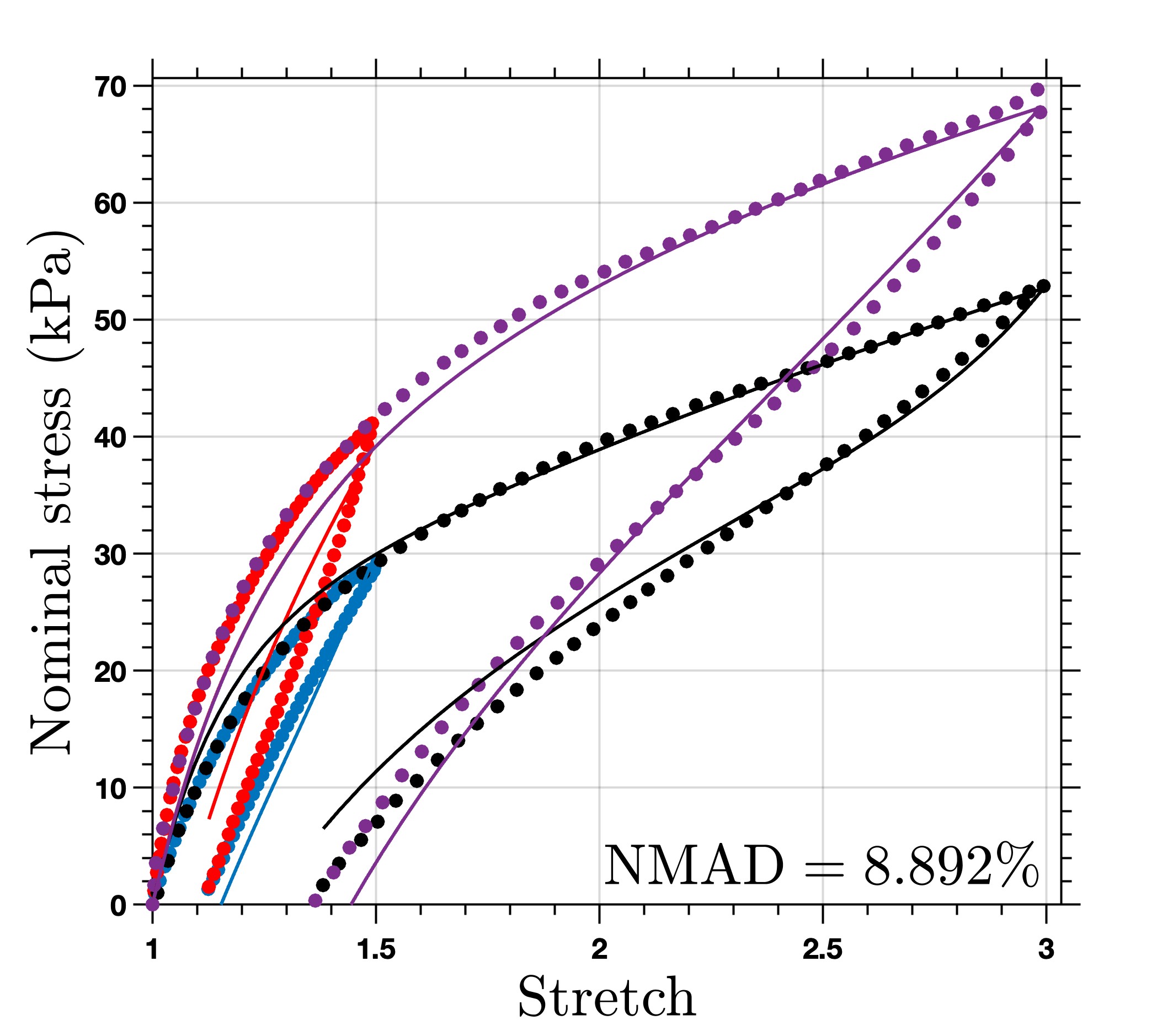} \\
(c) & (d)
\end{tabular}
\end{center}
\caption{Simultaneous fitting results with stretch rates of $0.01$ s$^{-1}$ (blue and black) and $0.05$ s$^{-1}$ (red and purple): (a) micromechanical model with the multiplicative decomposition, (b) micromechanical model with the Hencky strain, (c) micromechanical model with the Curnier-Rakotomanana strain, (d) FLV model with the Curnier-Rakotomanana strain.}
\label{fig:results_micro_compare}
\end{figure}

\paragraph{Calibration of the micromechanical models}
We assess the micromechanical models proposed in Section \ref{sec:micromechanical-model} through model calibration. Experimental data sets with maximum stretches of $1.5$ and $3.0$ at stretch rates of $0.01$ and $0.05$ s$^{-1}$ are used in the simultaneous fitting. The values of NMAD are calculated based on the $\mathcal M=4$ data sets. In the case of multiplicative decomposition, the elastic deformation is constructed by $\bar{\bm F}^\mathrm{e\:T} \bar{\bm F}^\mathrm{e}$, as detailed in Section \ref{sec:choice-of-ISV}. Figure \ref{fig:results_micro_compare} (a) and (b) present the results of the micromechanical models based on the multiplicative decomposition and the additive decomposition with the Hencky strain. Both models effectively capture the rate-dependent behavior and characterize the strain-stress curve, yielding similar NMAD values. This outcome, to some extent, is expected, as the two kinematic assumptions are identical when $\bar{\bm F}^{\mathrm v}$ is a pure stretch tensor coaxial with $\bm C$, see Appendix C of \cite{Liu2024}.

\begin{table}[h]
\centering
\begin{tabular}{>{\centering\arraybackslash}m{1.0cm} >{\centering\arraybackslash}m{1.5cm} >{\centering\arraybackslash}m{1.5cm} >{\centering\arraybackslash}m{0.75cm}>{\centering\arraybackslash}m{0.75cm}>{\centering\arraybackslash}m{1.5cm} >{\centering\arraybackslash}m{1.0cm} >{\centering\arraybackslash}m{0.75cm} >{\centering\arraybackslash}m{0.75cm} >{\centering\arraybackslash}m{1.5cm} >{\centering\arraybackslash}m{2.0cm}}
\toprule
\textbf{Figure} & $\boldsymbol{\mu^\infty}\:\mathbf{(kPa)}$  & $\boldsymbol{N^\infty}$ & $\boldsymbol{m^\infty}$ & $\boldsymbol{n^\infty}$ & $\boldsymbol{\mu^\mathrm{neq}}\:\mathbf{(kPa)}$  & $\boldsymbol{N^\mathrm{neq}}$ & $\boldsymbol{m^\mathrm{neq}}$ & $\boldsymbol{n^\mathrm{neq}}$ & $\boldsymbol{\eta}\:\mathbf{(kPa\cdot s)}$ & \textbf{NMAD(\%)} \\
\midrule
\textbf{\ref{fig:results_micro_compare} (a)} & 17.64 & 303.23 & -  & - & 27.11& 350.24 & - & - & 434.30 & 8.46\\
\midrule
\textbf{\ref{fig:results_micro_compare} (b)} & 17.48 & 320.29 & -  & - & 27.73& 102.09 & - & - & 1789.94 & 8.209\\
\midrule
\textbf{\ref{fig:results_micro_compare} (c)} & 17.26 & 177.63 & -  & - & 31.50& 122.63 & 0.90 & 1.67 & 1725.83 & 7.618\\
\midrule
\textbf{\ref{fig:results_micro_compare} (d)} & 20.01 & - & 0.82  & 0.25 & 35.25& - & 0.08 & 1.34 & 933.12 & 8.892\\
\bottomrule
\end{tabular}
\caption{Material parameters from the simultaneous fitting results of Figure \ref{fig:results_micro_compare} }
\label{table:parameters_micro}
\end{table}

\begin{figure}
\begin{center}
\begin{tabular}{cc}
\includegraphics[angle=0, trim=50 30 160 120, clip=true, scale=0.075]{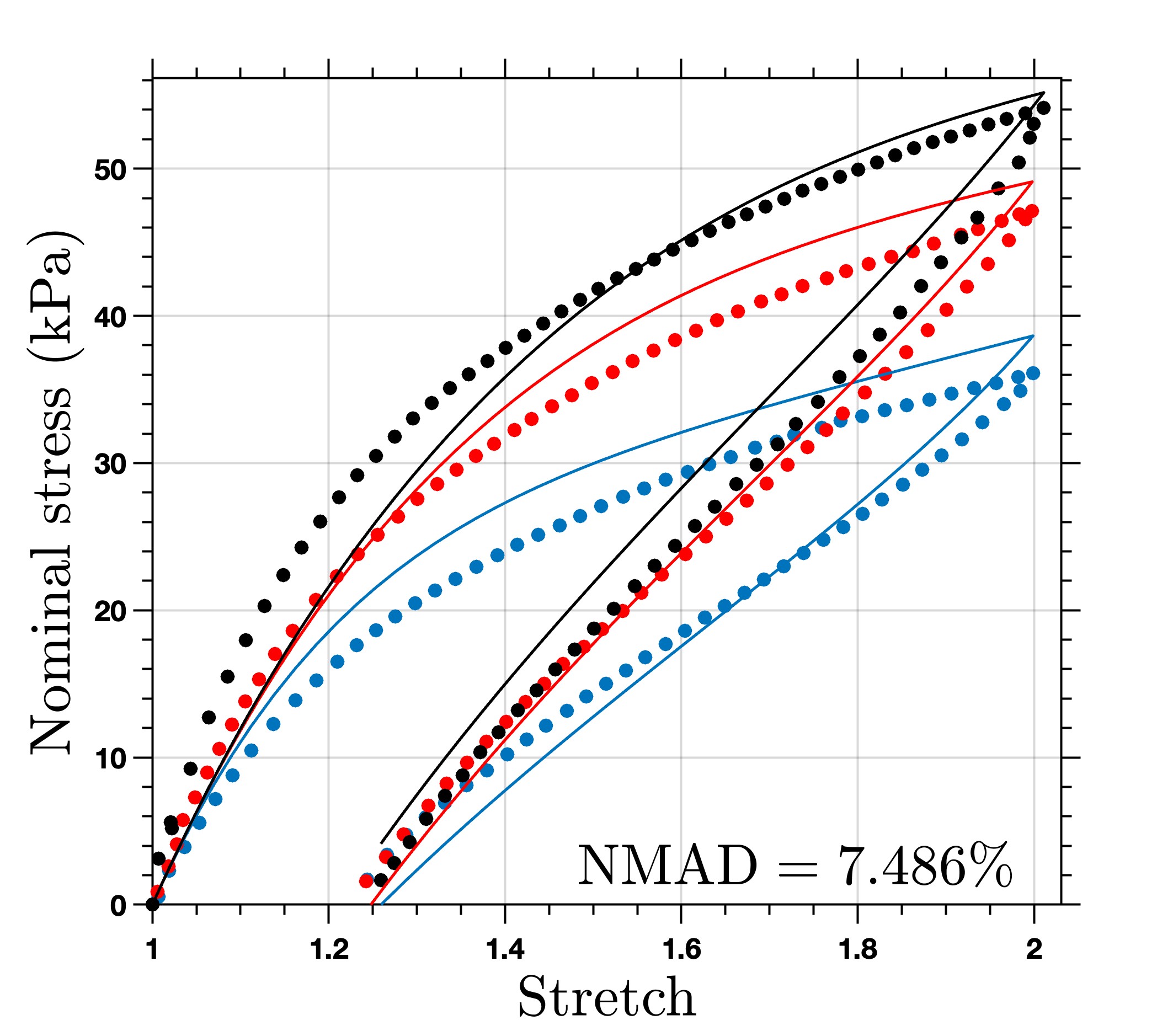}&
\includegraphics[angle=0, trim=50 30 160 120, clip=true, scale=0.075]{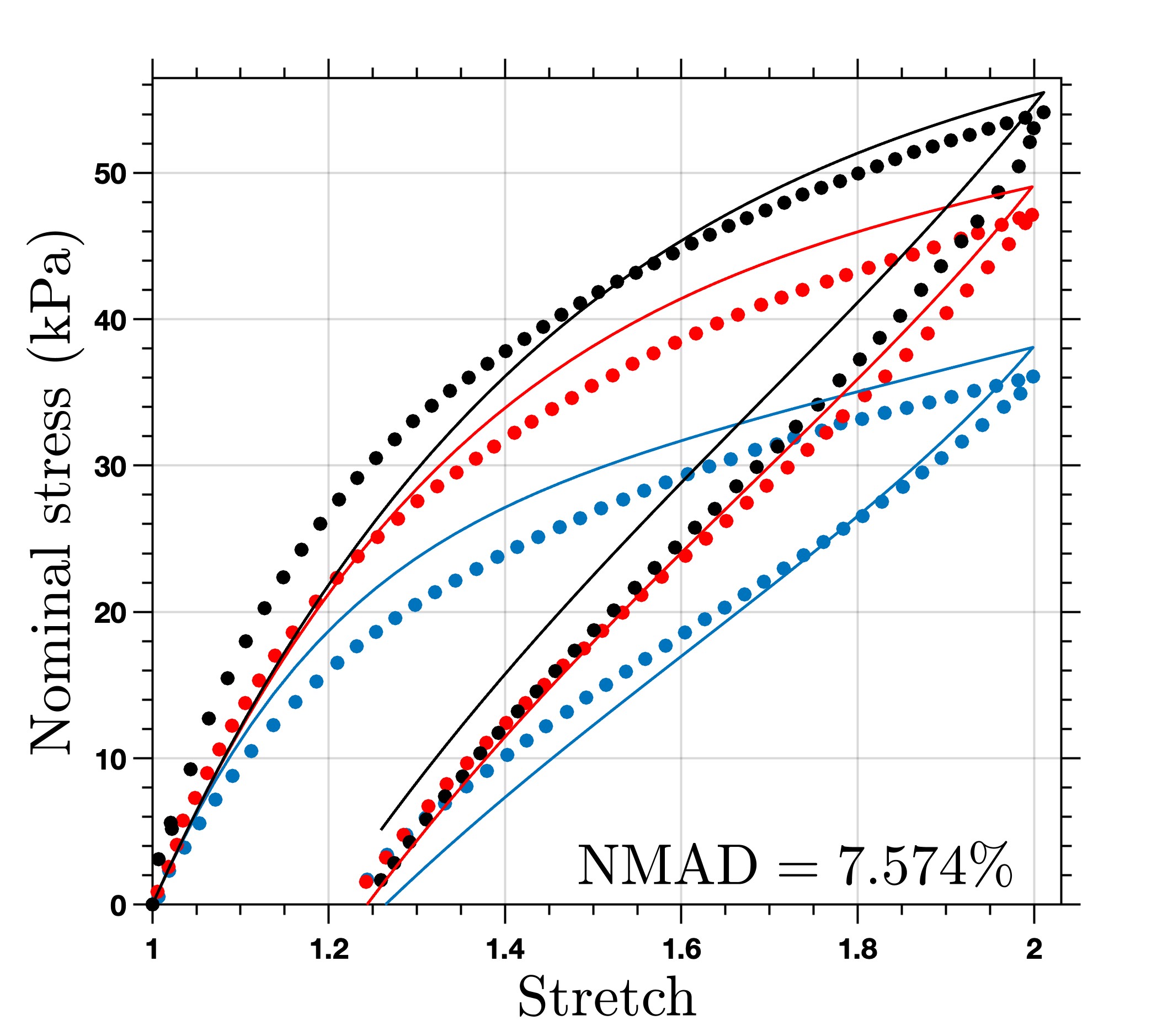} \\
(a) & (b) \\
\includegraphics[angle=0, trim=50 30 160 120, clip=true, scale=0.075]{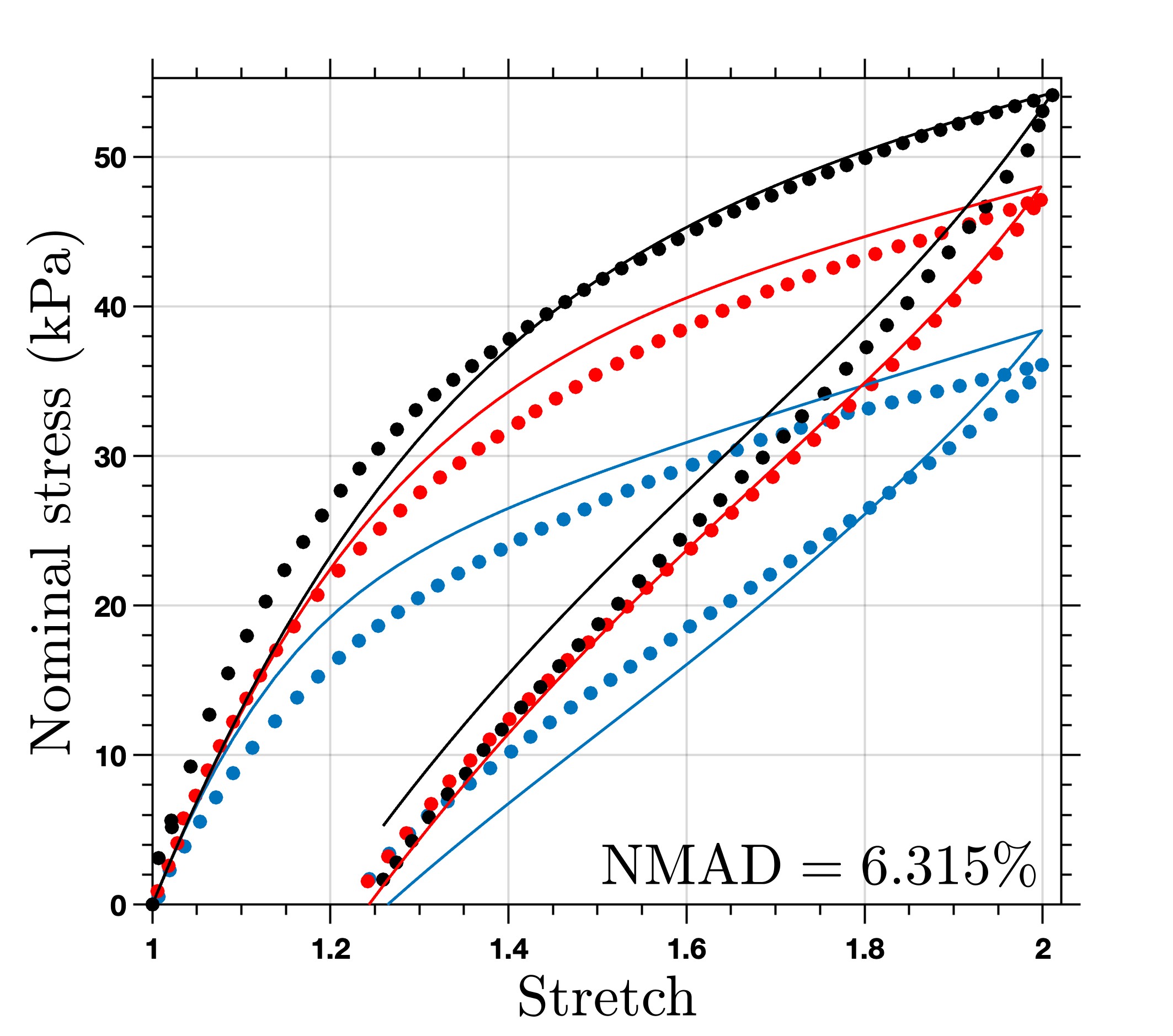} &
\includegraphics[angle=0, trim=50 30 160 120, clip=true, scale=0.075]{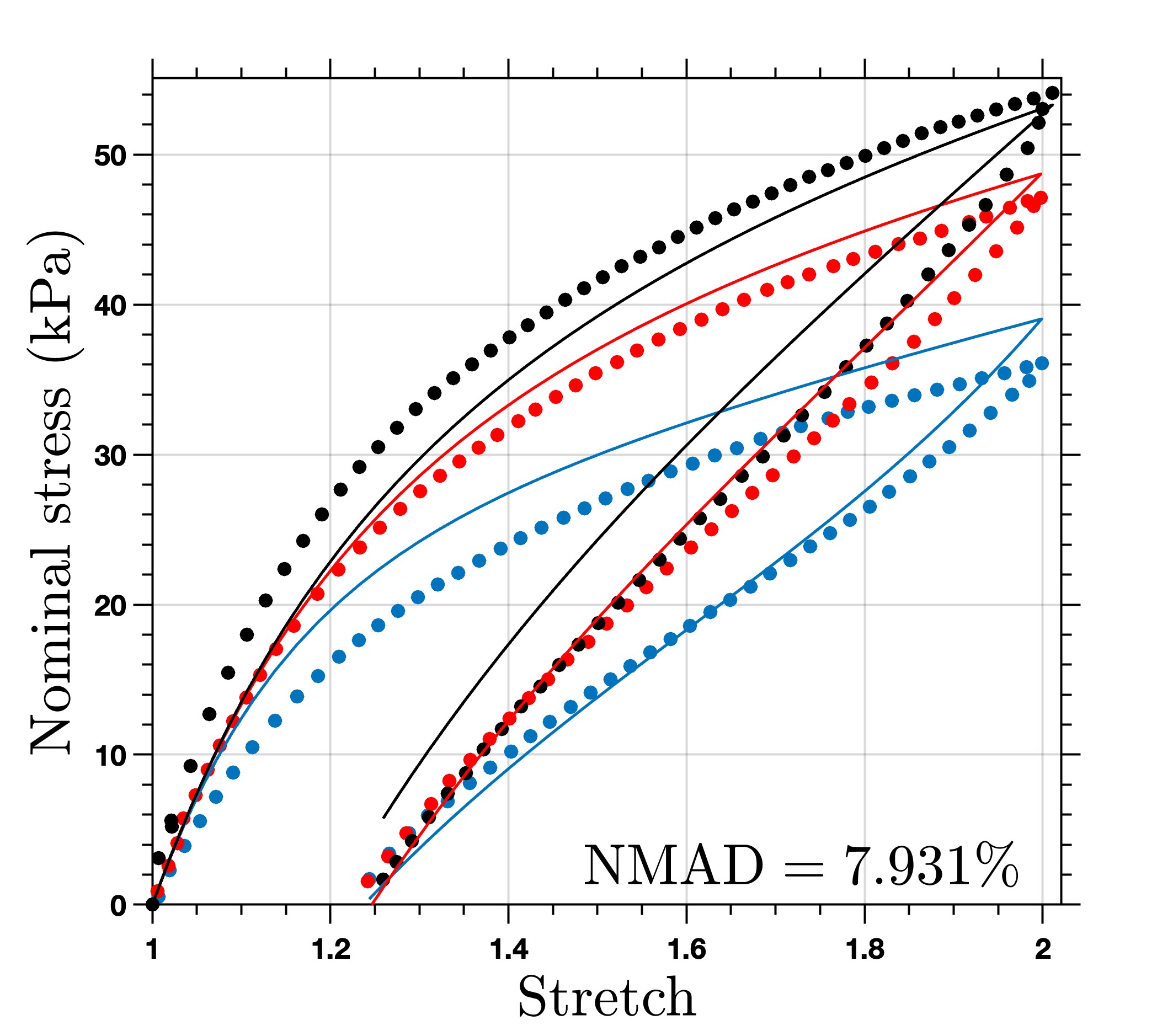} \\
(c) & (d)
\end{tabular}
\end{center}
\caption{Prediction results for the data of the maximum stretch of 2.0 using the parameters listed in Table \ref{table:parameters_micro}: (a) micromechanical model with the multiplicative decomposition, (b) micromechanical model with the Hencky strain, (c) micromechanical model with the Curnier-Rakotomanana strain, (d) FLV model with the Curnier-Rakotomanana strain. The experimental data correspond to stretch rates of $0.01$ s$^{-1}$ (blue), $0.03$ s$^{-1}$ (red), and $0.05$ s$^{-1}$ (black).}
\label{fig:results_micro_prediction_compare}
\end{figure}

The same fitting procedure is performed to evaluate the capability of the FLV model and the proposed micromechanical model, both using the Curnier-Rakotomanana strain. The results, shown in Figure \ref{fig:results_micro_compare} (c) and (d), indicate that both models exhibit robust performance in capturing mechanical behavior. The micromechanical model with the Curnier-Rakotomanana strain demonstrates the best fitting quality among the four considered models. The results of Figure \ref{fig:results_micro_compare} suggest that invoking the generalized strain may enhance the ability of the model to effectively represent nonlinear inelastic behavior.

This section concludes with predictions for three additional experimental data sets using the optimized parameters listed in Table \ref{table:parameters_micro}, obtained from the previous calibration processes. The data sets, with a maximum stretch of $2.0$ and three distinct stretch rates, are employed to validate the optimized parameters for the different models. The proposed model using the Hencky strain and the model based on the multiplicative decomposition exhibit similar NMAD values in their prediction results. The FLV model shows the largest NMAD, likely due to the limitations of the quadratic energy. The proposed micromechanical model with the Curnier-Rakotomanana strain achieves the lowest NMAD among all four cases. Both the calibration and prediction results suggest that the proposed framework offers an effective approach for characterizing complex material behaviors.

\subsection{Results of numerical simulation}
\label{sec:results_FEM}
We perform finite element analysis to demonstrate the effectiveness of the proposed integration scheme and further assess the proposed models. For the spatial discretization, we employ an inf-sup stable element pair based on non-uniform rational B-splines (NURBS) (\cite{Liu2019a}). The pressure field is interpolated using the tri-linear basis functions, while the velocity and displacement are interpolated by tri-quadratic $C^1$-continuous NURBS basis functions. These higher-order and higher-continuity basis functions enjoy both superior accuracy and robustness in large strain analysis. For the time integration, we utilize the generalized-$\alpha$ scheme, parameterized by a single parameter $\varrho_{\infty}$ (\cite{Jansen2000}), which represents the spectral radius of the amplification matrix at the highest mode. Its value is fixed to be $0$ in our numerical analysis. In the global Newton-Raphson iteration, we set the relative and absolute tolerances as $10^{-10}$ as the convergence criteria; in the local Newton-Raphson iteration, we adopt the relative and absolute tolerances as $10^{-12}$ as the convergence criteria. Unless otherwise specified, we use the meter-kilogram-second system of units. At least two different spatiotemporal meshes are used to ensure mesh independence of the reported results. In this study, we also considered the micromechanical model based on the multiplicative decomposition for comparison purposes. Its constitutive integration follows the algorithm outlined in \cite{Reese1998}, while the discretization of the balance equations is identical to that described above.

\begin{figure}
\begin{center}
\begin{tabular}{cc}
\includegraphics[width=0.35\linewidth, trim=0 20 170 110, clip]{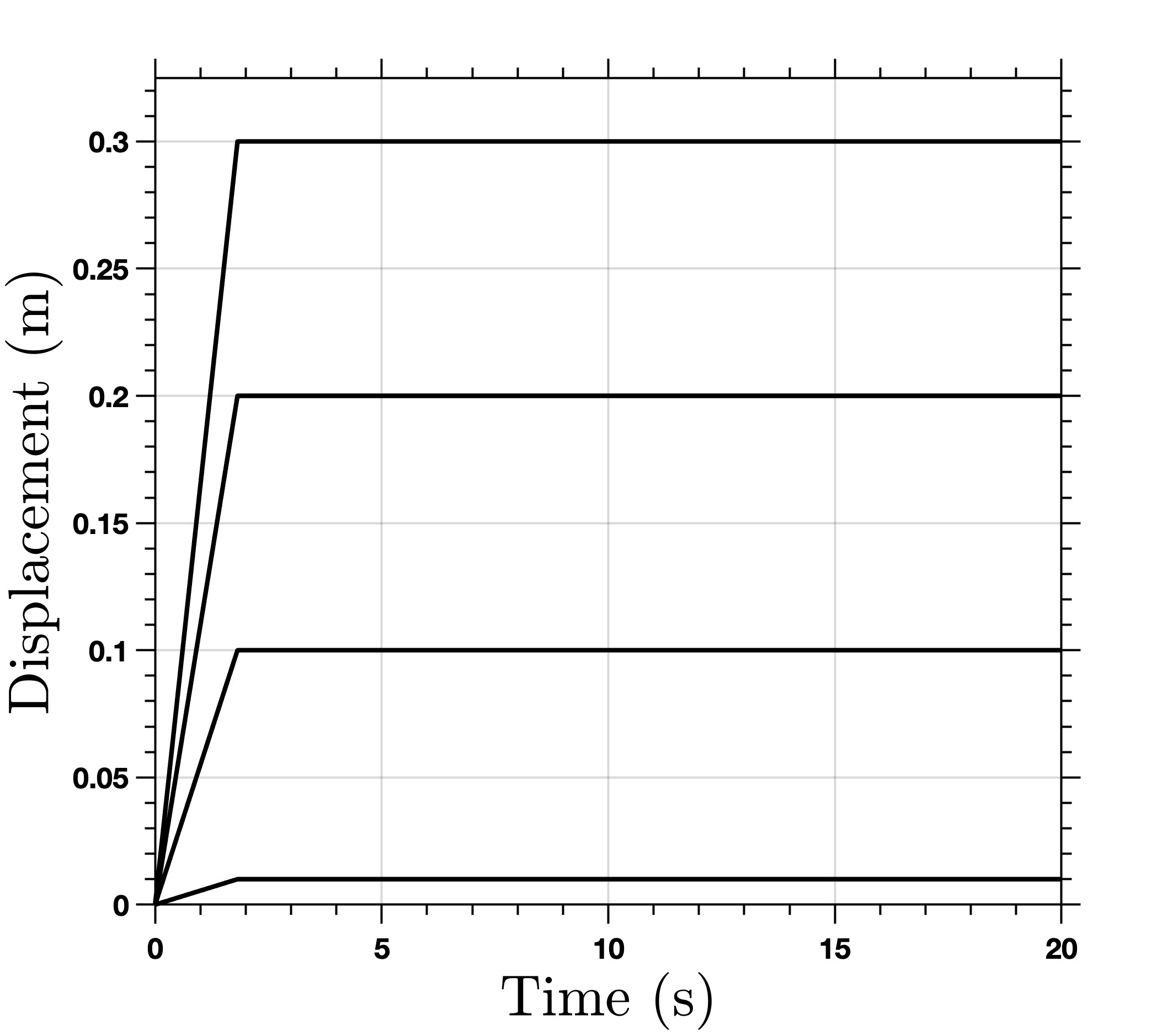} &
\includegraphics[width=0.35\linewidth, trim=0 20 170 110, clip]{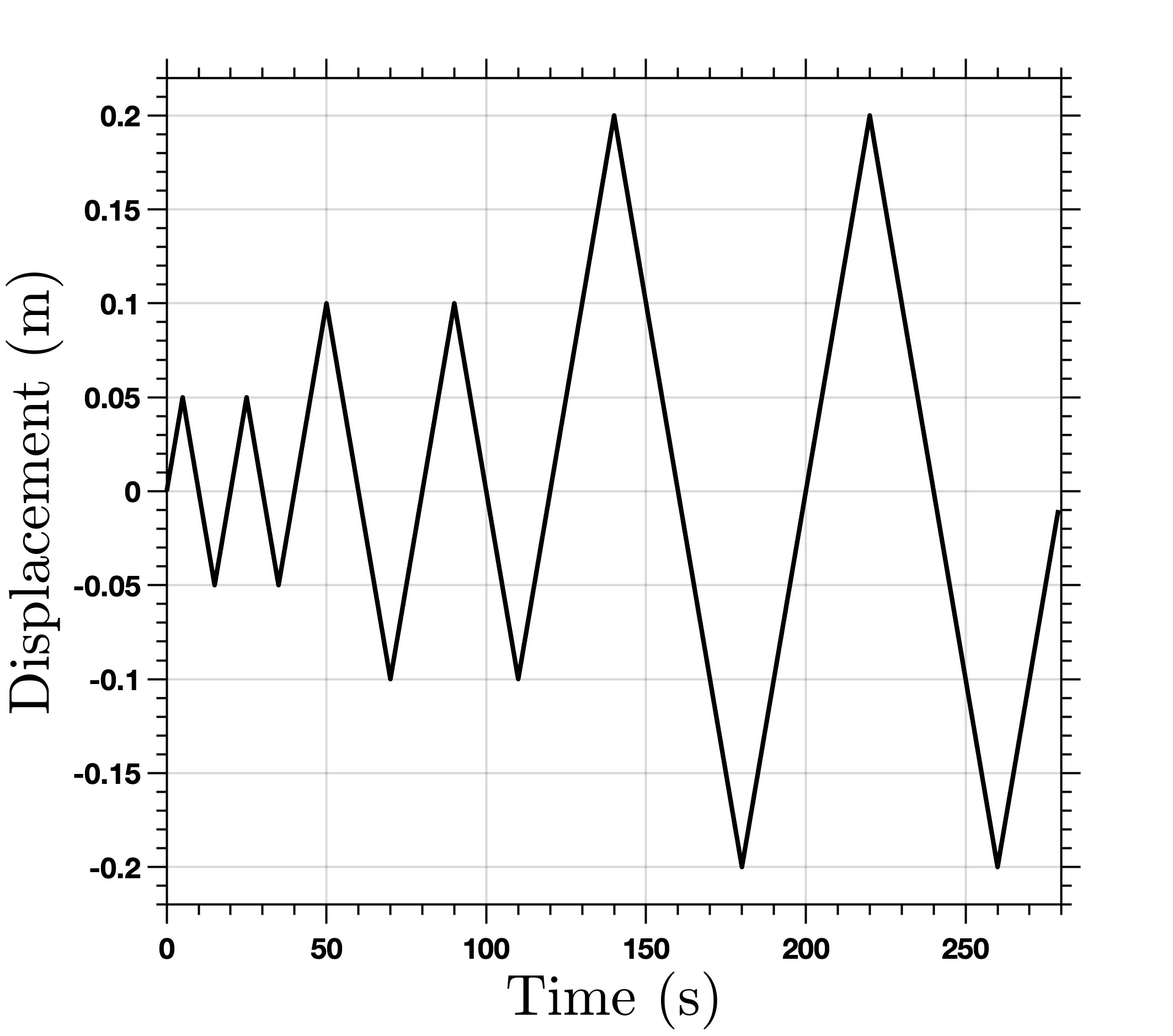}
\end{tabular}
\end{center}
\caption{Loading history: uniaxial tensile-relaxation test (left) and cyclic shear test (right).}
\label{fig:relax-loading-history}
\end{figure}

\paragraph{Uniaxial tensile-relaxation tests}
We perform uniaxial tensile-relaxation tests on a cube to investigate the dissipation behavior of the proposed viscoelasticity models. The size of the cube is $0.1\times 0.1 \times 0.1$, which is clamped at the bottom and subjected to a displacement loading in the vertical direction on its top surface. As shown in Figure \ref{fig:relax-loading-history}, the loading time spans from $t=0$ to $t=20$, and the displacement on the top surface increase from zero to $\{0.01, 0.1, 0.2,0.3\}$ in one second, respectively, and then holds constant for the remainder of the test. On the rest four boundary surfaces, stress-free boundary conditions are applied. The time step size is $\Delta t_n = 0.01$. In this study, we monitor the total dissipation $D_{n+1}$ of the cube over time, which is defined as
\begin{align*}
D_{n+1} := \int_{\Omega_{\bm X}} \Phi_{n+1} d\Omega_{\bm X}.
\end{align*}
For the proposed models, 
\begin{align*}
\Phi_{n+1} := \eta \left\lvert \frac{\bm E^{\mathrm v}_{n+1} -\bm E^{\mathrm v}_{n} }{\Delta t_n} \right\rvert^2,
\end{align*}
while for the model based on the multiplicative decomposition, $\Phi_{n+1} := \eta \left\lvert \bar{\bm d}^{\mathrm v}_{n+1} \right\rvert^2$.

\begin{figure}
\begin{center}
\begin{tabular}{cc}
\includegraphics[angle=0, trim=15 380 160 70, clip=true, scale=0.1]{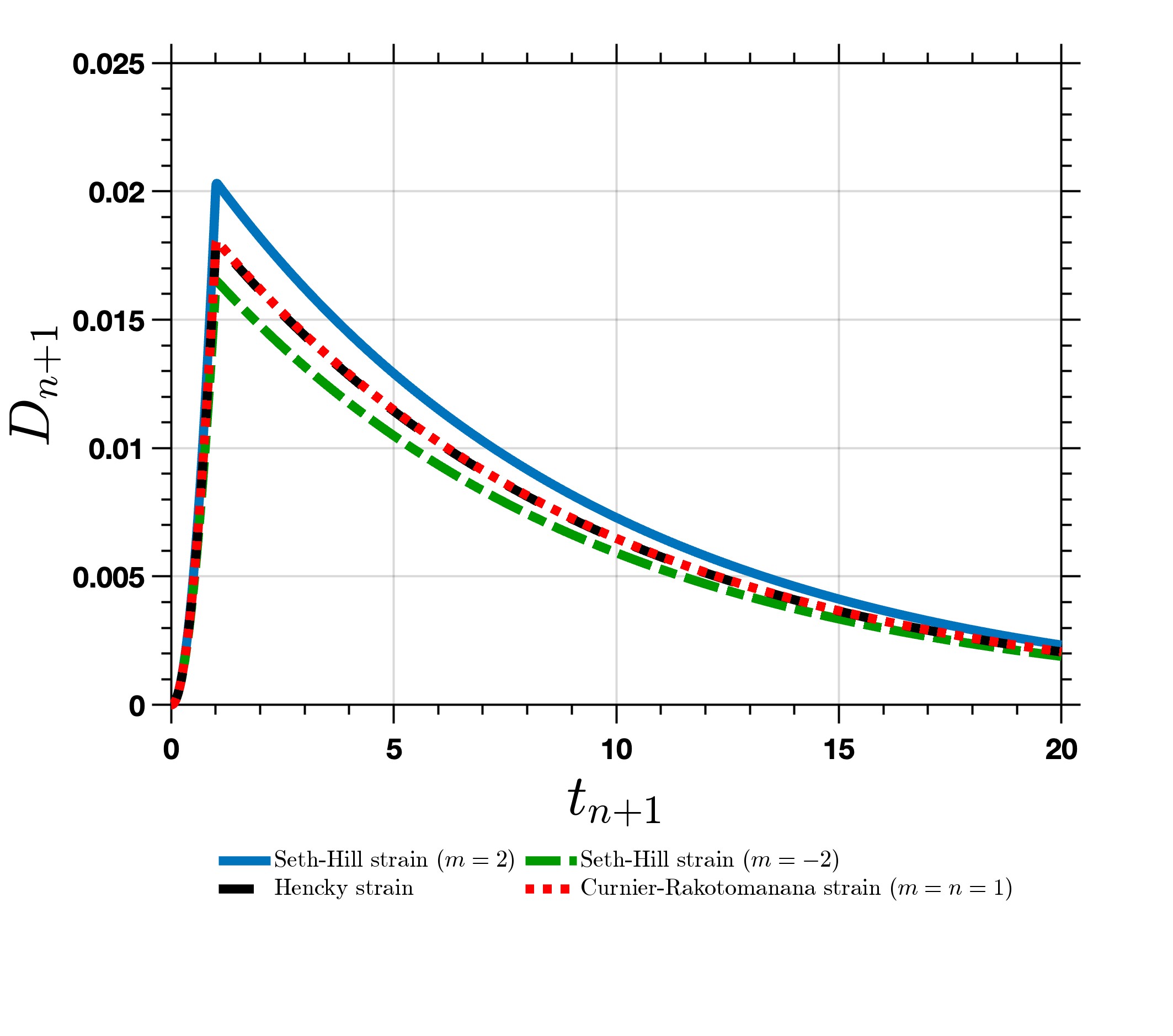} &
\includegraphics[angle=0, trim=15 380 160 70, clip=true, scale=0.1]{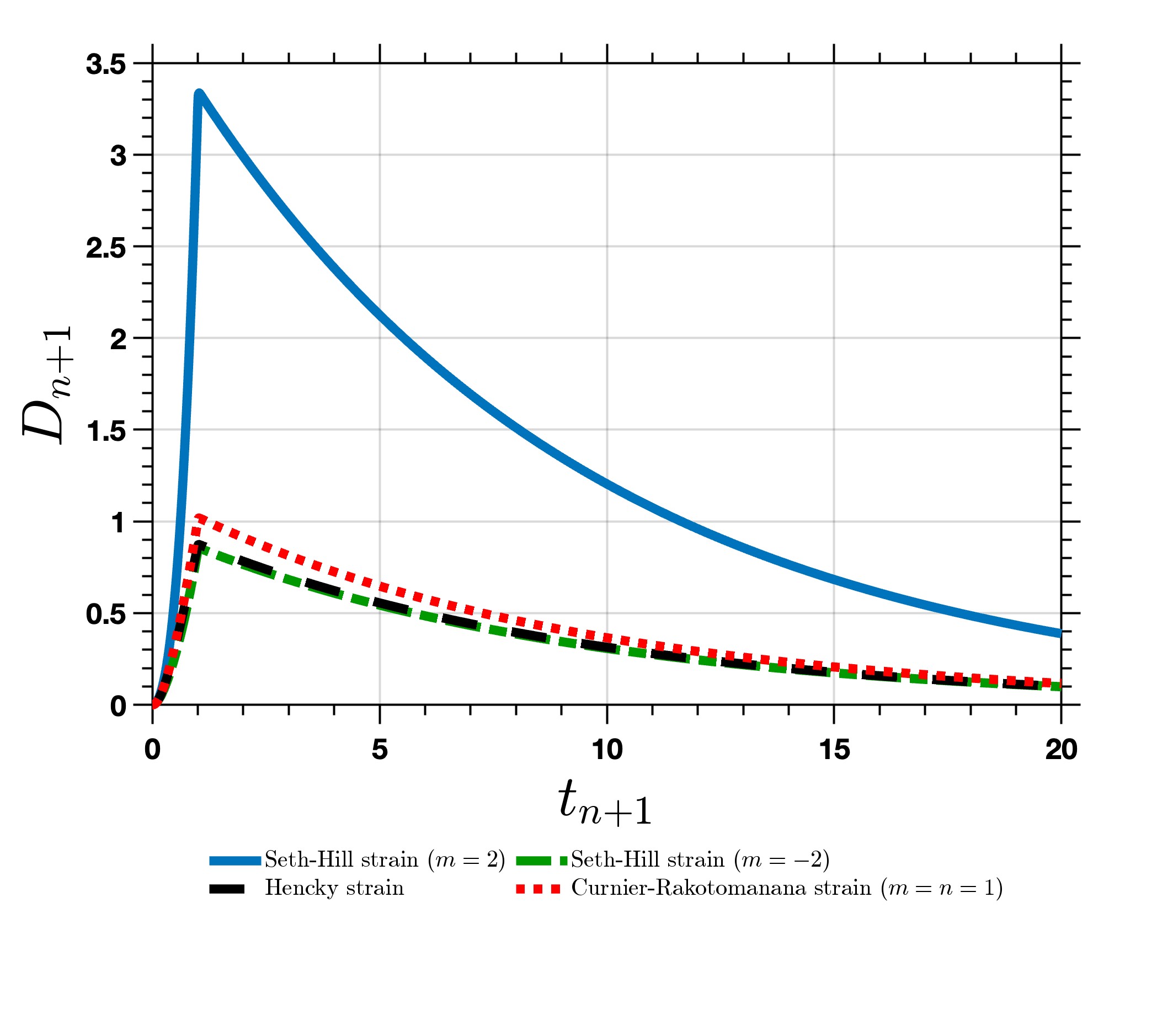} \\
(a) & (b) \\
\includegraphics[angle=0, trim=15 380 160 70, clip=true, scale=0.1]{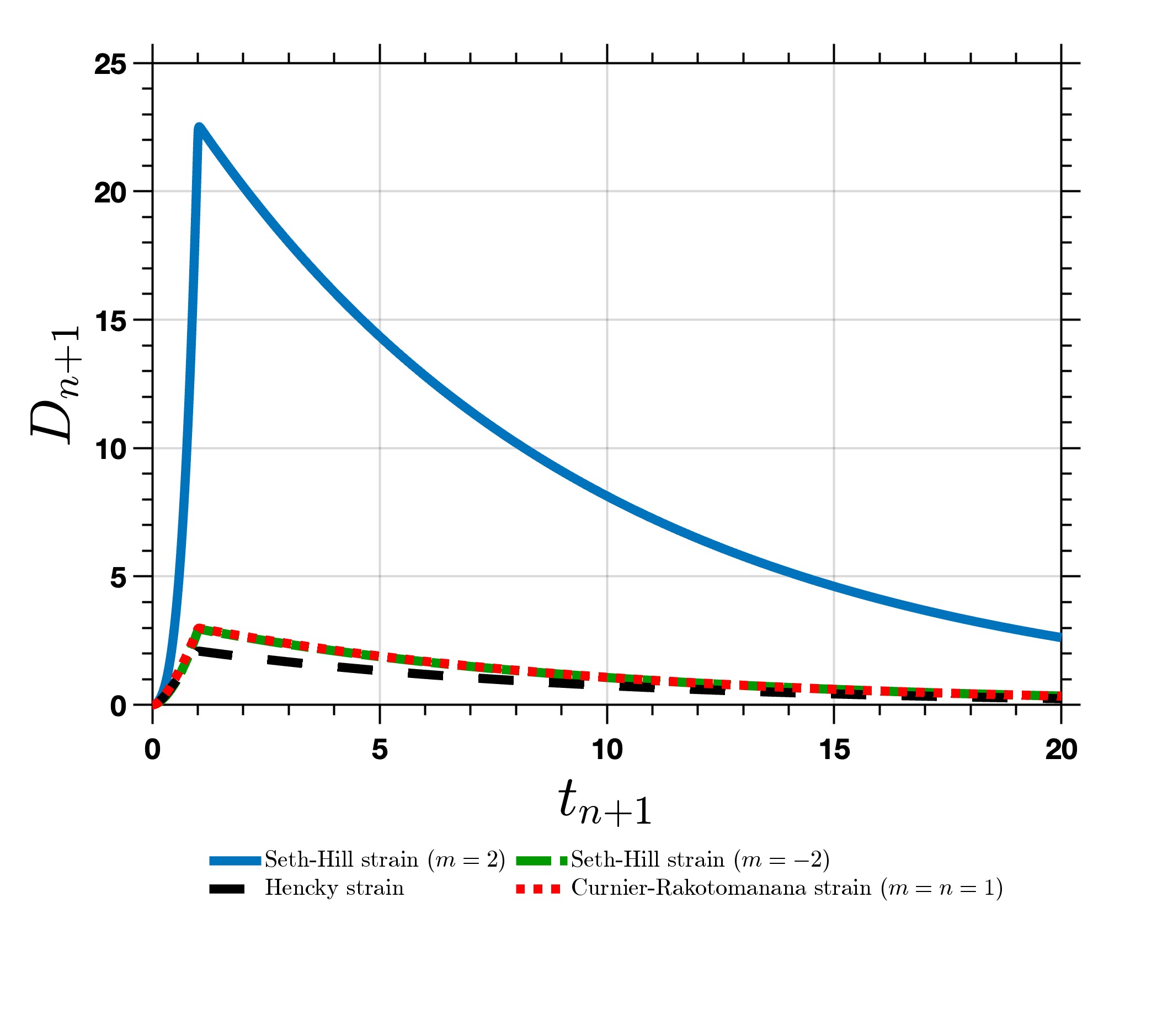} &
\includegraphics[angle=0, trim=15 380 160 0, clip=true, scale=0.1]{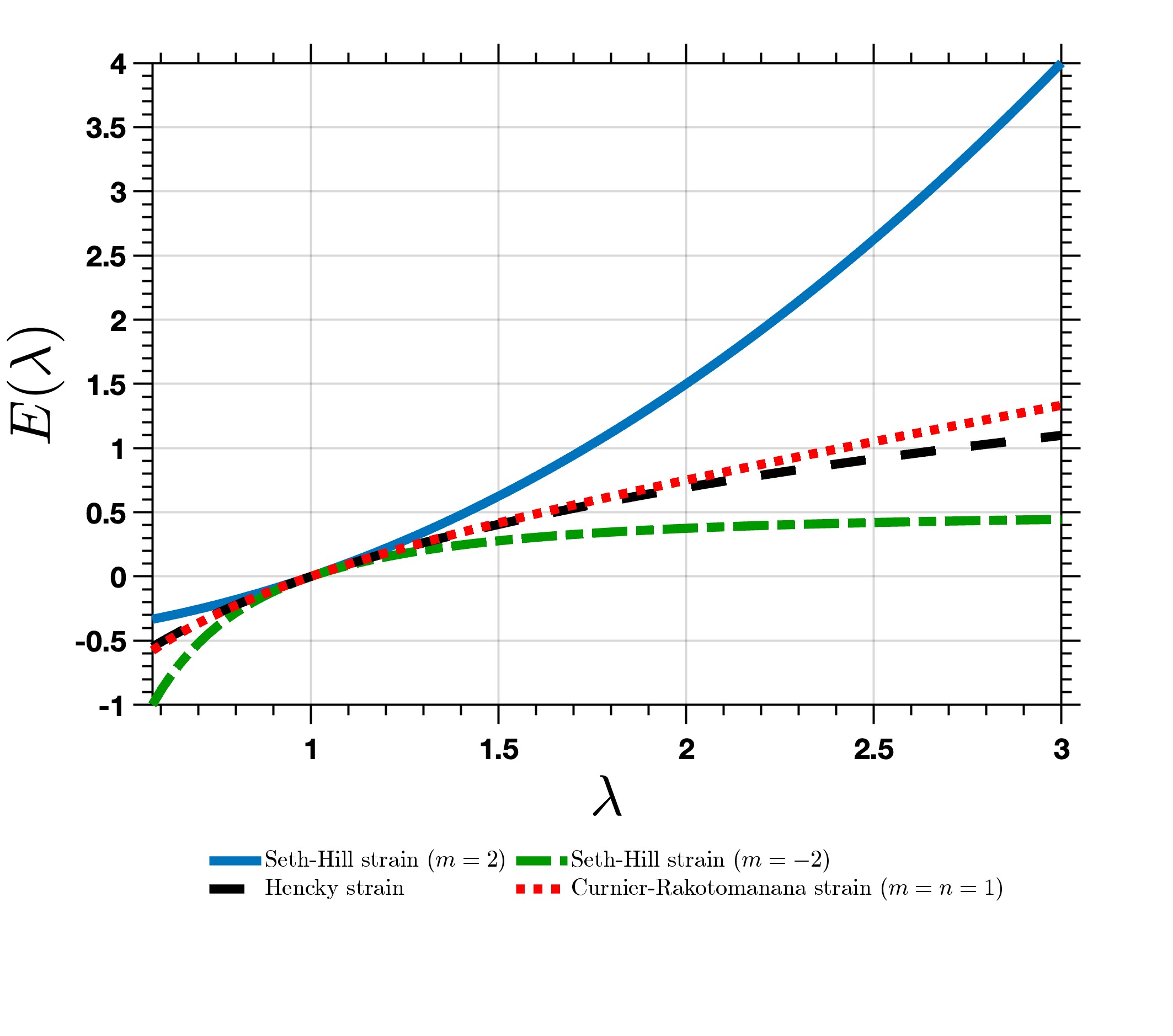} \\
(c) & (d) \\
\multicolumn{2}{c}{
\includegraphics[angle=0, trim=350 240 300 1500, clip=true, scale=0.21]{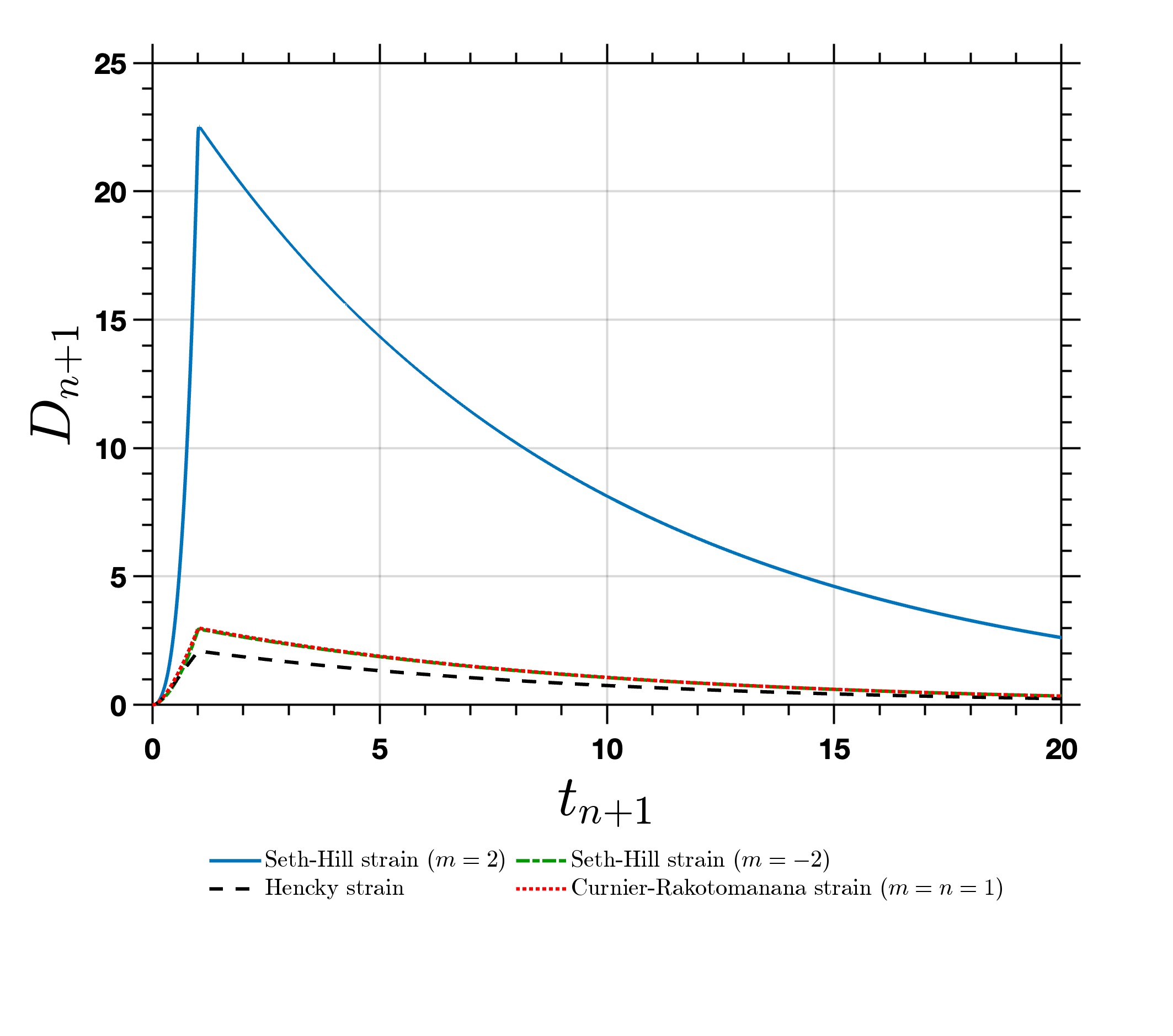}
}
\end{tabular}
\end{center}
\caption{The total dissipation over time for the FLV models with maximum stretches of (a) $1.1$, (b) $2.0$, and (c) $3.0$. The scale functions of the strains are depicted in (d).}
\label{fig:quad_tensile_relax}
\end{figure}

For the FLV model \eqref{eq:sec-results-flv-model}, we consider $N=M=1$, $\rho_0 = 1.0\times 10^3$, $\mu^{\infty} = 2.0 \times 10^4$, and the strain for the equilibrium part is given by the Curnier-Rakotomanana strain with $m=n=1$. For the non-equilibrium part, we adopt $\mu^\mathrm{neq} = 2.0 \times 10^4$, $\tau=17.5$, and $\tilde{\bm E}^\mathrm{neq}$ takes the following four different options: the Hencky strain, the Green-Lagrange strain, the Euler-Almansi strain, and the Curnier-Rakotomanana strain with $m=n=1$. The dissipation of the FLV models over time with different generalized strains is presented in Figure \ref{fig:quad_tensile_relax}. When the maximum stretch is relatively small, the models with different strains exhibit similar dissipation. Notably, the total dissipation of the models with the Hencky strain and Curnier-Rakotomanana strain are nearly identical, a trend also reflected in the similarity of their scale function curves. Indeed, the  Curnier-Rakotomanana strain with $m=n=1$ is also known as the Itskov-Ba\v{z}ant strain, which was designed by \cite{Bazant1998} to approximate the Hencky strain. As the maximum stretch increases to $2.0$ and $3.0$, the model using the Green-Lagrange strain demonstrates the largest dissipation. This behavior is likely due to the quadratic amplification effect of its scale function. The models employing the other three generalized strains display comparable dissipation behavior, which can be attributed to the similar shapes of their respective scale function within the considered stretch range.

\begin{figure}
\begin{center}
\begin{tabular}{cc}
\includegraphics[angle=0, trim=15 320 170 70, clip=true, scale=0.1]{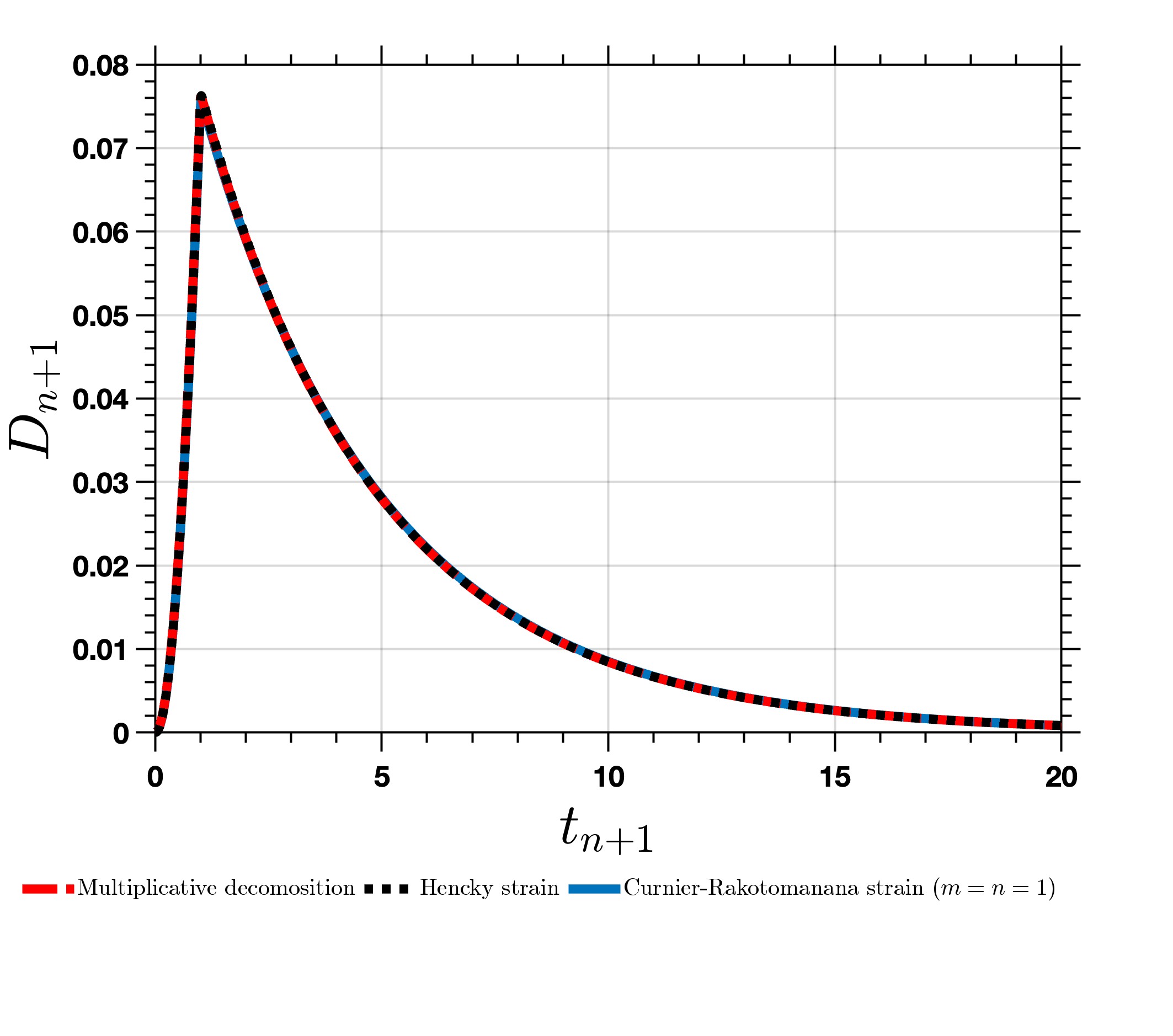} &
\includegraphics[angle=0, trim=15 320 170 70, clip=true, scale=0.1]{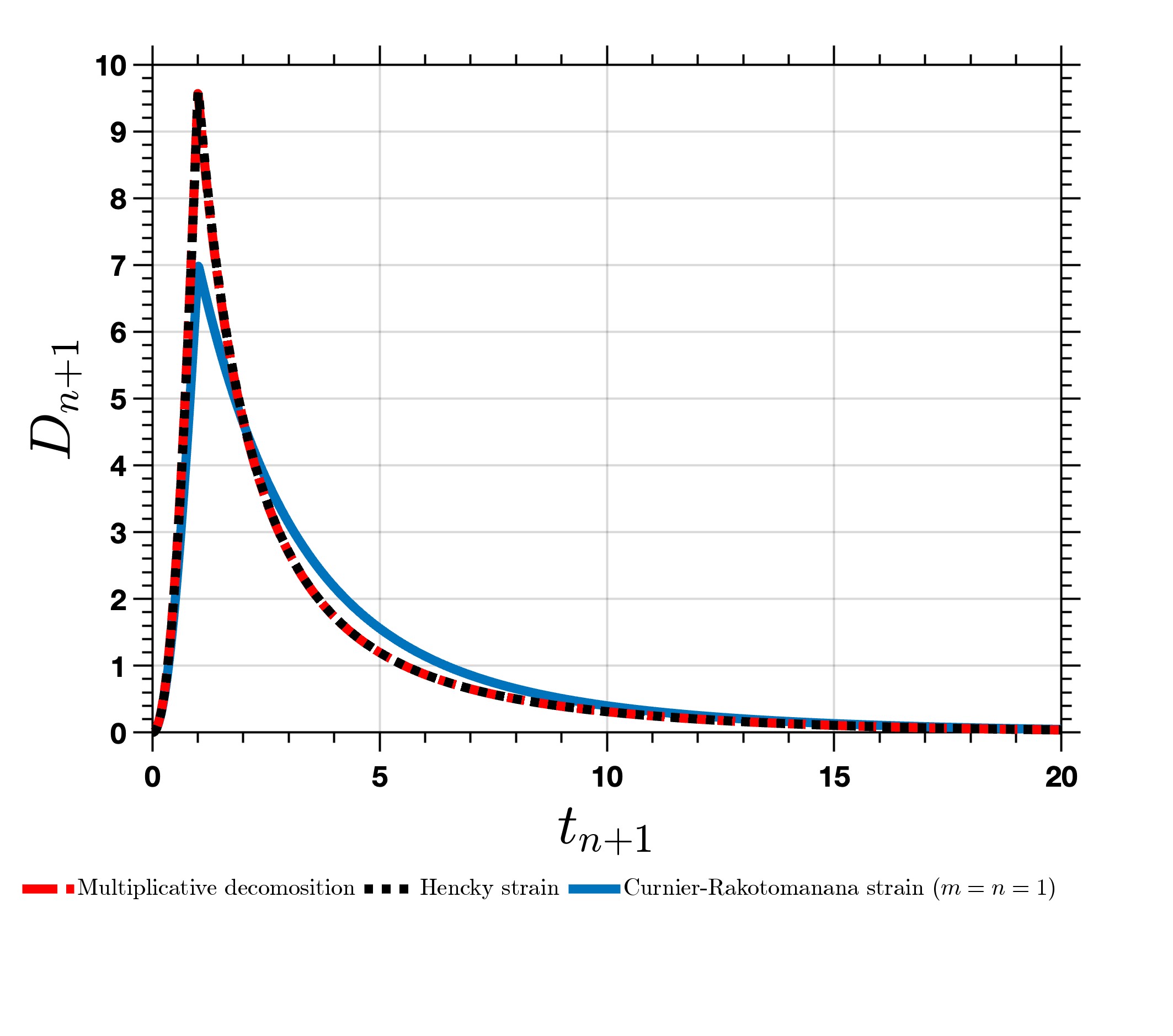} \\
(a) & (b) \\
\includegraphics[angle=0, trim=15 320 170 70, clip=true, scale=0.1]{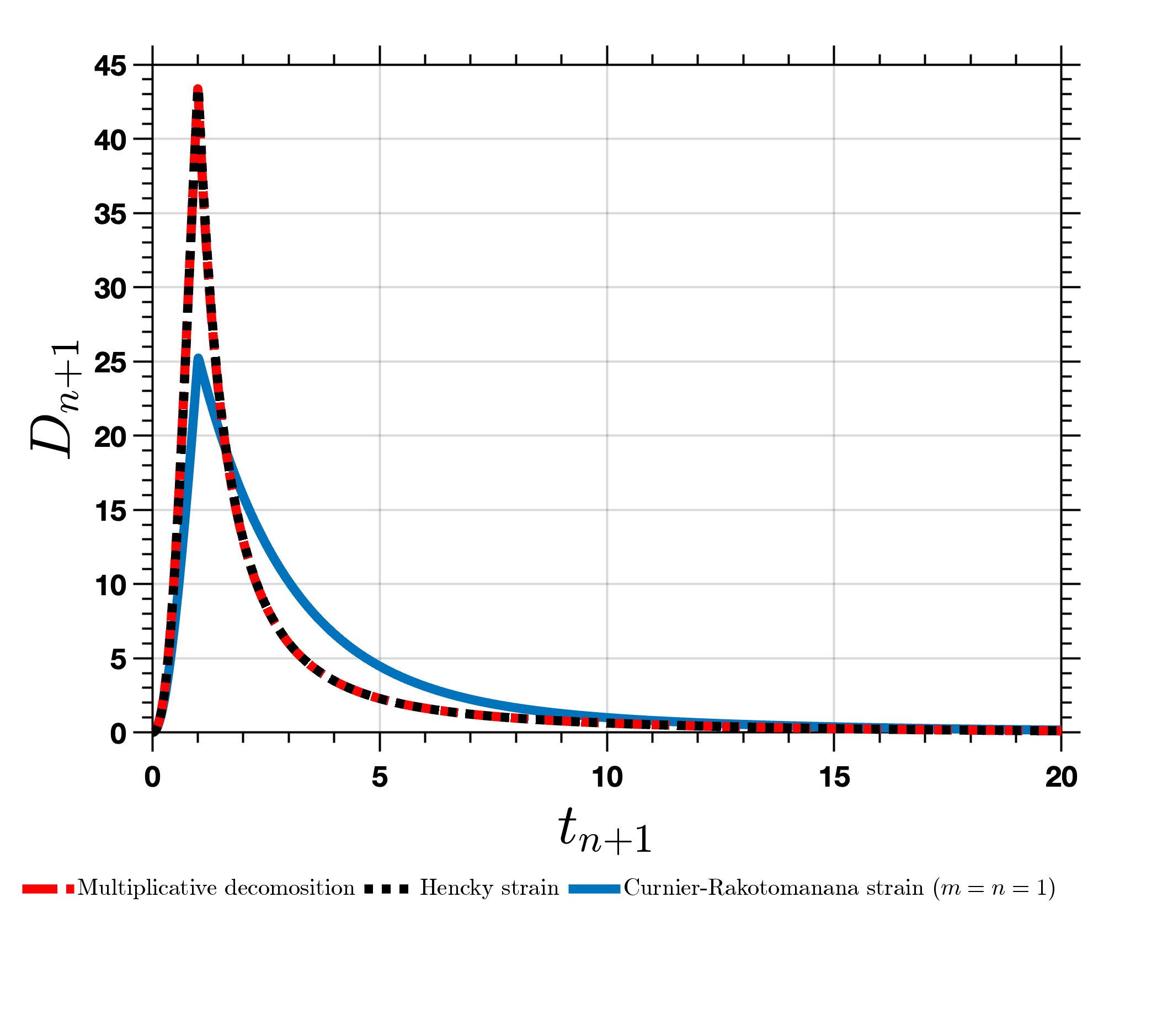} &
\includegraphics[angle=0, trim=15 320 170 70, clip=true, scale=0.1]{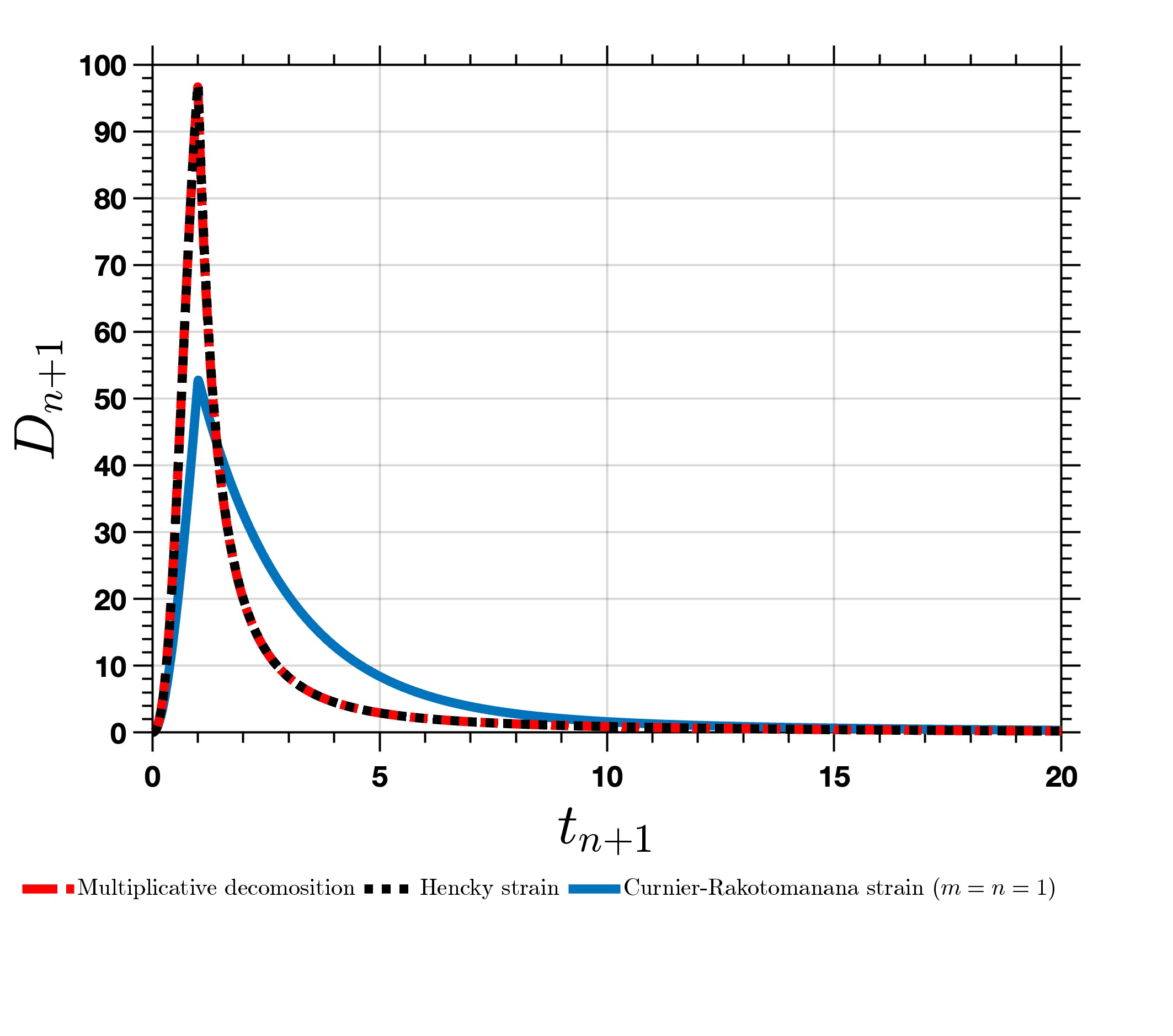} \\
(c) & (d) \\
\multicolumn{2}{c}{
\includegraphics[angle=0, trim=20 220 200 1550, clip=true, scale=0.21]{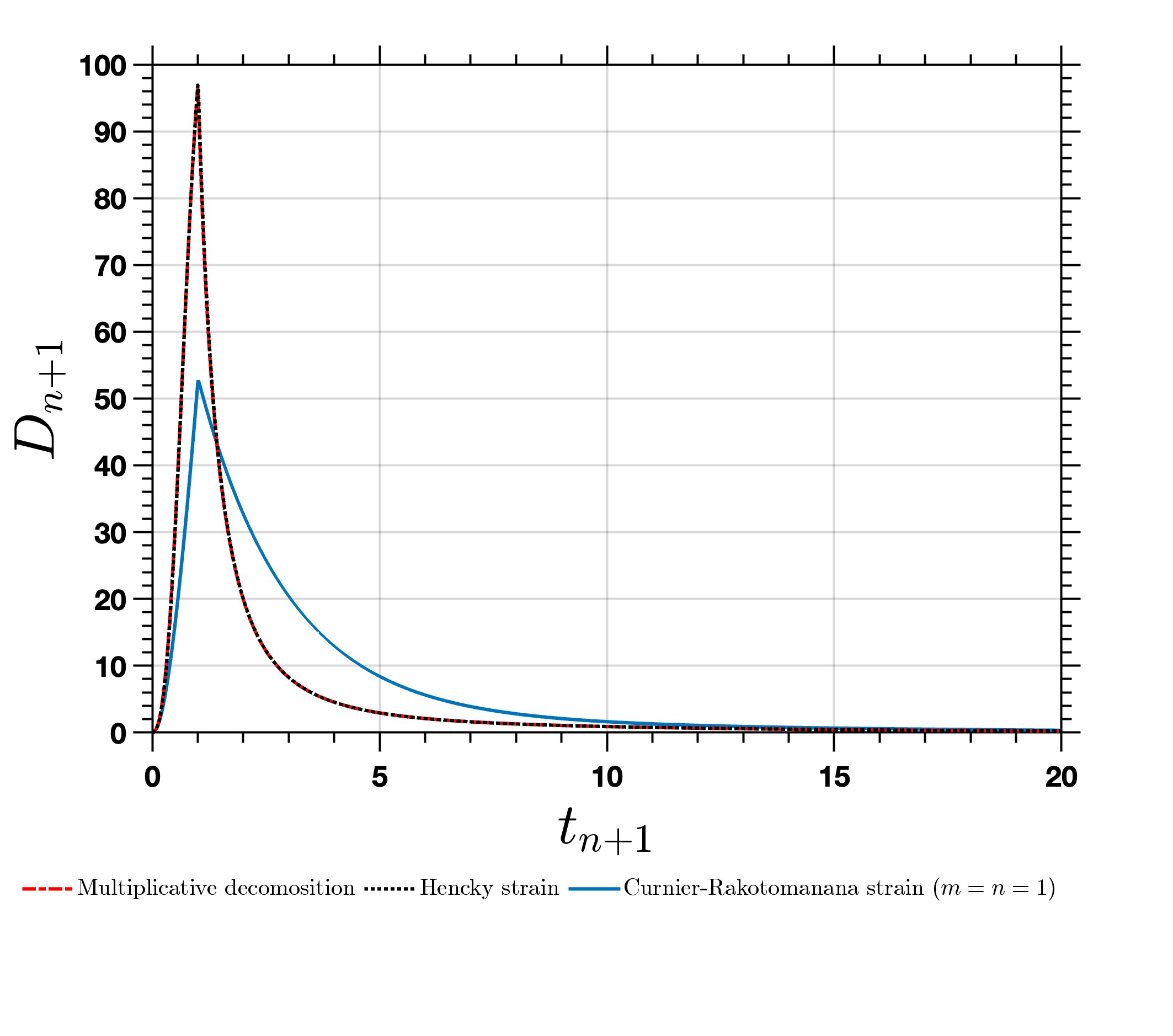}
}
\end{tabular}
\end{center}
\caption{The total dissipation over time for the micromechanical models with maximum stretches of (a) $1.1$, (b) $2.0$, (c) $3.0$, and (d) $4.0$.}
\label{fig:AB_tensile_relax}
\end{figure}

For the micromechanical models \eqref{eq:sec-results-mm-1}-\eqref{eq:sec-results-mm-2}, we set the material parameters as $\rho_0 = 1.0 \times 10^3$, $\mu^{\infty}=\mu^\mathrm{neq} = 2.0 \times 10^4$, $N^\infty= N^\mathrm{neq} = 150$. We consider the following three definitions of the tensor $\bm C^\mathrm{e}$ in the non-equilibrium branch. First, it is defined as $\bar{\bm F}^{\mathrm e \: T} \bar{\bm F}^{\mathrm e}$ based on the multiplicative decomposition; second, it is constructed from $\bm E^{\mathrm e} = \bm E - \bm E^{\mathrm v}$ using the Hencky strain; third, it is constructed from $\bm E^{\mathrm e} = \bm E - \bm E^{\mathrm v}$ using the Curnier-Rakotomanana strain with $m=n=1$. The dissipation of the micromechanical models over time is presented in Figure \ref{fig:AB_tensile_relax}. It can be observed that they dissipate faster compared to the FLV models. This observation aligns with findings reported by \cite{Gouhier2024}. All three models show nearly identical dissipation effects under small deformations. As the maximum stretch increases, the model based on the multiplicative decomposition and the proposed model with the Hencky strain exhibit almost identical dissipation curves, both showing higher values than the proposed model with the Curnier-Rakotomanana strain. The differences are attributed to the distinct shapes of their scale functions.

\begin{figure}
\begin{center}
\begin{tabular}{cc}
\includegraphics[angle=0, trim=10 340 130 40, clip=true, scale=0.1]{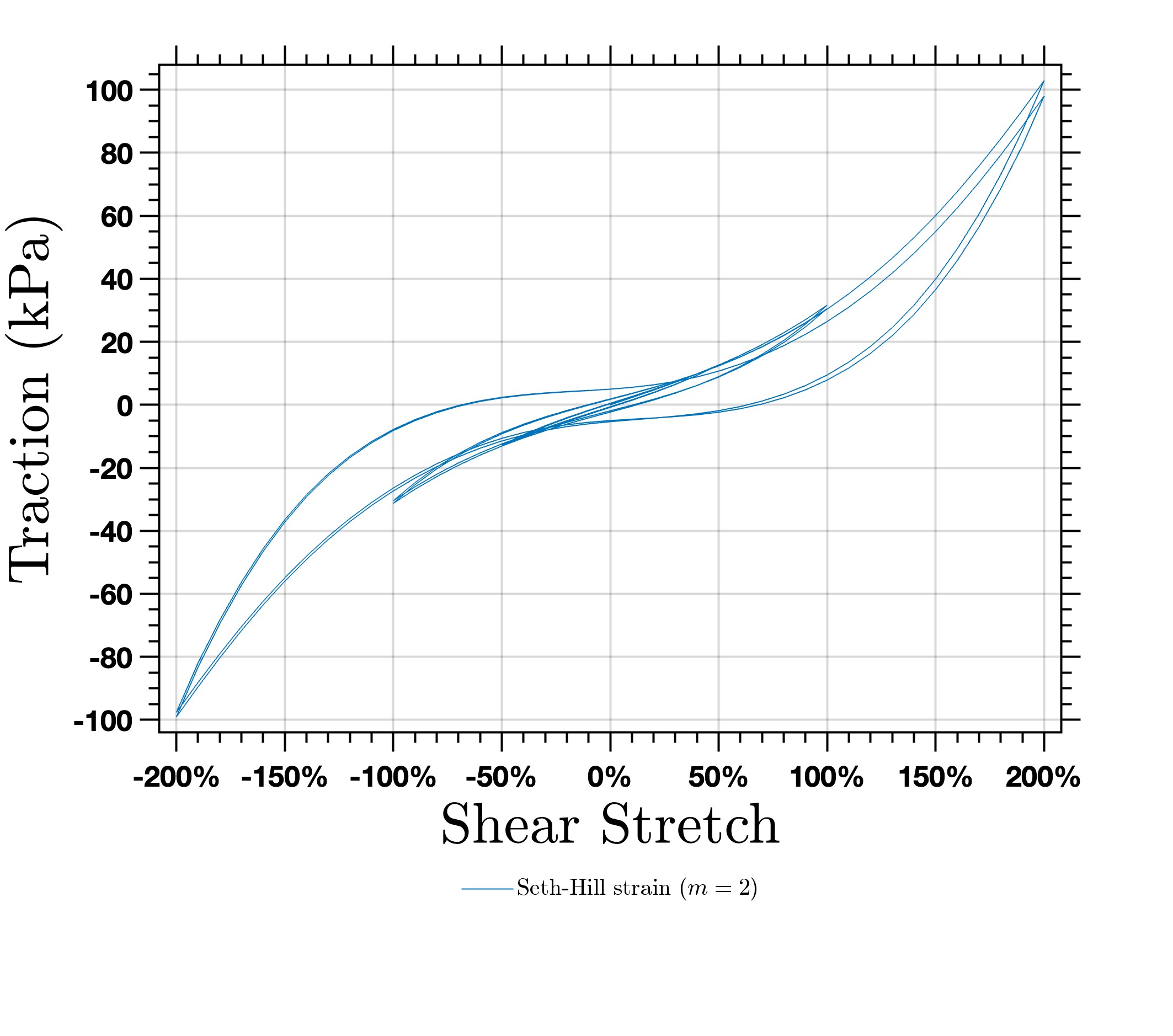} &
\includegraphics[angle=0, trim=10 340 130 40, clip=true, scale=0.1]{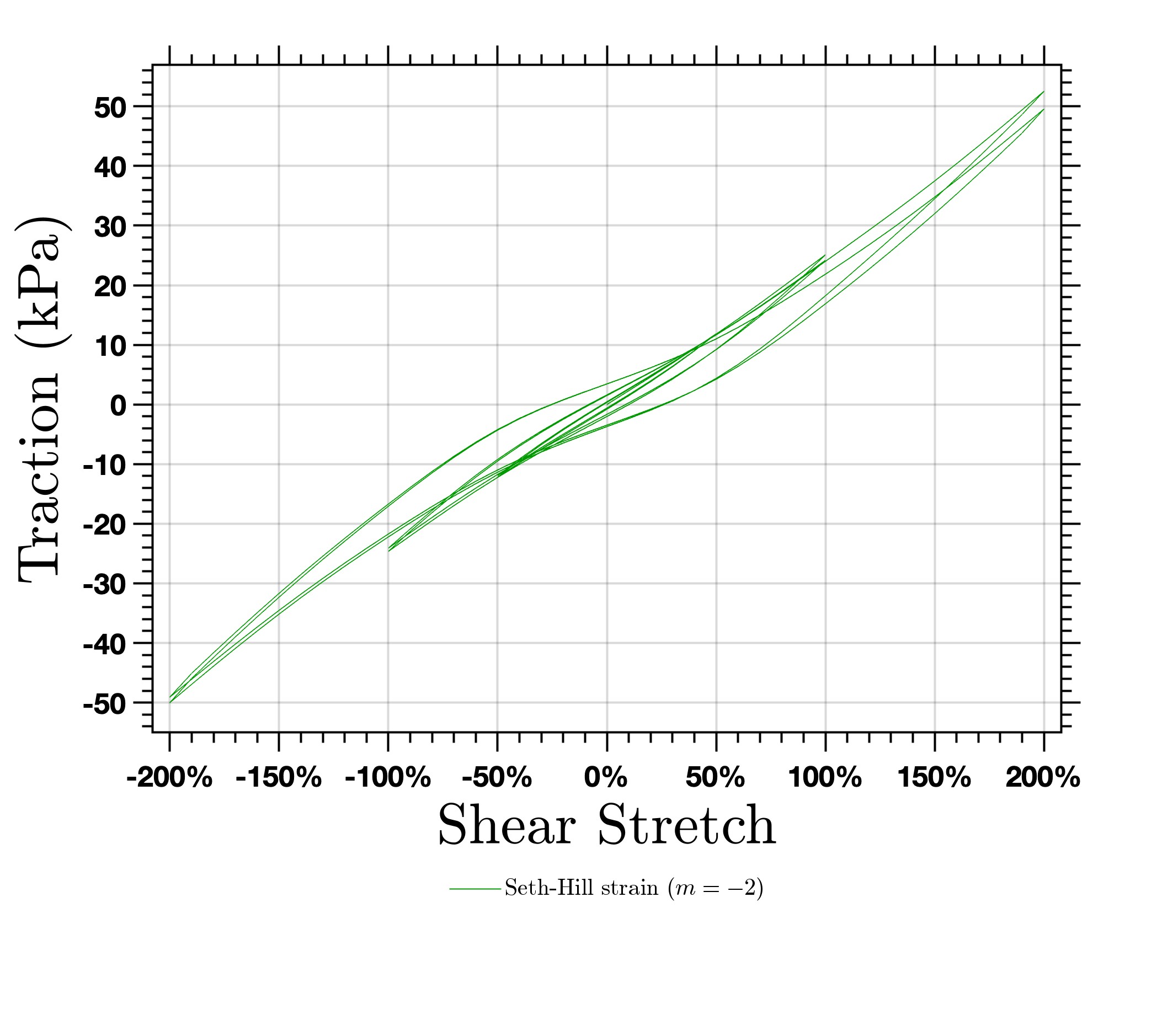} \\
(a) & (b) \\
\includegraphics[angle=0, trim=10 340 130 40, clip=true, scale=0.1]{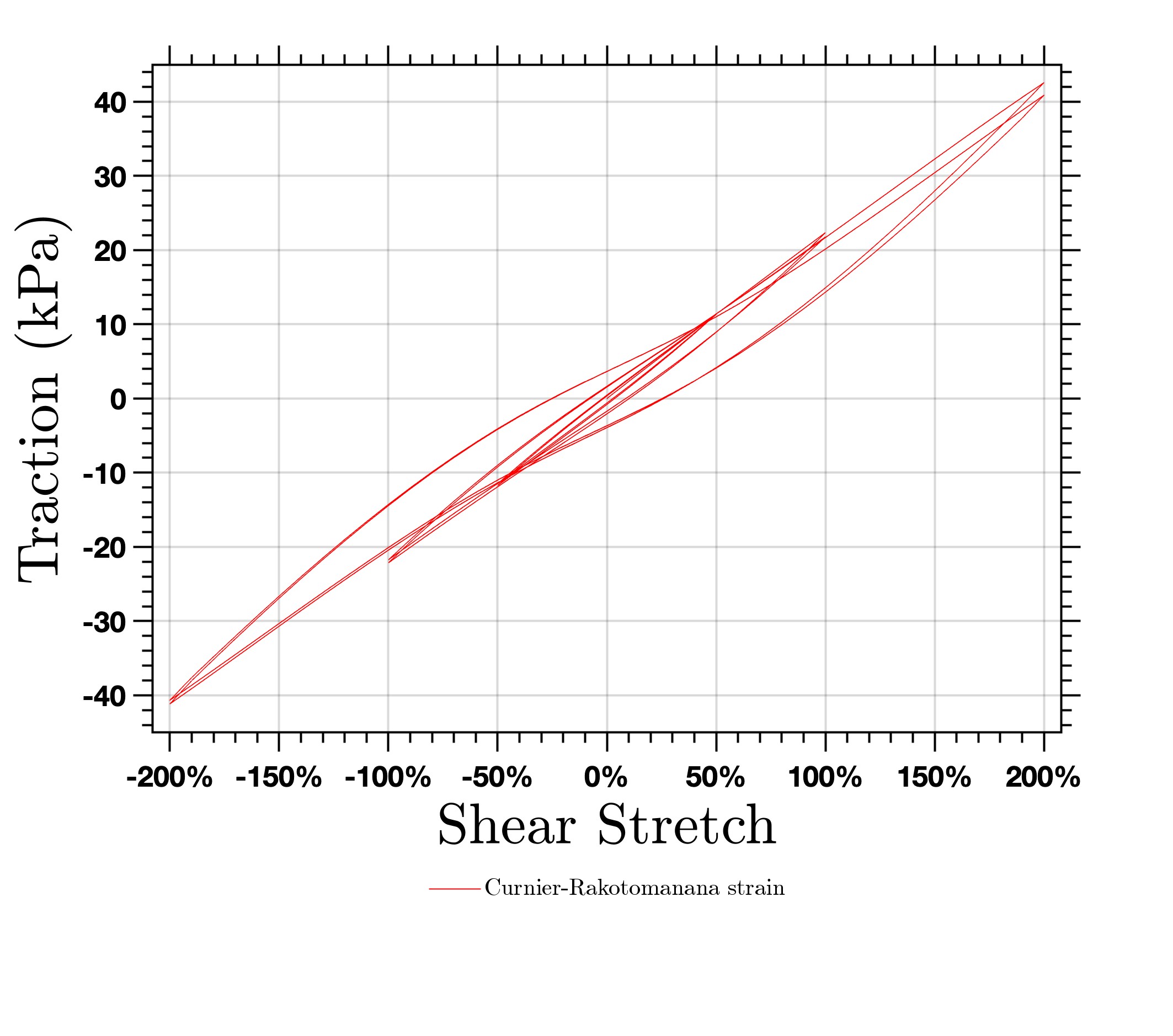} &
\includegraphics[angle=0, trim=10 340 130 40, clip=true, scale=0.1]{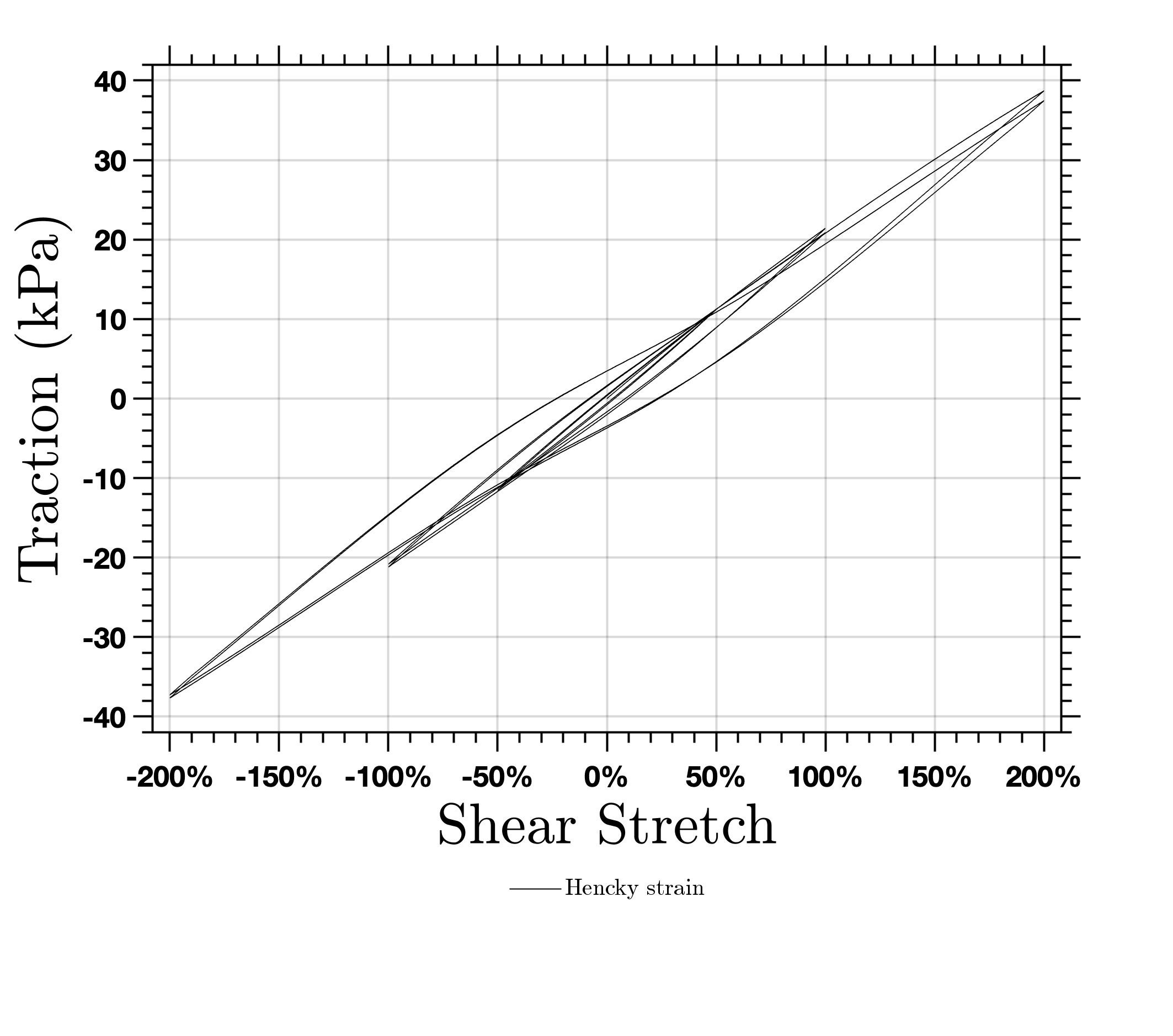} \\
(c) & (d) 
\end{tabular}
\end{center}
\caption{Traction with respect to shear stretch for FLV models: (a) Green-Lagrange strain, (b) Euler-Almansi strain, (c) Curnier-Rakotomanana strain ($m=n=1$), and (d) Hencky strain.}
\label{fig:quad_cyclic_shear}
\end{figure}

\begin{figure}
\begin{center}
\begin{tabular}{c}
\includegraphics[angle=0, trim=10 340 130 40, clip=true, scale=0.12]{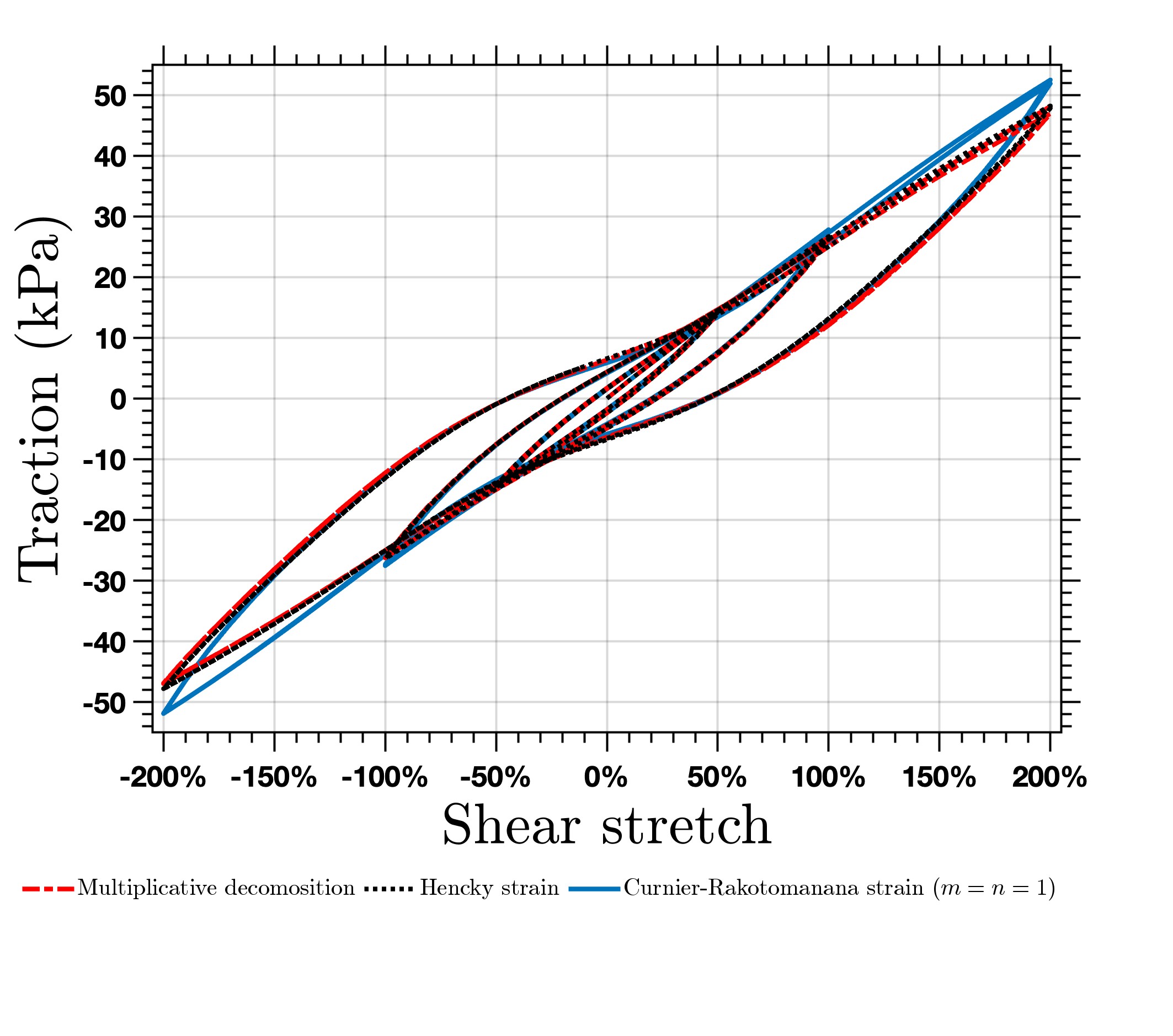} \\
\includegraphics[angle=0, trim=0 240 180 1550, clip=true, scale=0.2]{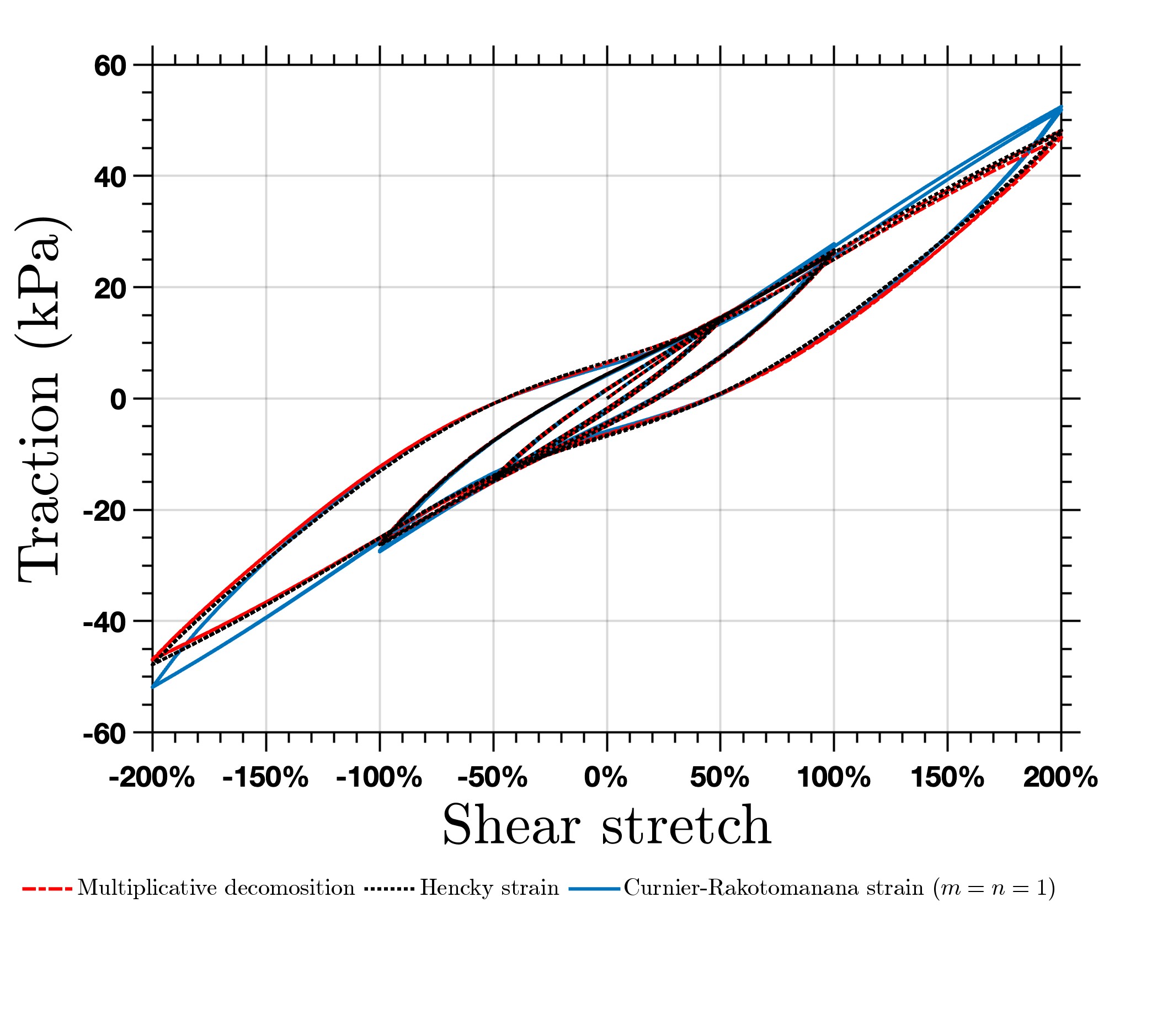}
\end{tabular}
\end{center}
\caption{Traction with respect to shear stretch for the micromechanical models.}
\label{fig:AB_cyclic_shear}
\end{figure}

\paragraph{Cyclic shear tests}
We conduct cyclic shear tests with the geometric configuration and material model settings identical to those used in the tensile-relaxation tests. The bottom boundary is clamped, while the top surface is subjected to horizontal displacement loading, and the loading-unloading cycle is illustrated in Figure \ref{fig:relax-loading-history}. Traction on the top surface is monitored, which is defined as the total force in the loading direction divided by the surface area. The responses of the FLV models are shown in Figure \ref{fig:quad_cyclic_shear}. When the stretch is small (less than 100\%), all four models exhibit almost linear elastic behavior, as indicated by the narrow hysteresis loops and nearly linear stress-strain curve. When the stretch increases to 200\%, the dissipation becomes more evident in the hysteresis loops. In particular, the curve associated with the Green-Lagrange strain gets steeper at large strains with the largest hysteresis loop area. The other three models exhibit similar hysteresis loop curves with no significant stiffening in their material response. The observation is consistent with their scale functions illustrated in Figure \ref{fig:quad_tensile_relax} (d).

Figure \ref{fig:AB_cyclic_shear} presents the hysteresis loops of the micromechanical models, indicating more pronounced nonlinear dissipation. The model based on the multiplicative decomposition and the proposed model with the Hencky strain produce nearly identical loop curves, indicating similar model behavior under shear loading. The proposed model with Curnier-Rakotomanana strain exhibits slightly higher stiffness at the largest shear stretch level. Other than that, it is almost identical to the other two models. The similarity is attributed to the fact that the Curnier-Rakotomanana strain with $m=n=1$ approximates the Hencky strain. 

\section{Conclusion}
\label{sec:conclusion}
In this work, we construct a finite viscoelasticity modeling framework. The adopted modeling assumption is inspired by the kinematic decomposition proposed by \cite{Green1965} in plasticity. Interestingly, this assumption in viscoelasticity guarantees the relaxation of the non-equilibrium stress in the thermodynamic equilibrium limit. The developed theory does not involve an imaginary intermediate configuration, which has been a controversial concept induced by the widely employed multiplicative decomposition of the deformation gradient (\cite{Lee1969}). In this regard, the proposed kinematic decomposition makes the constitutive theory fundamentally different from existing theories based on the multiplicative decomposition. If both $\bm E$ and $\bm E^{\mathrm v}$ are chosen as the Hencky strains, one may show that the proposed decomposition is equivalent to the multiplicative decomposition under certain deformation states (\cite[Appendix~C]{Liu2024}). This insight is corroborated by our model calibration and finite element analysis results. We may therefore view the proposed kinematic decomposition as a generalization or enrichment of the existing multiplicative decomposition, considering there are infinitely many choices of the generalized strains. The kinematic decomposition thus offers a more flexible mechanism for characterizing the viscous deformation, resulting in a constitutive theory that allows one to adjust and calibrate the decomposition. The modeling framework is further refined by considering coercive strains, which enables us to introduce the elastic strain $\bm E^{\mathrm e}$ as $\bm E - \bm E^{\mathrm v}$ locally. It subsequently leads to the elastic deformation tensor $\bm C^{\mathrm e}$, from which we may construct the configurational free energy by leveraging existing designs of hyperelastic free energy. In particular, we exploited the eight-chain model of \cite{Arruda1993} in this work to characterize the response in the non-equilibrium branch.

A particularly interesting instantiation is a family of models based on the quadratic form of the configurational free energy, which yields linear evolution equations for the internal variables. We therefore term it the family of finite linear viscoelasticity models. They characterize the nonlinear material response through the generalized strain $\bm E$. The evolution equation and its hereditary integral are analogous to those of the identical polymer chain model (\cite{Govindjee1992,Holzapfel1996}), which can be conveniently integrated by a single-step recurrence formula. If we further claim the existence of $\bm \Gamma \in$ Sym(3)$_{+}$, the finite linear viscoelasticity models recover the models of \cite{Green1946} and \cite{Simo1987}  by setting $\bm E^{\mathrm v}$ as the Euler-Almansi and Green-Lagrange strain of $\bm \Gamma$, respectively. This reveals the fundamental connection between those established models and offers a generalization of them.

In the calibration results based on the experimental data of VHB 4910, it is found that both the fitting and prediction qualities are improved when the strain parameters are subject to optimization in both the finite linear viscoelasticity models and micromechanical models. When the Crunier-Rakotomanana strain is utilized, the performance of the finite linear viscoelasticity model becomes comparable to that of the micromechanical models. Considering its simplicity in constitutive integration, the generalized strains make the finite linear viscoelasticity model both effective and appealing. The results demonstrate the overall effectiveness of the proposed framework. Furthermore, we examined different models based on varying kinematic assumptions through finite element analysis. The finite linear viscoelasticity models dissipate more slowly than the micromechanical models. Their dissipation and shape of hysteresis loops are closely related to the form of the strains.

We believe that a new door has been opened in modeling viscoelasticity, or inelasticity in general. It is worthwhile to further investigate the performance of the finite linear viscoelasticity models, considering that there are many possibilities in terms of the strain choice. The linear evolution equation makes it particularly attractive in practice. Additionally, the recently developed micromechanical models offer further potential for exploration. The viscous behavior considered in this work is grounded in the reciprocal principle of Onsager, leading to a relatively simple model for the viscous mechanism. The non-Newtonian effect warrants further investigation within the framework of generalized thermodynamics. Finally, the Mullins effect and material anisotropy will be incorporated into the framework to model more realistic material behaviors.

\section*{Acknowledgements}
This work is supported by the National Natural Science Foundation of China [Grant Numbers 12472201, 12172160], Shenzhen Science and Technology Program [Grant Number JCYJ20220818100600002], Southern University of Science and Technology [Grant Number Y01326127], and the Department of Science and Technology of Guangdong Province [2021QN020642]. Computational resources are provided by the Center for Computational Science and Engineering at the Southern University of Science and Technology.

\appendix
\section{The volumetric free energies}
\label{appendix:vol-energy}
Although the forms of $G_{\mathrm{vol}}$ listed in Table \ref{table:vol_energy} offer certain simple and appealing constitutive relations, it needs to be pointed out that it is not always feasible to obtain an analytical form of $G_{\mathrm{vol}}$ from the existing forms of $\Psi_{\mathrm{vol}}$ in the literature. Examples include the ANSYS 2000 model, HN03 model, and O72 model listed in Table \ref{table:vol_energy_expansion}. However, one may still examine the behavior of their corresponding $G_{\mathrm{vol}}$ by constructing a series representation. From the series representations, one may notice that all models can be viewed as perturbations of the quadratic model with higher-order terms. In Figure \ref{fig:illustration-vol-energy}, different volumetric models are depicted in terms of the Helmholtz energy $\Psi_{\mathrm{vol}}$ as a function of $J$ and the Gibbs energy $G_{\mathrm{vol}}$ as a function of $P/\kappa$.

\begin{table}[htbp]
\begin{center}
\tabcolsep=0.19cm
\renewcommand{\arraystretch}{1.6}
\begin{tabular}{P{4.5cm} P{4.6cm} P{6.5cm}  }
\hline
& $\Psi_{\mathrm{vol}} / \kappa$ & Series representation of $G_{\mathrm{vol}}/\kappa$  \\
\hline
Incompressible model & - & $\frac{P}{\kappa}$ \\
Quadratic model \cite{Peng1975} &  $(J-1)^2/2$ & $\frac{P}{\kappa} - \frac{P^2}{2\kappa^2}$ \\
ST91 model \cite{Simo1991} & $\frac{1}{4} \left( J^2 -2\ln(J) - 1 \right)$ & $\frac{P}{\kappa} - \frac{P^2}{2\kappa^2} + \frac{P^3}{6\kappa^3} + \mathcal O(\frac{P^5}{\kappa^5})$ \\
M94 model \cite{Miehe1994} & $(J - \ln(J) - 1)$ & $\frac{P}{\kappa} - \frac{P^2}{2\kappa^2} + \frac{P^3}{3\kappa^3} - \frac{P^4}{4\kappa^4} + \mathcal O(\frac{P^5}{\kappa^5})$   \\
L94 model \cite{Liu1994} & $(J \ln(J)-J+1)$ & $\frac{p}{\kappa} - \frac{P^2}{2\kappa^2} + \frac{P^3}{6\kappa^3} - \frac{P^4}{24\kappa^4} + \mathcal O(\frac{P^5}{\kappa^5})$  \\
ANSYS 2000 model \cite{Hartmann2003} & $\frac{1}{32}(J^2 - J^{-2})^2$ & $\frac{P}{\kappa} - \frac{P^2}{2\kappa^2} + \frac{P^3}{2\kappa^3} + \mathcal O(\frac{P^5}{\kappa^5})$ \\
HN03 model \cite{Hartmann2003} & $\frac{1}{50}(J^5+J^{-5} - 2)$ & $\frac{P}{\kappa} - \frac{P^2}{2\kappa^2} + \frac{P^3}{2\kappa^3} + \frac{3P^4}{8\kappa^4} + \mathcal O(\frac{P^5}{\kappa^5})$ \\
O72 model \cite{Ogden1972} & $\gamma^{-2}(\gamma \ln(J) + J^{-\gamma} -1)$ & $\frac{P}{\kappa} - \frac{P^2}{2\kappa^2} + \frac{(\gamma+3)P^3}{6\kappa^3} - \frac{(\gamma+2)(\gamma+4)P^4}{12\kappa^4} + \mathcal O(\frac{P^5}{\kappa^5})$ \\
\hline
\end{tabular}
\end{center}
\caption{A list of volumetric energies and the series representation of the corresponding $G_{\mathrm{vol}}$. Notice that both energies are scaled by $\kappa$ to make their representations dimensionless. }
\label{table:vol_energy_expansion}
\end{table}

\begin{figure} 
  \begin{center}
  \begin{tabular}{cc} 
\includegraphics[angle=0, trim=120 85 160 120, clip=true, scale=0.22]{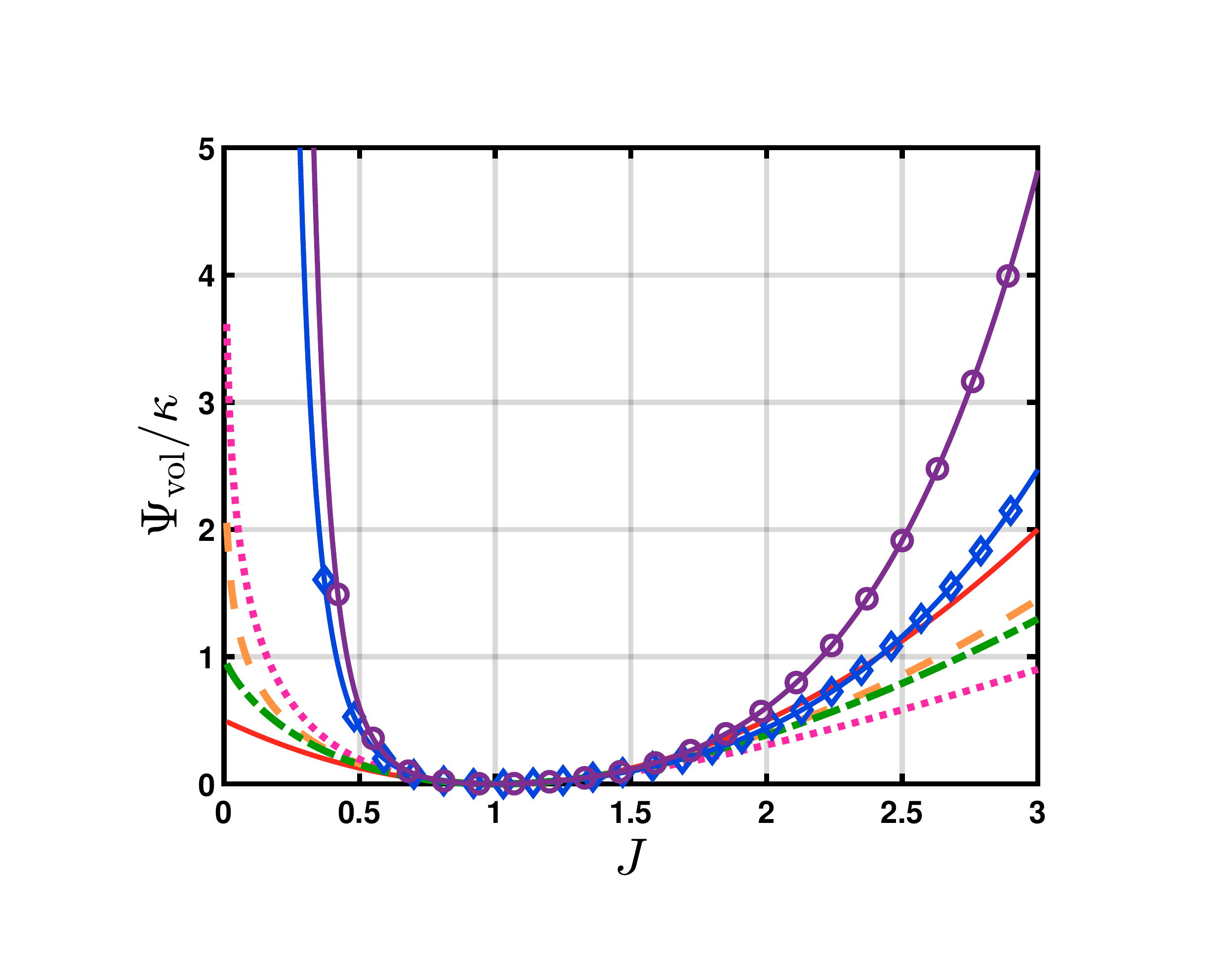} &
\includegraphics[angle=0, trim=90 85 150 120, clip=true, scale=0.22]{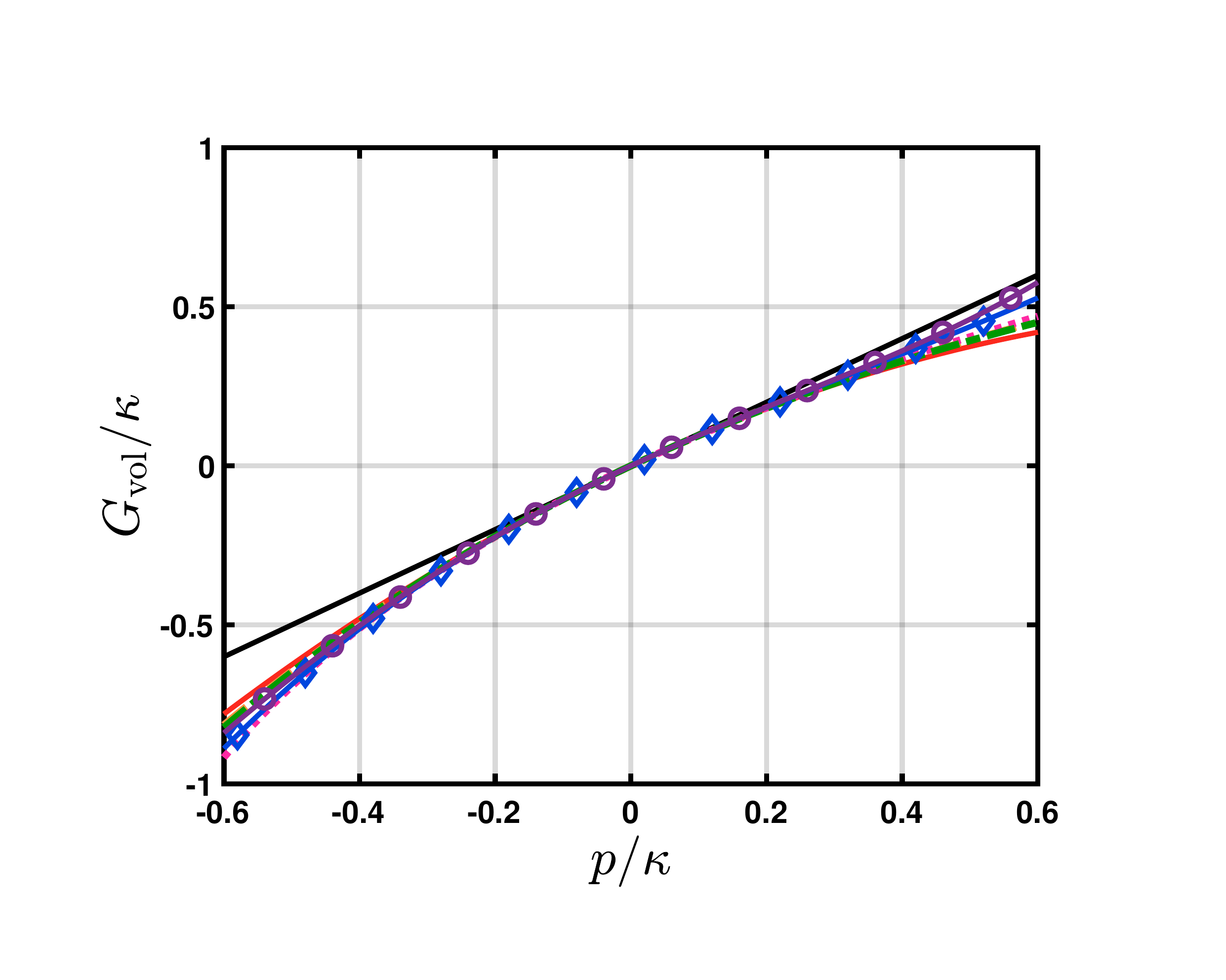} \\
\multicolumn{2}{c}{
\includegraphics[angle=0, trim=420 210 360 1165, clip=true, scale = 0.32]{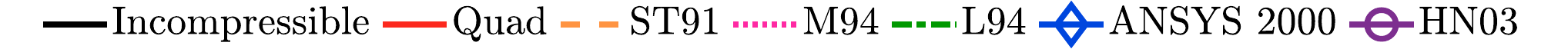} 
}
\end{tabular} 
\end{center}
\caption{Illustration of the volumetric energies $\Psi_{\mathrm{vol}}/\kappa$ (left) and $G_{\mathrm{vol}}/\kappa$ (right). Notice that the plots of $G_{\mathrm{vol}}/\kappa$ for the ANSYS 2000 model and HN03 model are made based on their series representation given in Table \ref{table:vol_energy_expansion}.}
\label{fig:illustration-vol-energy} 
\end{figure}

\section{An explicit expression of $\bm{\mathcal{K}}$}
\label{appendix:mathcal-K}
Following the derivations and formulas of $\bm{\mathcal{L}}$ documented in \cite{Liu2024}, we provide the explicit expressions for the rank-six tensor 
\begin{align*}
\bm{\mathcal K} := \frac12 \frac{\partial \mathbb Q^{-1}}{\partial \bm E} = \frac14 \frac{\partial^2 \bm C}{\partial \bm E \partial \bm E},
\end{align*}
as
\begin{align*}
\bm{\mathcal{K}} = \sum_{a=1}^3 g_{a} \bm{M}_a  \otimes \bm{M}_a  \otimes \bm{M}_a + \sum_{a = 1}^3 \sum_{b \neq a}^3 \zeta_{ab} \left( \mathbb{H}_{aba} + \mathbb{H}_{aab} + \mathbb{H}_{baa} \right) + \sum_{a = 1}^3 \sum_{b \neq a}^3 \sum_{\substack{ c \neq a \\  c \neq b} }^3 \varrho \mathbb{H}_{abc},
\end{align*}
in which
\begin{gather*}
g_{a} = \frac12 \left( \frac{1}{E^{'\:2}(\lambda_{a})} - \lambda_{a} \frac{E^{''}(\lambda_{a})}{E^{'\:3}(\lambda_{a})}\right), \displaybreak[2] \\
\zeta_{ab} = \frac18 \frac{\sigma_{ab}- c_{b}}{E(\lambda_{a})-E(\lambda_{b})}, \quad c_{a} = \frac{\lambda_a}{E'(\lambda_{a})}, \quad \sigma_{ab} = \frac{\lambda_a^2 - \lambda_b^2}{ 2(E(\lambda_{a}) - E(\lambda_{b})) }, \displaybreak[2] \\
\varrho = \frac{1}{32}\sum_{a = 1}^3 \sum_{b \neq a}^3 \sum_{\substack{ c \neq a \\  c \neq b} }^3 \frac{\lambda_a^2}{\left(E(\lambda_{a}) - E(\lambda_{b}) \right) \left(E(\lambda_{a}) - E(\lambda_{c}) \right)} \displaybreak[2] \\
\mathbb H_{abc} := \bm N_a \otimes \bm N_b \otimes \bm N_c \otimes \bm N_a \otimes \bm N_b \otimes \bm N_c + \bm N_a \otimes \bm N_b \otimes \bm N_c \otimes \bm N_a \otimes \bm N_c \otimes \bm N_b \displaybreak[2] \\
+ \bm N_a \otimes \bm N_b \otimes \bm N_a \otimes \bm N_c \otimes \bm N_b \otimes \bm N_c + \bm N_a \otimes \bm N_b \otimes \bm N_a \otimes \bm N_c \otimes \bm N_c \otimes \bm N_b \displaybreak[2] \\
 + \bm N_b \otimes \bm N_a \otimes \bm N_c \otimes \bm N_a \otimes \bm N_b \otimes \bm N_c + \bm N_b \otimes \bm N_a \otimes \bm N_c \otimes \bm N_a \otimes \bm N_c \otimes \bm N_b \displaybreak[2] \\
+ \bm N_b \otimes \bm N_a \otimes \bm N_a \otimes \bm N_c \otimes \bm N_b \otimes \bm N_c + \bm N_b \otimes \bm N_a \otimes \bm N_a \otimes \bm N_c \otimes \bm N_c \otimes \bm N_b.
\end{gather*}
When there exist identical stretches, $\zeta_{ab}$ and $\varrho$ are replaced by the following
\begin{align*}
\lim\limits_{\lambda_b \to \lambda_a} \zeta_{ab} = \frac18 g_a, \quad \lim\limits_{\substack{\lambda_b \to \lambda_a \\ \lambda_c \neq \lambda_a}} \varrho = \zeta_{ca}, \quad \mbox{and} \quad \lim\limits_{\substack{\lambda_b \to \lambda_a \\ \lambda_c \to \lambda_a}} \varrho = \frac{1}{8} g_{a}.
\end{align*}
The tensor $\bm{\mathcal{K}}$ is needed in the constitutive integration for the nonlinear models discussed in Section \ref{sec:micromechanical-model} and is in fact constructed based on the eigendecomposition of $\bm C^{\mathrm e}$. The superscript $\mathrm e$ is omitted here for notational simplicity in the formulas.

\bibliography{viscoelasticity_theory}

\begin{thebibliography}{95}
\providecommand{\natexlab}[1]{#1}
\providecommand{\url}[1]{\texttt{#1}}
\expandafter\ifx\csname urlstyle\endcsname\relax
  \providecommand{\doi}[1]{doi: #1}\else
  \providecommand{\doi}{doi: \begingroup \urlstyle{rm}\Url}\fi

\bibitem[Arruda and Boyce(1993)]{Arruda1993}
E.~Arruda and M.~Boyce.
\newblock A three-dimensional constitutive model for the large stretch behavior
  of rubber elastic materials.
\newblock \emph{Journal of the Mechanics and Physics of Solids}, 41\penalty0
  (2):\penalty0 389--412, 1993.

\bibitem[Bahreman et~al.(2022)Bahreman, Darijani, and Narooei]{Bahreman2022}
M.~Bahreman, H.~Darijani, and K.~Narooei.
\newblock Investigation of multiplicative decompositions in the form of ${F}_e
  {F}_v$ and ${F}_v {F}_e$ to extend viscoelasticity laws from small to finite
  deformations.
\newblock \emph{Mechanics of Materials}, 167:\penalty0 104235, 2022.

\bibitem[Ba\v{z}ant(1998)]{Bazant1998}
Z.~Ba\v{z}ant.
\newblock Easy-to-compute tensors with symmetric inverse approximating {H}encky
  finite strain and its rate.
\newblock \emph{Journal of Engineering Materials and Technology}, 120:\penalty0
  131--136, 1998.

\bibitem[Beex(2019)]{Beex2019}
L.~Beex.
\newblock Fusing the {S}eth-{H}ill strain tensors to fit compressible elastic
  material responses in the nonlinear regime.
\newblock \emph{International Journal of Mechanical Sciences}, 163:\penalty0
  105072, 2019.

\bibitem[Bergstr{\"o}m and Boyce(1998)]{Bergstroem1998}
J.~Bergstr{\"o}m and M.~Boyce.
\newblock Constitutive modeling of the large strain time-dependent behavior of
  elastomers.
\newblock \emph{Journal of the Mechanics and Physics of Solids}, 46\penalty0
  (5):\penalty0 931--954, 1998.

\bibitem[Bischoff et~al.(2001)Bischoff, Arruda, and Grosh]{Bischoff2001}
J.~Bischoff, E.~Arruda, and K.~Grosh.
\newblock A new constitutive model for the compressibility of elastomers at
  finite deformations.
\newblock \emph{Rubber chemistry and technology}, 74\penalty0 (4):\penalty0
  541--559, 2001.

\bibitem[Ciambella and Nardinocchi(2021)]{Ciambella2021}
J.~Ciambella and P.~Nardinocchi.
\newblock A structurally frame-indifferent model for anisotropic
  visco-hyperelastic materials.
\newblock \emph{Journal of the Mechanics and Physics of Solids}, 147:\penalty0
  104247, 2021.

\bibitem[Curnier and Rakotomanana(1991)]{Curnier1991}
A.~Curnier and L.~Rakotomanana.
\newblock Generalized strain and stress measures: critical survey and new
  results.
\newblock \emph{Engineering Transactions}, 39\penalty0 (3-4):\penalty0
  461--538, 1991.

\bibitem[Curnier and Zysset(2006)]{Curnier2006}
A.~Curnier and P.~Zysset.
\newblock A family of metric strains and conjugate stresses, prolonging usual
  material laws from small to large transformations.
\newblock \emph{International Journal of Solids and Structures}, 43\penalty0
  (10):\penalty0 3057--3086, 2006.

\bibitem[Dal et~al.(2020)Dal, G{\"u}ltekin, and A{\c{c}}{\i}kg{\"o}z]{Dal2020}
H.~Dal, O.~G{\"u}ltekin, and K.~A{\c{c}}{\i}kg{\"o}z.
\newblock An extended eight-chain model for hyperelastic and finite
  viscoelastic response of rubberlike materials: Theory, experiments and
  numerical aspects.
\newblock \emph{Journal of the Mechanics and Physics of Solids}, 145:\penalty0
  104159, 2020.

\bibitem[Dal et~al.(2021)Dal, A{\c{c}}{\i}kg{\"o}z, and Badienia]{Dal2021}
H.~Dal, K.~A{\c{c}}{\i}kg{\"o}z, and Y.~Badienia.
\newblock On the performance of isotropic hyperelastic constitutive models for
  rubber-like materials: {A} state of the art review.
\newblock \emph{Applied Mechanics Reviews}, 73\penalty0 (2):\penalty0 020802,
  2021.

\bibitem[Darijani and Naghdabadi(2010)]{Darijani2010}
H.~Darijani and R.~Naghdabadi.
\newblock Constitutive modeling of solids at finite deformation using a
  second-order stress-strain relation.
\newblock \emph{International Journal of Engineering Science}, 48:\penalty0
  223--236, 2010.

\bibitem[Darijani and Naghdabadi(2013)]{Darijani2013}
H.~Darijani and R.~Naghdabadi.
\newblock Kinematics and kinetics modeling of thermoelastic continua based on
  the multiplicative decomposition of the deformation gradient.
\newblock \emph{International Journal of Engineering Science}, 62:\penalty0
  56--69, 2013.

\bibitem[Darijani et~al.(2010)Darijani, Naghdabadi, and
  Kargarnovin]{Darijani2010b}
H.~Darijani, R.~Naghdabadi, and M.~Kargarnovin.
\newblock Hyperelastic materials modelling using a strain measure consistent
  with the strain energy postulates.
\newblock \emph{Proceedings of the Institution of Mechanical Engineers, Part C:
  Journal of Mechanical Engineering Science}, 224\penalty0 (3):\penalty0
  591--602, 2010.

\bibitem[Doyle and Ericksen(1956)]{Doyle1956}
T.~Doyle and J.~Ericksen.
\newblock Nonlinear elasticity.
\newblock \emph{Advances in Applied Mechanics}, 4:\penalty0 53--115, 1956.

\bibitem[Du et~al.(2020)Du, Zhang, Guo, Tang, and Guo]{Du2020}
Z.~Du, G.~Zhang, T.~Guo, S.~Tang, and X.~Guo.
\newblock Tension-compression asymmetry at finite strains: {A} theoretical
  model and exact solutions.
\newblock \emph{Journal of the Mechanics and Physics of Solids}, 143:\penalty0
  104084, 2020.

\bibitem[Ferry(1980)]{Ferry1980}
J.~Ferry.
\newblock \emph{Viscoelastic properties of polymers}.
\newblock John Wiley \& Sons, 1980.

\bibitem[Flaschel et~al.(2023)Flaschel, Kumar, and Lorenzis]{Flaschel2023}
M.~Flaschel, S.~Kumar, and L.~D. Lorenzis.
\newblock Automated discovery of generalized standard material models with
  {EUCLID}.
\newblock \emph{Computer Methods in Applied Mechanics and Engineering},
  405:\penalty0 115867, 2023.

\bibitem[Flory(1961)]{Flory1961}
P.~Flory.
\newblock Thermodynamic relations for high elastic materials.
\newblock \emph{Transactions of the Faraday Society}, 57:\penalty0 829--838,
  1961.

\bibitem[Gouhier and Diani(2024)]{Gouhier2024}
F.~Gouhier and J.~Diani.
\newblock A comparison of finite strain viscoelastic models based on the
  multiplicativedecomposition.
\newblock \emph{European Journal of Mechanics-A/Solids}, 108:\penalty0 105424,
  2024.

\bibitem[Govindjee and Simo(1992)]{Govindjee1992}
S.~Govindjee and J.~Simo.
\newblock Mullins' effect and the strain amplitude dependence of the storage
  modulus.
\newblock \emph{International Journal of Solids and Structures}, 29:\penalty0
  1737--1751, 1992.

\bibitem[Govindjee et~al.(2014)Govindjee, Potter, and Wilkening]{Govindjee2014}
S.~Govindjee, T.~Potter, and J.~Wilkening.
\newblock Dynamic stability of spinning viscoelastic cylinders at finite
  deformation.
\newblock \emph{International Journal of Solids and Structures}, 51:\penalty0
  3589--3603, 2014.

\bibitem[Green and Naghdi(1965)]{Green1965}
A.~Green and P.~Naghdi.
\newblock A general theory of an elastic-plastic continuum.
\newblock \emph{Archive for Rational Mechanics and Analysis}, 18:\penalty0
  251--281, 1965.

\bibitem[Green and Naghdi(1971)]{Green1971}
A.~Green and P.~Naghdi.
\newblock Some remarks on elastic-plastic deformation at finite strain.
\newblock \emph{International Journal of Engineering Science}, 9:\penalty0
  1219--1229, 1971.

\bibitem[Green and Tobolsky(1946)]{Green1946}
M.~Green and A.~Tobolsky.
\newblock A new approach to the theory of relaxing polymeric media.
\newblock \emph{The Journal of Chemical Physics}, 14\penalty0 (2):\penalty0
  80--92, 1946.

\bibitem[G{\"u}ltekin et~al.(2019)G{\"u}ltekin, Dal, and
  Holzapfel]{Gueltekin2019}
O.~G{\"u}ltekin, H.~Dal, and G.~Holzapfel.
\newblock On the quasi-incompressible finite element analysis of anisotropic
  hyperelastic materials.
\newblock \emph{Computational Mechanics}, 63:\penalty0 443--453, 2019.

\bibitem[Hackl(1997)]{Hackl1997}
K.~Hackl.
\newblock Generalized standard media and variational principles in classical
  and finite strain elastoplasticity.
\newblock \emph{Journal of the Mechanics and Physics of Solids}, 45\penalty0
  (5):\penalty0 667--688, 1997.

\bibitem[Hartmann and Neff(2003)]{Hartmann2003}
S.~Hartmann and P.~Neff.
\newblock Polyconvexity of generalized polynomial-type hyperelastic strain
  energy functions for near-incompressibility.
\newblock \emph{International Journal of Solids and Structures}, 40:\penalty0
  2767--2791, 2003.

\bibitem[Helfenstein et~al.(2010)Helfenstein, Jabareen, Mazza, and
  Govindjee]{Helfenstein2010}
J.~Helfenstein, M.~Jabareen, E.~Mazza, and S.~Govindjee.
\newblock On non-physical response in models for fiber-reinforced hyperelastic
  materials.
\newblock \emph{International Journal of Solids and Structures}, 47:\penalty0
  2056--2061, 2010.

\bibitem[Hencky(1928)]{Hencky1928}
H.~Hencky.
\newblock \"{U}ber die form des elastizit\"{a}tsgesetzes bei ideal elastischen
  stoffen.
\newblock \emph{Zeitschrift fur Technische Physik}, 9:\penalty0 215--220, 1928.

\bibitem[Hill(1968)]{Hill1968}
R.~Hill.
\newblock On constitutive inequalities for simple materials-{I}.
\newblock \emph{Journal of the Mechanics and Physics of Solids}, 16\penalty0
  (4):\penalty0 229--242, 1968.

\bibitem[Hill(1979)]{Hill1979}
R.~Hill.
\newblock Aspects of invariance in solid mechanics.
\newblock \emph{Advances in Applied Mechanics}, 18:\penalty0 1--75, 1979.

\bibitem[Holzapfel(2000)]{Holzapfel2000}
G.~Holzapfel.
\newblock \emph{Nonlinear {S}olid {M}echanics: {A} {C}ontinuum {A}pproach for
  {E}ngineering}.
\newblock John Wiley \& Sons, 2000.

\bibitem[Holzapfel and Gasser(2001)]{Holzapfel2001}
G.~Holzapfel and T.~Gasser.
\newblock A viscoelastic model for fiber-reinforced composites at finite
  strains: {C}ontinuum basis, computational aspects and applications.
\newblock \emph{Computer Methods in Applied Mechanics and Engineering},
  190:\penalty0 4379--4403, 2001.

\bibitem[Holzapfel and Simo(1996)]{Holzapfel1996}
G.~Holzapfel and J.~Simo.
\newblock A new viscoelastic constitutive model for continuous media at finite
  thermomechanical changes.
\newblock \emph{International Journal of Solids and Structures}, 33:\penalty0
  3019--3034, 1996.

\bibitem[Hong(2011)]{Hong2011}
W.~Hong.
\newblock Modeling viscoelastic dielectrics.
\newblock \emph{Journal of the Mechanics and Physics of Solids}, 59:\penalty0
  637--650, 2011.

\bibitem[Hossain et~al.(2012)Hossain, Vu, and Steinmann]{Hossain2012}
M.~Hossain, D.~Vu, and P.~Steinmann.
\newblock Experimental study and numerical modelling of vhb 4910 polymer.
\newblock \emph{Computational Materials Science}, 59:\penalty0 65--74, 2012.

\bibitem[Houlsby and Puzrin(2000)]{Houlsby2000}
G.~Houlsby and A.~Puzrin.
\newblock A thermomechanical framework for constitutive models for
  rate-independent dissipative materials.
\newblock \emph{International Journal of Plasticity}, 16\penalty0 (9):\penalty0
  1017--1047, 2000.

\bibitem[Huber and Tsakmakis(2000)]{Huber2000}
N.~Huber and C.~Tsakmakis.
\newblock Finite deformation viscoelasticity laws.
\newblock \emph{Mechanics of Materials}, 32:\penalty0 1--18, 2000.

\bibitem[Itskov(2004)]{Itskov2004}
M.~Itskov.
\newblock On the application of the additive decomposition of generalized
  strain measures in large strain plasticity.
\newblock \emph{Mechanics Research Communications}, 31:\penalty0 507--517,
  2004.

\bibitem[Jansen et~al.(2000)Jansen, Whiting, and Hulbert]{Jansen2000}
K.~Jansen, C.~Whiting, and G.~Hulbert.
\newblock A generalized-$\alpha$ method for integrating the filtered
  {N}avier-{S}tokes equations with a stabilized finite element method.
\newblock \emph{Computer Methods in Applied Mechanics and Engineering},
  190:\penalty0 305--319, 2000.

\bibitem[Korobeynikov et~al.(2022)Korobeynikov, Larichkin, and
  Rotanova]{Korobeynikov2022}
S.~Korobeynikov, A.~Larichkin, and T.~Rotanova.
\newblock Hyperelasticity models extending hooke’s law from small to moderate
  strains and experimental verification of their scope of application.
\newblock \emph{International Journal of Solids and Structures}, 252:\penalty0
  111815, 2022.

\bibitem[Kroeger(2015)]{Kroeger2015}
M.~Kroeger.
\newblock Simple, admissible, and accurate approximants of the inverse langevin
  and brillouin functions, relevant for strong polymer deformations and flows.
\newblock \emph{Journal of Non-Newtonian Fluid Mechanics}, 223:\penalty0
  77--87, 2015.

\bibitem[Kr{\"o}ner(1959)]{Kroener1959}
E.~Kr{\"o}ner.
\newblock Allgemeine kontinuumstheorie der versetzungen und eigenspannungen.
\newblock \emph{Archive for Rational Mechanics and Analysis}, 4:\penalty0
  273--334, 1959.

\bibitem[Kumar and Lopez-Pamies(2016)]{Kumar2016}
A.~Kumar and O.~Lopez-Pamies.
\newblock On the two-potential constitutive modeling of rubber viscoelastic
  materials.
\newblock \emph{Comptes Rendus Mecanique}, 344\penalty0 (2):\penalty0 102--112,
  2016.

\bibitem[Laiarinandrasana et~al.(2003)Laiarinandrasana, Piques, and
  Robisson]{Laiarinandrasana2003}
L.~Laiarinandrasana, R.~Piques, and A.~Robisson.
\newblock Visco-hyperelastic model with internal state variable coupled with
  discontinuous damage concept under total lagrangian formulation.
\newblock \emph{International Journal of Plasticity}, 19\penalty0 (7):\penalty0
  977--1000, 2003.

\bibitem[Latorre and Mont{\'a}ns(2016)]{Latorre2016}
M.~Latorre and F.~Mont{\'a}ns.
\newblock Fully anisotropic finite strain viscoelasticity based on a reverse
  multiplicative decomposition and logarithmic strains.
\newblock \emph{Computers \& Structures}, 163:\penalty0 56--70, 2016.

\bibitem[{Le Tallec} et~al.(1993){Le Tallec}, Rahier, and Kaiss]{Tallec1993}
P.~{Le Tallec}, C.~Rahier, and A.~Kaiss.
\newblock Three-dimensional incompressible viscoelasticity in large strains:
  formulation and numerical approximation.
\newblock \emph{Computer Methods in Applied Mechanics and Engineering},
  109:\penalty0 233--258, 1993.

\bibitem[Lee(1969)]{Lee1969}
E.~Lee.
\newblock Elastic-plastic deformation at finite strains.
\newblock \emph{Journal of Applied Mechanics}, 36:\penalty0 1--6, 1969.

\bibitem[Li et~al.(2024)Li, Zhang, Guan, Liu, and Yuan]{Li2024}
X.~Li, D.~Zhang, J.~Guan, J.~Liu, and H.~Yuan.
\newblock Magneto-viscoelastic rod model for hard-magnetic soft rods under 3{D}
  large deformation: {T}heory and numerical implementation.
\newblock \emph{International Journal of Solids and Structures}, 305:\penalty0
  113101, 2024.

\bibitem[Lion(1997)]{Lion1997}
A.~Lion.
\newblock A physically based method to represent the thermo-mechanical
  behaviour of elastomers.
\newblock \emph{Acta Mechanica}, 123:\penalty0 1--25, 1997.

\bibitem[Liu et~al.(1994)Liu, Hofstetter, and Mang]{Liu1994}
C.~Liu, G.~Hofstetter, and H.~Mang.
\newblock {3D} finite element analysis of rubber-like materials at finite
  strains.
\newblock \emph{Engineering Computations}, 11:\penalty0 111--128, 1994.

\bibitem[Liu and Marsden(2018)]{Liu2018}
J.~Liu and A.~Marsden.
\newblock A unified continuum and variational multiscale formulation for
  fluids, solids, and fluid-structure interaction.
\newblock \emph{Computer Methods in Applied Mechanics and Engineering},
  337:\penalty0 549--597, 2018.

\bibitem[Liu et~al.(2019)Liu, Marsden, and Tao]{Liu2019a}
J.~Liu, A.~Marsden, and Z.~Tao.
\newblock An energy-stable mixed formulation for isogeometric analysis of
  incompressible hyperelastodynamics.
\newblock \emph{International Journal for Numerical Methods in Engineering},
  120:\penalty0 937--963, 2019.

\bibitem[Liu et~al.(2021)Liu, Latorre, and Marsden]{Liu2021b}
J.~Liu, M.~Latorre, and A.~Marsden.
\newblock A continuum and computational framework for viscoelastodynamics:
  \rom{1}. {F}inite deformation linear models.
\newblock \emph{Computer Methods in Applied Mechanics and Engineering},
  385:\penalty0 114059, 2021.

\bibitem[Liu et~al.(2024)Liu, Guan, Zhao, and Luo]{Liu2024}
J.~Liu, J.~Guan, C.~Zhao, and J.~Luo.
\newblock A continuum and computational framework for viscoelastodynamics:
  \rom{3}. {A} nonlinear theory.
\newblock \emph{Computer Methods in Applied Mechanics and Engineering},
  430:\penalty0 117248, 2024.

\bibitem[Lubliner(1985)]{Lubliner1985}
J.~Lubliner.
\newblock A model of rubber viscoelasticity.
\newblock \emph{Mechanics Research Communications}, 12:\penalty0 93--99, 1985.

\bibitem[Mao et~al.(2017)Mao, Lin, Zhao, and Anand]{Mao2017}
Y.~Mao, S.~Lin, X.~Zhao, and L.~Anand.
\newblock A large deformation viscoelastic model for double-network hydrogels.
\newblock \emph{Journal of the Mechanics and Physics of Solids}, 100:\penalty0
  103--130, 2017.

\bibitem[Martin et~al.(2018)Martin, M\"{u}nch, Eidel, and Neff]{Martin2018}
R.~Martin, I.~M\"{u}nch, B.~Eidel, and P.~Neff.
\newblock A brief history of logarithmic strain measures in nonlinear
  elasticity.
\newblock \emph{Proceedings in Applied Mathematics and Mechanics}, 18\penalty0
  (1):\penalty0 e201800366, 2018.

\bibitem[Martyushev and Seleznev(2006)]{Martyushev2006}
L.~Martyushev and V.~Seleznev.
\newblock Maximum entropy production principle in physics, chemistry and
  biology.
\newblock \emph{Physics Reports}, 426\penalty0 (1):\penalty0 1--45, 2006.

\bibitem[Meunier et~al.(2008)Meunier, Chagnon, Favier, Org{\'e}as, and
  Vacher]{Meunier2008}
L.~Meunier, G.~Chagnon, D.~Favier, L.~Org{\'e}as, and P.~Vacher.
\newblock Mechanical experimental characterisation and numerical modelling of
  an unfilled silicone rubber.
\newblock \emph{Polymer testing}, 27\penalty0 (6):\penalty0 765--777, 2008.

\bibitem[Miehe(1994)]{Miehe1994}
C.~Miehe.
\newblock Aspects of the formulation and finite element implementation of large
  strain isotropic elasticity.
\newblock \emph{International Journal for Numerical Methods in Engineering},
  37:\penalty0 1981--2004, 1994.

\bibitem[Miehe(1998{\natexlab{a}})]{Miehe1998}
C.~Miehe.
\newblock A constitutive frame of elastoplasticity at large strains based on
  the notion of a plastic metric.
\newblock \emph{International Journal of Solids and Structures}, 35:\penalty0
  3859--3897, 1998{\natexlab{a}}.

\bibitem[Miehe(1998{\natexlab{b}})]{Miehe1998a}
C.~Miehe.
\newblock A formulation of finite elastoplasticity based on dual co-and
  contra-variant eigenvector triads normalized with respect to a plastic
  metric.
\newblock \emph{Computer Methods in Applied Mechanics and Engineering},
  159:\penalty0 223--260, 1998{\natexlab{b}}.

\bibitem[Miehe and Keck(2000)]{Miehe2000}
C.~Miehe and J.~Keck.
\newblock Superimposed finite elastic-viscoelastic-plastoelastic stress
  response with damage in filled rubbery polymers. {E}xperiments, modelling and
  algorithmic implementation.
\newblock \emph{Journal of the Mechanics and Physics of Solids}, 48:\penalty0
  323--365, 2000.

\bibitem[Miehe and Lambrecht(2001)]{Miehe2001b}
C.~Miehe and M.~Lambrecht.
\newblock Algorithms for computation of stresses and elasticity moduli in terms
  of seth-hill's family of generalized strain tensors.
\newblock \emph{Communications in Numerical Methods in Engineering},
  17\penalty0 (5):\penalty0 337--353, 2001.

\bibitem[Miehe et~al.(2002)Miehe, Apel, and Lambrecht]{Miehe2002}
C.~Miehe, N.~Apel, and M.~Lambrecht.
\newblock Anisotropic additive plasticity in the logarithmic strain space:
  modular kinematic formulation and implementation based on incremental
  minimization principles for standard materials.
\newblock \emph{Computer Methods in Applied Mechanics and Engineering},
  191:\penalty0 5383--5425, 2002.

\bibitem[Moerman et~al.(2016)Moerman, Simms, and Nagel]{Moerman2016}
K.~Moerman, C.~Simms, and T.~Nagel.
\newblock Control of tension--compression asymmetry in {O}gden hyperelasticity
  with application to soft tissue modelling.
\newblock \emph{Journal of the Mechanical Behavior of Biomedical Materials},
  56:\penalty0 218--228, 2016.

\bibitem[Murakami(2012)]{Murakami2012}
S.~Murakami.
\newblock \emph{Continuum damage mechanics: a continuum mechanics approach to
  the analysis of damage and fracture}.
\newblock Springer Science \& Business Media, 2012.

\bibitem[Naghdi(1990)]{Naghdi1990}
P.~Naghdi.
\newblock A critical review of the state of finite plasticity.
\newblock \emph{Zeitschrift f{\"u}r angewandte Mathematik und Physik ZAMP},
  41\penalty0 (3):\penalty0 315--394, 1990.

\bibitem[Nolan et~al.(2014)Nolan, Gower, Destrade, Ogden, and
  McGarry]{Nolan2014}
D.~Nolan, A.~Gower, M.~Destrade, R.~Ogden, and J.~McGarry.
\newblock A robust anisotropic hyperelastic formulation for the modelling of
  soft tissue.
\newblock \emph{Journal of the Mechanical Behavior of Biomedical Materials},
  39:\penalty0 48--60, 2014.

\bibitem[Ogden(1972)]{Ogden1972}
R.~Ogden.
\newblock Large deformation isotropic elasticity-on the correlation of theory
  and experiment for incompressible rubberlike solids.
\newblock \emph{Proceedings of the Royal Society of London A: Mathematical,
  Physical and Engineering Sciences}, 326:\penalty0 565--584, 1972.

\bibitem[Ogden(1997)]{Ogden1997}
R.~Ogden.
\newblock \emph{Non-linear Elastic Deformations}.
\newblock Dover Publications, 1997.

\bibitem[Ogden et~al.(2004)Ogden, Saccomandi, and Sgura]{Ogden2004}
R.~Ogden, G.~Saccomandi, and I.~Sgura.
\newblock Fitting hyperelastic models to experimental data.
\newblock \emph{Computational Mechanics}, 34:\penalty0 484--502, 2004.

\bibitem[Onsager(1931)]{Onsager1931}
L.~Onsager.
\newblock Reciprocal relations in irreversible processes. i.
\newblock \emph{Physical Review}, 37\penalty0 (4):\penalty0 405--416, 1931.

\bibitem[Papadopoulos and Lu(1998)]{Papadopoulos1998}
P.~Papadopoulos and J.~Lu.
\newblock A general framework for the numerical solution of problems in finite
  elasto-plasticity.
\newblock \emph{Computer Methods in Applied Mechanics and Engineering},
  159:\penalty0 1--18, 1998.

\bibitem[Peng and Landel(1975)]{Peng1975}
S.~Peng and R.~Landel.
\newblock Stored energy function and compressibility of compressible rubberlike
  materials under large strain.
\newblock \emph{Journal of Applied Physics}, 46:\penalty0 2599--2604, 1975.

\bibitem[Reese and Govindjee(1998)]{Reese1998}
S.~Reese and S.~Govindjee.
\newblock A theory of finite viscoelasticity and numerical aspects.
\newblock \emph{International Journal of Solids and Structures}, 35:\penalty0
  3455--3482, 1998.

\bibitem[Sadik and Yavari(2024)]{Sadik2024}
S.~Sadik and A.~Yavari.
\newblock Nonlinear anisotropic viscoelasticity.
\newblock \emph{Journal of the Mechanics and Physics of Solids}, 182:\penalty0
  105461, 2024.

\bibitem[Sansour(2008)]{Sansour2008}
C.~Sansour.
\newblock On the physical assumptions underlying the volumetric-isochoric split
  and the case of anisotropy.
\newblock \emph{European Journal of Mechanics-A/Solids}, 27:\penalty0 28--39,
  2008.

\bibitem[Schr{\"o}der et~al.(2002)Schr{\"o}der, Gruttmann, and
  L{\"o}blein]{Schroeder2002}
J.~Schr{\"o}der, F.~Gruttmann, and J.~L{\"o}blein.
\newblock A simple orthotropic finite elasto-plasticity model based on
  generalized stress-strain measures.
\newblock \emph{Computational Mechanics}, 30:\penalty0 48--64, 2002.

\bibitem[Seth(1964)]{Seth1964}
B.~Seth.
\newblock \emph{Second-Order Effects in Elasticity, Plasticity and Fluid
  Dynamics}, chapter Generalized strain measure with application to physical
  problems, pages 162--172.
\newblock Pergamon Press, Oxford, 1964.

\bibitem[Sidoroff(1974)]{Sidoroff1974}
F.~Sidoroff.
\newblock Un mod{\`e}le visco{\'e}lastique non lin{\'e}aire avec configuration
  interm{\'e}diaire.
\newblock \emph{J. M{\'e}canique}, 13:\penalty0 679--713, 1974.

\bibitem[Simo(1987)]{Simo1987}
J.~Simo.
\newblock On a fully three-dimensional finite-strain viscoelastic damage model:
  formulation and computational aspects.
\newblock \emph{Computer Methods in Applied Mechanics and Engineering},
  60:\penalty0 153--173, 1987.

\bibitem[Simo(1992)]{Simo1992e}
J.~Simo.
\newblock Algorithms for static and dynamic multiplicative plasticity that
  preserve the classical return mapping schemes of the infinitesimal theory.
\newblock \emph{Computer Methods in Applied Mechanics and Engineering},
  99\penalty0 (1):\penalty0 61--112, 1992.

\bibitem[Simo and Hughes(2006)]{Simo2006}
J.~Simo and T.~Hughes.
\newblock \emph{Computational Inelasticity}.
\newblock Springer Science \& Business Media, 2006.

\bibitem[Simo and Taylor(1991)]{Simo1991}
J.~Simo and R.~Taylor.
\newblock Quasi-incompressible finite elasticity in principal stretches.
  continuum basis and numerical algorithms.
\newblock \emph{Computer Methods in Applied Mechanics and Engineering},
  85:\penalty0 273--310, 1991.

\bibitem[Stewart and Anand(2023)]{Stewart2023}
E.~Stewart and L.~Anand.
\newblock Magneto-viscoelasticity of hard-magnetic soft-elastomers:
  {A}pplication to modeling the dynamic snap-through behavior of a bistable
  arch.
\newblock \emph{Journal of the Mechanics and Physics of Solids}, 179:\penalty0
  105366, 2023.

\bibitem[Terzano et~al.(2023)Terzano, Wollner, Kainz, Rolf-Pissarczyk,
  G{\"o}tzen, and Holzapfel]{Terzano2023}
M.~Terzano, M.~Wollner, M.~Kainz, M.~Rolf-Pissarczyk, N.~G{\"o}tzen, and
  G.~Holzapfel.
\newblock Modelling the anisotropic inelastic response of polymeric scaffolds
  for in situ tissue engineering applications.
\newblock \emph{Journal of the Royal Society Interface}, 20\penalty0
  (206):\penalty0 20230318, 2023.

\bibitem[Vernerey et~al.(2017)Vernerey, Long, and Brighenti]{Vernerey2017}
F.~Vernerey, R.~Long, and R.~Brighenti.
\newblock A statistically-based continuum theory for polymers withtransient
  networks.
\newblock \emph{Journal of the Mechanics and Physics of Solids}, 107:\penalty0
  1--20, 2017.

\bibitem[Wang and Chester(2018)]{Wang2018}
S.~Wang and S.~Chester.
\newblock Experimental characterization and continuum modeling of inelasticity
  in filled rubber-like materials.
\newblock \emph{International Journal of Solids and Structures},
  136-137:\penalty0 125--136, 2018.

\bibitem[Wollner et~al.(2023)Wollner, Terzano, Rolf-Pissarczyk, and
  Holzapfel]{Wollner2023}
M.~Wollner, W.~Terzano, M.~Rolf-Pissarczyk, and G.~Holzapfel.
\newblock A general model for anisotropic pseudo-elasticity and viscoelasticity
  at finite strains.
\newblock \emph{Journal of the Mechanics and Physics of Solids}, 180:\penalty0
  105403, 2023.

\bibitem[Zhan et~al.(2023)Zhan, Wang, Qu, Steinmann, and Xiao]{Zhan2023}
L.~Zhan, S.~Wang, S.~Qu, P.~Steinmann, and R.~Xiao.
\newblock A general continuum damage model for soft composites.
\newblock \emph{Journal of the Mechanics and Physics of Solids}, 175:\penalty0
  105290, 2023.

\bibitem[Ziegler(1983)]{Ziegler1983}
H.~Ziegler.
\newblock \emph{An introduction to thermomechanics}.
\newblock Elsevier, 1983.

\bibitem[Ziegler and Wehrli(1987)]{Ziegler1987}
H.~Ziegler and C.~Wehrli.
\newblock The derivation of constitutive relations from the free energy and the
  dissipation function.
\newblock \emph{Advances in Applied Mechanics}, 25:\penalty0 183--238, 1987.

\end{thebibliography}

\end{document}